\documentclass[12pt]{article}
\usepackage{amssymb,amsmath}
\input{amssym.def}\input{amssym}
\usepackage{graphicx, tikz}
 \usepackage{hyperref}
\usepackage{changebar}

\def\subjclass#1{\par\medskip
\noindent\textbf{Mathematics Subject Classification (2010):} #1}
\def\keywords#1{\par\medskip
\noindent\textbf{Keywords.} #1}

\newcommand{\E}{{\mathbb E}}

\newcommand{\N}{{\mathbb N}}
\newcommand{\Q}{{\mathbb Q}}
\newcommand{\R}{{\mathbb R}}
\newcommand{\Z}{{\mathbb Z}}
\renewcommand{\S}{{\mathbb S}}
\renewcommand{\P}{{\mathbb P}}

\newcommand\cA{{\mathcal A}}
\newcommand\cB{{\mathcal B}}
\newcommand\cC{{\mathcal C}}

\newcommand\cE{{\mathcal E}}

\newcommand\cG{{\mathcal G}}

\newcommand\cI{{\mathcal I}}
\newcommand\cK{{\mathcal K}}
\newcommand\cL{{\mathcal L}}
\newcommand\cO{{\mathcal O}}
\newcommand\cR{{\mathcal R}}
\newcommand\cM{{\mathcal M}}

\newcommand\cT{{\mathcal T}}
\newcommand\cW{{\mathcal W}}

\newcommand\cS{{\mathcal S}}
\newcommand\cZ{{\mathcal Z}}
\newcommand\cXi{{\Psi}}

\newcommand\bH{{\mathbb H}}

\newcommand\bP{{\mathbb P}}

\newcommand\bR{{\mathbb R}}

\newcommand\Id{{\bf 1}}

\newcommand\hxi{\hat{\xi}}
\newcommand\hm{\hat{m}}
\newcommand\shxi{\sum_{\hat{\xi}}}
\newcommand\sxi{\sum_{\xi}}
\newcommand\wCs{\bar{C}_{\scat}}
\newcommand\wgs{\bar{\gamma}_{\scat}}
\newcommand\wts{\bar{\vartheta}_{\scat}}

\newtheorem{theorem}{Theorem}[section]
\newtheorem{thm}[theorem]{Theorem}
\newtheorem{prop}[theorem]{Proposition}
\newtheorem{corollary}[theorem]{Corollary}
\newtheorem{lemma}[theorem]{Lemma}
\newtheorem{remark}[theorem]{Remark}

\newenvironment{proof}{\noindent {\bf Proof.}}{ \hfill $\Box$\\ }
\newenvironment{proofof}[1]{\noindent {\bf Proof of #1.}}{ \hfill $\Box$\\ }

\newtheorem{mthm}{Theorem}

\newcommand\eps{\varepsilon}

\def\vs{\varsigma}
\def\scat{\rho}
\def\ks{\kappa_\scat}
\def\hks{\hat\kappa_\scat}

\def\ga{\gamma}
\def\gb{\hat\gamma}

\def\supp{\mathop{\mbox{supp}}}

\def\red{\color{red}}

\begin{document}

\title{Periodic Lorentz gas with small scatterers}

\author{P\'eter B\'alint\thanks{
MTA-BME Stochastics Research Group,
Budapest University of Technology and Economics,
M{\H u}egyetem rkp. 3., H-1111 Budapest, Hungary and
 Department of Stochastics, Institute of Mathematics,
Budapest University of Technology and Economics,
M{\H u}egyetem rkp. 3., H-1111 Budapest, Hungary,
\textit{pet@math.bme.hu}} , Henk Bruin\thanks{Faculty of Mathematics, University of Vienna,
Oskar Morgensternplatz 1, 1090 Vienna, Austria, {\it henk.bruin@univie.ac.at}} and Dalia Terhesiu
\thanks{Institute of Mathematics, University of Leiden,
	Niels Bohrweg 1, 2333 CA Leiden, The Netherlands,
		{\it daliaterhesiu@gmail.com}}}

\date{\today}

\maketitle
\abstract{We prove
 limit laws for infinite horizon planar periodic Lorentz gases when, as time $n$ tends to infinity, the scatterer size $\scat$ may also tend to zero simultaneously at a sufficiently slow pace. In particular we obtain a non-standard Central Limit Theorem as well as a Local Limit Theorem for the displacement function. To the best of our knowledge, these are the first results on an intermediate case between the two well-studied regimes with superdiffusive $\sqrt{n\log n}$ scaling (i) for fixed infinite horizon configurations -- letting first $n\to \infty$ and then $\scat\to 0$ -- studied e.g.~by Sz\'asz \& Varj\'u (2007) and (ii) Boltzmann-Grad type situations -- letting first $\scat\to 0$ and then $n\to \infty$ -- studied by Marklof \& T\'oth (2016).
}

\subjclass{Primary: 37D50, Secondary 37A60, 60F05, 60F17, 82C05, 82C40}
\keywords{Lorentz gas, small scatterers, billiards, limit theorems, Nagaev-Guivarc'h method}

\section{Introduction}\label{sec:intro}

In this paper we are interested in limit laws for infinite horizon planar periodic Lorentz gases with small scatterers.
The Lorentz gas, a popular model of mathematical physics introduced by H. Lorentz in 1905 \cite{L}, is a dynamical system on the infinite billiard table obtained by removing strictly convex scatterers from $\R^2$. We study the periodic model when these scatterers are round disks of radius $\scat\in (0,1/2)$ positioned at the points of the Euclidean lattice $\Z^2$. This table can be split up into countably many compact cells, each congruent to the unit square, which can be also regarded as the $2$-dimensional flat torus. As usual, a point particle on the table
moves with a unit velocity vector  along straight lines inside the table, and collides elastically -- angle of incidence equals angle of reflection -- at the scatterers. This billiard flow produces a billiard map for the Poincar\'e section of outgoing collisions. The phase space of the billiard map in a single cell is
$\cM = \partial O \times [-\frac{\pi}{2}, \frac{\pi}{2}]$,
where $O$ is a round disk at the origin with radius $\scat$.
The phase space representing all cells together is $\widehat\cM =  \cM \times \Z^2$ and the displacement function
$\ks:\cM \to \Z^2$ indicates the difference in cell
numbers going from one collision to the next.
As $O$ is strictly convex, the billiard is dispersing, and the dynamics has good hyperbolicity properties.

For any $\scat\in (0,1/2)$, the horizon is infinite which means that the time between two consecutive collisions -- and accordingly, $\ks:\cM\to\Z^2$ -- is unbounded. This corresponds to \textit{corridors}, that is, infinite strips on $\R^2$ parallel to some direction $\xi\in\Z^2\setminus\{0\}$ which do not intersect any of the scatterers of the infinite billiard table. Both the number and the geometry of these corridors depend on $\scat$. Hence, the value of the parameter $\scat$ strongly affects the asymptotic properties of the dynamics, and in particular, plays a central role in our exposition.

\subsection{Recalling limit laws for fixed $\scat\in (0,1/2)$ as time $n \to \infty$}
\label{subsec:recfix}

A consequence of the infinite horizon is the superdiffusive behaviour of $\ks$ with
$\scat\in (0,1/2)$
fixed, captured in the first place in the Central Limit Theorem  (CLT) with non-standard normalization.
To recall this result along with its refinements, we introduce some notation to be used throughout this paper.
Let $T_{\scat}:\cM \to  \cM$ be the billiard map, recall that it preserves the canonical invariant probability  measure $\mu$.
Set
\begin{equation}
\label{eq:nottr}
 \kappa_{n,\scat}=\sum_{j=0}^{n-1}\ks\circ T_\scat^j,
 \qquad  \Sigma=\frac{1}{\pi} \begin{pmatrix} 1&0\\0&1 \end{pmatrix}.
\end{equation}
Choosing the initial point on $\cM$ according to $\mu$, we can regard $\kappa_{n,\scat}$ as a family of random variables.
Throughout we let $\implies$ stand for convergence in distribution.
We recall the CLT with non-standard normalization: for every $\scat\in(0,\frac12)$  there exists a positive definite matrix $\Sigma_{\scat}$ such that:
\begin{align}
\label{eq:Cltfix}
 \text{For fixed }\scat\in (0,1/2),\quad \frac{\kappa_{n,\scat}}{\sqrt{n\log n}}\implies \mathcal{N}(0,\Sigma_{\scat}) \text{ as }
  n\to\infty.
\end{align}
This result was conjectured by Bleher~\cite{Bl92} and proved rigorously via two different methods: Sz\'asz \& Varj\'u in~\cite{SV07}, and Chernov \& Dolgopyat in~\cite{ChDo09}. In the setting above,
the requirement of having two non-parallel corridors (present in~\cite{SV07, ChDo09}) is automatically satisfied because the scatterers are positioned at the lattice points.

It is important to note that there is an explicit  formula for $\Sigma_{\scat}$, which involves the scatterer geometry for fixed $\scat$, see for example \cite[Formula (2.1)]{ChDo09}. A computation (similar to our proof of Lemma~\ref{lem:corridor_sum}) shows that
\begin{equation}
\label{eq:LimSigmaScat}
\lim_{\scat\to 0} (4\pi \scat^2) \Sigma_{\scat}=\Sigma,
\end{equation}
where $\Sigma$ is the diagonal matrix given in~\eqref{eq:nottr}.
To compare these results with our Theorem~\ref{thm:meta} below, we point out the following direct consequence of~\eqref{eq:Cltfix} and~\eqref{eq:LimSigmaScat}:
\begin{align}
\label{eq:CltSV}
  \frac{\kappa_{n,\scat}}{(\sqrt{4\pi}\scat)^{-1}\sqrt{n\log n}}
  \implies \mathcal{N}(0,\Sigma) \text{ as first }
  n\to\infty \text{ and then } \scat\to 0.
\end{align}

The method of proof in~\cite{SV07} relies on the existence of a Young tower for $T_\scat$
as in~\cite{Young98, Chernov99} and an abstract result of B\'alint \& Gou\"ezel~\cite{BalintGouezel06}
along with several additional properties of $(\ks, T_\scat)$ established in~\cite{SV07}.
One notable feature of this method is that it provides a refinement
of the CLT~\eqref{eq:Cltfix}, namely the Local Limit Theorem (LLT):
\begin{align}
\label{eq:lltfix}
\text{For fixed }\scat\in (0,1/2),\quad (n\log n)\mu(\kappa_{n,\scat}=0)\to\Phi_{\Sigma_{\scat}}(0)\text{ as } n\to\infty,
\end{align}
where $\Phi_{\Sigma_{\scat}}$ is the density of the Gaussian random variable on the r.h.s~of~\eqref{eq:Cltfix}.

The method of proof in~\cite{ChDo09} exploits exponential mixing for the sequence $\{\ks\circ T_\scat^n\}_{n\ge 1}$.
The authors of that work develop an argument based on standard pairs to establish a bound on the correlations for $\ks$:
\begin{equation}\label{eq:corrfix}
\begin{array}{l}
\text{For fixed }\scat\in (0,1/2),\text{ there exist }\hat{\vartheta}_{\scat}\in (0,1)
\text{ and }\hat{C}_\scat > 0\\
\text{so that }\left|\int_{\cM} \ks\cdot\ks\circ T_\scat^n\,d\mu\right|\le \hat{C}_\scat\cdot\hat{\vartheta}_{\scat}^n \text{ for all }
n\ge 1.\\
\end{array}
\end{equation}
Using~\eqref{eq:corrfix}, the CLT~\eqref{eq:Cltfix} is proved in~\cite[Proof of Theorem 8 a)]{ChDo09} via blocking type arguments; we refer to Denker~\cite{Den86} for a classical reference. Furthermore, as shown in~\cite[Proof of Theorem 8]{ChDo09},
 the limit law~\eqref{eq:Cltfix} together with a tightness argument for a truncated version of $\ks$
provides another refinement of the CLT, namely, the Weak Invariance Principle (WIP):
\begin{equation}\label{eq:wipfix}
\begin{array}{l}
\text{For fixed }\scat\in (0,1/2)
\text{ and } s\in (0,1), \ \frac{\kappa_{\lfloor ns \rfloor ,\scat}+ \{ns\}(\kappa_{\lfloor ns \rfloor +1,\scat}- \kappa_{\lfloor ns \rfloor,\scat}) }{\sqrt{n\log n}}
\text{ converges as }\\
n\to\infty \text{ to a Brownian motion with mean } 0\text{ and variance }\Sigma_{\scat}.
\end{array}
\end{equation}
Similar versions of the CLT~\eqref{eq:Cltfix} and the WIP~\eqref{eq:wipfix}
hold for the flight time function taking values in $\R^2$, see~\cite{ChDo09}.

In a different direction, a further important consequence of the LLT~\eqref{eq:lltfix} established in~\cite{SV07}
is that it allows one to study mixing of the infinite measure preserving
billiard dynamics on the entire lattice $\widehat\cM =  \cM \times \Z^2$. This can be modelled by a $\Z^2$ extension
$$
\hat T_{\scat}^n(\theta,\phi,\ell) = ( T_{\scat}^n(\theta,\phi),
\ell+\kappa_{n,\scat}(\theta,\phi)),
\qquad  (\theta,\phi) \in  \cM,\ \ell\in \Z^2.
$$
The dynamics $\hat T_\scat$ preserves the measure $\hat\mu=\mu\times \mbox{Leb}_{\Z^2}$, where $\mbox{Leb}_{\Z^2}$ denotes the counting measure. Let us introduce the notation $\cM_{\xi}:=\cM\times\{\xi\} \subset \widehat\cM$ for $\xi\in\Z^2$, and furthermore, for brevity, let
$\cM_0:=\cM\times \{(0,0)\}$, where the label $0$ refers here to the origin in $\Z^2$. An immediate consequence of~\eqref{eq:lltfix} is:
\begin{align}
\label{eq:mixfix}
\text{For fixed }\scat\in (0,1/2),\ (n\log n) \mu(\kappa_{n,\scat}=0)
=(n\log n) \hat\mu(\cM_0\cap\hat T_\scat^{-n}\cM_0)
\to\Phi_{\Sigma_{\scat}}(0) \text{ as }
n\to\infty.
\end{align}

A first refinement of the LLT~\eqref{eq:lltfix}
and of the mixing statement~\eqref{eq:mixfix} was obtained by P\`ene~\cite{Pene18}
who proved the analogue of these statements for the class of dynamically H\"older observables.
Later on, P\`ene \& Terhesiu~\cite{PT20}, building on the results in~\cite{BalintGouezel06}, obtained
sharp error rates in the LLT and the mixing for dynamically H\"older observables, including observables supported
on compact sets. Furthermore,~\cite{PT20} establish optimal error rates for mean zero observables.

\subsection{Recalling results as first $\scat\to 0$ and then $n \to \infty$ (Boltzmann-Grad limit)}
\label{subsec:recnonj}

In \cite{MS,MS2}, Marklof \& Str\"ombergsson studied the Boltzmann-Grad limit of the periodic Lorentz gas. This corresponds to letting the scatterer size $\scat \to 0$  and investigating the displacement in the rescaled continuous time $T=\scat t$ (so that the mean free path remains bounded). In particular, \cite{MS} proves that, in this Boltzmann-Grad limit, the displacement of the particle converges, on any finite time interval, to an explicitly given Markov process. Marklof \& T\'oth~\cite{MT16} then studied the large time asymptotic of this Markov process, and obtained the CLT and the WIP with  non-standard normalization $\sqrt{T\log T}$.

These results on the Boltzmann-Grad limit scenario hold in any dimension,
not just in $d=1, 2$ as the results mentioned in the previous subsection. For more details, we refer to the original references. What is most relevant for us is that \cite[Theorem 1.1]{MT16} and \cite[Theorem 1.3]{MT16} are reduced to discrete time statements that can be formulated in terms of the behavior of $\kappa_{n,\scat}$ in the limits $\scat\to 0$ first and then $n\to\infty$. In particular, \cite[Theorem 1.2]{MT16} states for $d=2$ that:
\begin{equation}\label{eq:recnonjoint}
  \frac{\kappa_{n,\scat}}{(\sqrt{4\pi}\scat)^{-1}\sqrt{n\log n}}
  \implies \mathcal{N}(0,\Sigma) \qquad \text{ as } \scat\to 0\text{ followed by } n\to\infty,
\end{equation}
where $\kappa_{n,\scat}$ and $\Sigma$ are as in~\eqref{eq:nottr}
\footnote{Actually, \cite[Theorem 1.2]{MT16} is stated for the flight time function taking values in $\R^2$,
as opposed to the displacement function taking values in $\Z^2$, but these are equivalent as the difference between the two processes is uniformly bounded, see Remark~\ref{rmk:flightev}.}, while
\cite[Theorem 1.4]{MT16} is the corresponding  WIP which, when $d=2$, reads as~\eqref{eq:wipfix} with the main difference of the limit paths:
$\scat\to 0$ followed by $ n\to\infty$, as opposed to fixed $\scat$.

In~\cite{MT16}, the authors state that
\emph{an  open problem is to consider the joint limit $\scat \to 0$ and $n\to\infty$}.
In the Boltzmann--Grad limit scenario with \emph{diffusive behaviour},  this type of question is answered  by Lutsko \& T\'oth in~\cite{LTo20} for \emph{random Lorentz gases}, in dimension $d=3$,
where, on top of the initial condition, additional randomness comes from the \emph{random placement of the scatterers}. However, their model is very different from the  model considered in~\cite{MT16}
and it is characterized by diffusion (Brownian motion with \emph{standard normalization}).

\subsection{Main results as $\scat\to 0$ and $n \to \infty$ in the  joint limit}
\label{subsec:ours}

Our main result takes a step in answering the open question in~\cite{MT16} for the planar periodic Lorentz gas. It reads as follows.
\begin{mthm}\label{thm:meta}
Let $\kappa_{n,\scat}$ and $\Sigma$ be as in~\eqref{eq:nottr}, and let
\[
b_{n,\scat}=\frac{\sqrt{n\log (n/\scat^2)}}{\sqrt{4\pi}\ \scat}.
\]
There exists a function $M(\scat)$ with $M(\scat)\to \infty$ as $\scat \to 0$ such that,
\begin{align*}
 \frac{\kappa_{n,\scat}}{b_{n,\scat}}\implies \mathcal{N}(0,\Sigma), \text{ as }
  n\to\infty \text { and } \scat\to 0 \text{ such that } M(\scat)=o(\log n).
\end{align*}
\end{mthm}
A precise expression of $M(\scat)$ is given in Theorem~\ref{th-mainth} in Section~\ref{sec-lmthm}.
At this stage we mention that $M(\scat)$  depends on the rate of correlation decay for H\"older observables as $\scat\to 0$.
How this decay rate depends on $\scat$ is not known and we do not attempt to study this in the present paper. However, we comment on some relevant aspects of correlation decay below.

In the remainder of this section, we make some further comments on how our results compare to various other works, and on some key ingredients of our argument.

\textbf{Comments on the rate of correlation decay}. Statistical limit laws in dynamical settings in general, and our results in particular, are strongly related to effective bounds on time correlations.
For several decades, it has been a major problem to prove exponential decay of correlations for H\"older observables in dispersing billiards, that is, bounds of the form:
\begin{equation}
\label{eq:edcHolder}
\left|\int_{\cM} \psi_1\cdot\psi_2\circ T_\scat^n \,d\mu\right|\le C_\scat(\psi_1,\psi_2)\cdot\hat{\theta}_{\scat}^n \text{ for all }
n\ge 0,
\end{equation}
where $\psi_1:\cM\to\R$ and $\psi_2:\cM\to\R$ are centered, H\"older continuous observables, and $\hat{\theta}_{\scat}<1$ may depend on the H\"older exponent, while $C_\scat(\psi_1,\psi_2)>0$ on the H\"older norm of these functions, and both constants depend also on $\scat$ (i.e.~on the billiard table). Several powerful methods have been designed to prove bounds of the form
\eqref{eq:edcHolder}, in particular using quasi-compactness of the transfer operator on Young towers \cite{Young98} or anisotropic Banach spaces \cite{DZ11}, coupling of standard pairs
\cite[Chapter 7]{CM} or most recently, Birkhoff cones \cite{DemersLiverani21}. However, each of these methods involve some non-constructive compactness argument which is the reason why there
is no explicit information available on how the rate of decay (i.e.~$C_\scat$ and $\hat{\theta}_{\scat}$) depends on $\scat$. For instance, in the framework of quasi-compact transfer operators,
this corresponds to having effective bounds on the essential spectral radius, but not on the spectral gap.

In fact, depending on the method, $\psi_1$ and $\psi_2$ may belong to a larger space (that contains H\"older functions), however, these spaces do not contain the unbounded observable
$\kappa_{\scat}$. Hence, even for fixed $\scat$, it requires additional effort to obtain correlation bounds for unbounded observables, in particular, to derive \eqref{eq:corrfix}.
As mentioned above, in our context of the infinite horizon Lorentz gas, \eqref{eq:corrfix} was proved by Chernov and Dolgopyat in \cite[Proposition 9.1]{ChDo09},
which is an important reference for our work. Let us also mention
\cite[Lemma 3.2]{BCD11} on a similar bound for the induced return time arising in dispersing billiards with cusps, and the more recent paper
\cite{WZZ} where correlation bounds for unbounded observables are studied in an axiomatic framework that includes further billiard models. Nonetheless, all these works consider the large
time asymptotics of a fixed billiard system. To treat the simultaneous scaling of $\scat\to 0$ and $n\to\infty$, in Appendix~\ref{sec:decay} of the present paper we extend \cite[Proposition 9.1]{ChDo09} in two directions. On the one hand,
on top of the mere existence of some $\hat{C}_\scat > 0$ and $\hat{\vartheta}_{\scat}<1$ in \eqref{eq:corrfix}, we explicitly relate these constants to $C_\scat$ and $\hat{\theta}_{\scat}$
 of \eqref{eq:edcHolder}, as expressed in \eqref{eq:edc_const_relate}.\footnote{We will also use the notations $\gamma_{\scat}=1-\hat{\theta}_{\scat}$ and
 $\hat{\gamma}_{\scat}=1-\hat{\vartheta}_{\scat}$.} On the other hand, to exploit correlation bounds of the type \eqref{eq:corrfix} when taking the joint limit,
 these have to be combined with the action of the perturbed transfer operator $R_{\scat}(t)$ (introduced below) as stated in our Proposition~\ref{prop:decaykappa}.

\textbf{Comparison with results on the random Lorentz gas}. To compare our Theorem~\ref{thm:meta} with the results of Lutsko \& T\'oth on the random Lorentz gas, it is important to emphasize that although both \cite{LTo20} and our paper  consider a joint limit of scatterer radius tending to $0$ and time tending to infinity simultaneously, the settings of these two papers are quite different. In particular, the starting point of Lutsko \& T\'oth is the Boltzmann Grad limit of the random Lorentz gas, and accordingly, \cite{LTo20} can handle situations when time tends to infinity at a \textit{sufficiently slow} pace in relation to the scatterer size tending to $0$. In contrast, the starting point of our work is the superdiffusive limit in the infinite horizon periodic Lorentz gas with fixed scatterer size (see subsection~\ref{subsec:recfix} for a summary of previous results), and accordingly we can handle situations when  time tends to infinity at a \textit{sufficiently fast} pace in relation to the scatterer size tending to $0$.

It is also important to note that under the condition $M(\scat)=o(\log n)$  we have
\[
\frac{b_{n,\scat}}{(\sqrt{4\pi}\scat)^{-1}\sqrt{n\log n}}\to 1,
\]
which shows that our Theorem~\ref{thm:meta} is indeed a direct analogue of both \eqref{eq:CltSV} and \eqref{eq:recnonjoint}.  To simplify the exposition, we omit the case $d=1$ (i.e., the Lorentz tube), but believe that similar results can be obtained by the same arguments.

\textbf{Further comments on some corollaries of our result and some elements of our proofs.}
A main advantage of the current method of proof via spectral methods
is that it allows us to obtain (with no additional effort) the LLT~\eqref{eq:lltfix} and the mixing statement~\eqref{eq:mixfix} with appropriate limit paths $\scat\to 0$ simultaneously with $n\to\infty$, as opposed to fixed $\scat$.
For the LLT we refer to Theorem~\ref{th-llt} and for the mixing result we refer to Corollary~\ref{cor:mix}.

We mention up front that unlike in the \emph{fixed $\scat$} scenario with main results recalled in Subsection~\ref{subsec:recfix}, we cannot exploit the existence of a Young tower because
it seems undoable to build such a tower in a fashion
that it depends continuously on $\scat$.
Instead, we prove Theorem~\ref{thm:meta} via the Nagaev method on Banach spaces of distributions introduced by Demers \& Zhang~\cite{DZ11, DZ13, DZ14} in the spirit
of the spaces constructed in Demers \& Liverani~\cite{DemersLiverani08}.
See Aaronson \& Denker~\cite{AD01,ADb}
for a classical reference
on the Nagaev method in (Gibbs Markov) dynamics beyond the CLT with standard normalization (that is $\sqrt n$). However, as we shall explain below,
the standard pairs argument in~\cite{ChDo09} plays a crucial role in our proof.

We end this introduction summarizing the main steps and challenges of our proofs.
A main difficulty comes from the fact that as $\scat\to 0$, more and more corridors open up and controlling their number and geometry is a non-trivial task.
Another challenge for the proofs of Theorem~\ref{thm:meta} and the LLT in Theorem~\ref{th-llt}
comes from the fact that the spaces in~\cite{DZ11, DZ13, DZ14} cannot be used in a straightforward way \emph{even} in the infinite horizon case with \emph{fixed} $\scat$.

The Nagaev method requires: 1) the existence of a Banach space $(\cB, \|\,\|_{\cB})$ on which the transfer
operator $R_\scat$ of $T_\scat$ has a spectral gap; 2) that the
 perturbed transfer operator ($R_\scat(t) \psi=R(e^{it\ks}\cdot \psi )$ for $\psi\in\cB$) has sufficiently good
continuity estimates $\|R_\scat(t)-R_\scat(0)\|_{\cB}\leq C |t|^\nu$; the larger $\nu > 0$, the better.

Regarding 1), using a Lasota-Yorke inequality on a strong space $\cB$ and a weak space $\cB_w$, Demers \& Zhang~\cite{DZ11, DZ13, DZ14} established the spectral gap for every fixed $\scat$, see Section~\ref{subsec-Bsp}.
This is the main reason why we resorted to the use of such Banach spaces.

Regarding 2), as in Keller \& Liverani~\cite{KellerLiverani99}, one could work with the weak space.
 For infinite horizon billiards, continuity estimates in the strong or weak Banach spaces in~\cite{DZ11, DZ13, DZ14} have not been obtained previously.
In Section~\ref{sec:periphspgap}, we give continuity estimates in such Banach spaces (strong or weak); the estimates there rely heavily on
a version of the growth lemma, namely Proposition~\ref{prop:growthscaled}.
These continuity estimates are $O(|t|^\nu)$ for $\nu<1/2$ with explicit dependence on $\scat$, in both the strong and the weak spaces. This exponent $\nu$ is too small to obtain the asymptotics of the leading eigenvalue of $R_\scat(t)$ directly.
Therefore we resort to a decomposition of the eigenvalue in several pieces (see the proof of Proposition~\ref{prop-eigv})  and exploit the standard pairs arguments in~\cite{ChDo09}  to deal with some  parts of the estimate, see Appendix~\ref{sec:decay}.
Along the way, we give a new proof of the LLT~\eqref{eq:lltfix} for fixed $\scat$
which is new at an abstract level as well,
namely by working on the Banach spaces~\cite{DZ11, DZ13, DZ14} in the absence of good continuity estimates but in the presence of exponential decay of correlations.
\\[4mm]
The paper is organised as follows: In Section~\ref{sec:Lorentz_model}
we recall some basic properties of hyperbolic billiards and also
estimate widths of corridors that open up as $\scat \to 0$.
Section~\ref{sec:growth} gives the Growth Lemmas, following~\cite{DZ11, DZ13, DZ14} but with estimates made explicit in terms of $\scat$,
and including sums over unbounded number of corridors (this is the reason why the estimates are worse than for the usual Growth Lemmas).
Section~\ref{subsec-Bsp} introduces the Banach spaces and recalls the proof of
the spectral gap property for the unperturbed transfer operator $R_\scat$, showing that the $\scat$-dependence can be controlled. Section~\ref{sec-opRk} is devoted to the continuity estimates of the perturbed transfer operator $R_\scat(t)$ and Section~\ref{sec:ev} gives the asymptotics of the corresponding leading eigenvalue. The precise statements and proofs of the limit theorems are gathered in Section~\ref{sec-lmthm}.

The appendices give further technical details on corridor sums (Appendix~\ref{sec:corridors}), distortion (Appendix~\ref{sec:distortion}) and decay of correlations by a combination of tower and standard pair arguments (Appendix~\ref{sec:decay}).
\\[3mm]
{\bf Acknowledgements:}
We thank the anonymous referees for useful comments.
HB was supported by the FWF Project P31950-N35.
PB was supported by National Research, Development and Innovation Office -– NKFIH, Projects K123782, K142169 and KKP144059.
PB and HB acknowledge the support of Stiftung A\"OU Project 103\"ou6.
DT was partially supported by EPSRC grant EP/S019286/1.

We would like to thank M. Weber for his helpful correspondence on the results established in~\cite{Weber}.
\\[3mm]

\section{Preliminaries on Lorentz gas on $\Z^2$}
\label{sec:Lorentz_model}

 Our general reference on hyperbolic billiards is Chernov \& Markarian~\cite{CM}, the conventions of which are followed in our exposition, except for some minor differences.
In particular, we use coordinates $(\theta,\phi) \in \S^1 \times [-\frac{\pi}{2}, \frac{\pi}{2}]$ on $\cM$, where
\begin{itemize}
\item  $\theta \in \S^1$ in clockwise orientation describes the collision point on the scatterer (so the corresponding point on $\partial O$ is $(\scat \sin \theta, \scat \cos \theta)$);
\item $\phi \in [-\frac{\pi}{2}, \frac{\pi}{2}]$ denotes the outgoing angle that the billiard trajectory makes
after a collision at a point with coordinate $\theta$ with the outward normal vector $\vec N_\theta$ at this point
(so $\phi = \frac{\pi}{2}$ corresponds to an outgoing trajectory tangent to $O$ in the positive $\theta$-direction).
\end{itemize}
In these coordinates $(\theta,\phi)$, the measure $\mu$ has the same form
$d\mu = \frac{1}{4\pi} \cos \phi \, d\phi\, d\theta$
for all values of the radius $\scat > 0$. Integrals
involving the displacement function $\ks$, however, do depend on $\scat$.
If the flight between
$(x,\ell)$ and $( T_\scat(x), \ell + \ks(x))$ goes through a corridor for a long time before hitting a scatterer at the boundary of this corridor, then  the angle at which the second
scatterer is hit is close to $\pm \frac{\pi}{2}$. This sparks another long
flight in the same corridor, i.e.,
$\| \ks(T_\scat x)\|$ is large, too.

In the remainder of this section, we record some properties of $T_\scat$ and  $\ks$. In Subsection~\ref{sec:corr} the geometry of corridors is described, with special emphasis on the asymptotics of small $\scat$. In Subsection~\ref{sec:singu} we focus on the singularities,  which, in addition to strong hyperbolicity, are the other main feature of the map $T_{\scat}:\cM\to\cM$.
In Subsection~\ref{sec-Hyperb}, the hyperbolic properties of $T_{\scat}:\cM\to\cM$ are discussed. Some lemmas of technical character are moved to Appendices~\ref{sec:corridors} and~\ref{sec:distortion}.
\\[3mm]
{\bf Notation:}
For functions (or sequences) $f$ and $g$,
we use the Vinogradov notation $f \ll g$ and the Landau big $O$ notation interchangeably: there is a constant $C > 0$ such that $f \leq C g$.
Similarly $f \asymp g$ means that there exists $C>1$ such that $C^{-1} g \le f \le Cg$.

\subsection{Corridors and their widths \label{sec:corr}}
Let $O_\ell$ denote the circular scatterer of radius $\scat$ placed at lattice point $\ell \in \Z^2$.
The computation of
$\mu(x \in \partial O_0 \times [-\frac{\pi}{2}, \frac{\pi}{2}] : \ks(x) =
(p,q))$
is based on the division of the phase space in corridors.
These are infinite strips in rational directions given by $\xi \in \Z^2 \setminus \{ 0 \}$
for $\scat$ sufficiently small,
that are disjoint from all scatterers (but maximal with respect to this property),
and they are periodically repeated under integer translations. As soon as $\scat < \frac12$, there are infinite corridors parallel to the coordinate axes.
If $\scat < \frac14 \sqrt{2}$, then corridors at angles of $\pm 45^\circ$ open up,
and the smaller $\scat$ becomes, the more corridors open up at rational angles.

Given $0 \neq \xi \in \Z^2$ and $\scat > 0$ sufficiently small,
there are two corridors simultaneous tangent to $O_0$ and $O_\xi$, one
corridor
on either side of the arc connecting $0$ and $\xi$.
The widths of the corridors
are denoted by $d_{\scat}(\xi)$ and $\tilde d_{\scat}(\xi)$, see Figure~\ref{fig:corridors}.

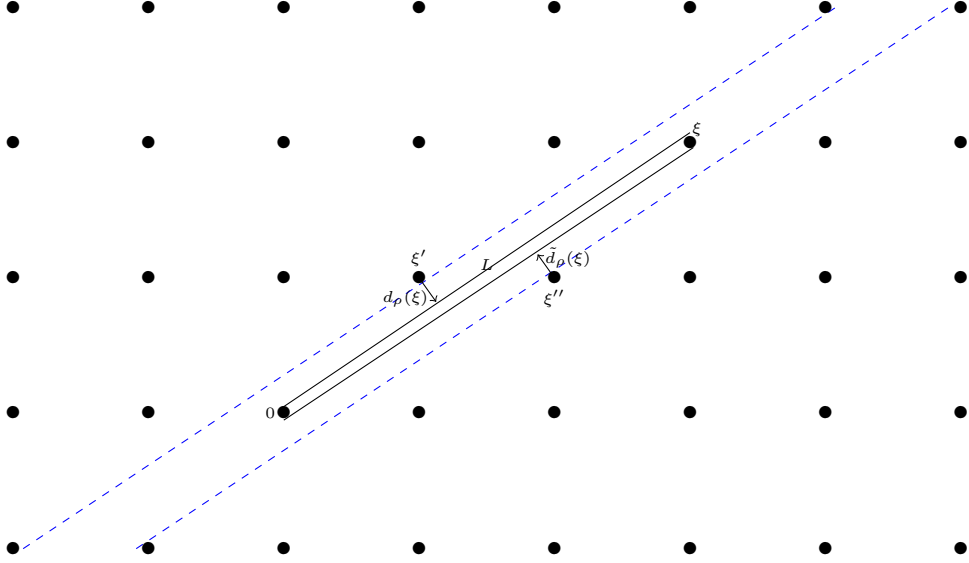
\begin{figure}[ht]
\begin{center}
\begin{tikzpicture}[scale=0.9]
\draw[-, draw=black] (4,-1.9) -- (10,2.15);
\draw[-, draw=black] (4,-2.1) -- (10.05,1.93);  \node at (7,0.2) {\tiny $L$};
\draw[-, draw=blue, dashed] (1.82,-4) -- (13.82,4);
\draw[-, draw=blue, dashed] (0.15,-4) -- (12.15,4);
\draw[->, draw=black] (8,0) -- (7.75,0.35);  \node at (8.2,0.3) {\tiny $\tilde
d_{\scat}(\xi)$};
\draw[->, draw=black] (6,0) -- (6.25,-0.35);  \node at (5.8,-0.3) {\tiny
$d_{\scat}(\xi)$};
\node at (6,0.3) {\tiny $\xi'$}; \node at (8,-0.3) {\tiny $\xi''$};
  \node at (3.8,-2) {\tiny $0$};  \node at (10.1,2.2)  {\tiny $\xi$};
 \node at (0,0) {$\bullet$}; \node at (2,0) {$\bullet$};
 \node at (4,0) {$\bullet$}; \node at (6,0) {$\bullet$};
  \node at (8,0) {$\bullet$}; \node at (10,0) {$\bullet$};
  \node at (12,0) {$\bullet$}; \node at (14,0) {$\bullet$};
 \node at (0,2) {$\bullet$}; \node at (2,2) {$\bullet$};
 \node at (4,2) {$\bullet$}; \node at (6,2) {$\bullet$};
 \node at (8,2) {$\bullet$}; \node at (10,2) {$\bullet$};
  \node at (12,2) {$\bullet$}; \node at (14,2) {$\bullet$};
 \node at (0,4) {$\bullet$}; \node at (2,4) {$\bullet$};
 \node at (4,4) {$\bullet$}; \node at (6,4) {$\bullet$};
 \node at (12,4) {$\bullet$}; \node at (14,4) {$\bullet$};
  \node at (8,4) {$\bullet$}; \node at (10,4) {$\bullet$};
  \node at (0,-4) {$\bullet$}; \node at (2,-4) {$\bullet$};
  \node at (4,-4) {$\bullet$}; \node at (6,-4) {$\bullet$};
   \node at (8,-4) {$\bullet$}; \node at (10,-4) {$\bullet$};
  \node at (12,-4) {$\bullet$}; \node at (14,-4) {$\bullet$};
  \node at (0,-2) {$\bullet$}; \node at (2,-2) {$\bullet$};
  \node at (4,-2) {$\bullet$}; \node at (6,-2) {$\bullet$};
   \node at (8,-2) {$\bullet$}; \node at (10,-2) {$\bullet$};
   \node at (12,-2) {$\bullet$}; \node at (14,-2) {$\bullet$};
\end{tikzpicture}
\caption{Corridors tangent to the scatterers at $0$ and $\xi = (3,2)$}
\label{fig:corridors}
\end{center}
\end{figure}

\begin{lemma}\label{lem:width}
 If $\scat = 0$ and $\xi = (p,q) \in \Z^2$ is expressed in lowest terms, then
$$
 d_0(\xi) = \tilde d_0(\xi) = \frac{1}{|\xi|}.
$$
For $\scat > 0$, the actual width of the corridor is then $d_{\scat}(\xi) =
\tilde d_{\scat}(\xi) = \max\{0,  |\xi|^{-1} - 2\scat\}$.
\end{lemma}

\begin{remark}\label{rem:width}
 Let us call these two corridors in the direction $\xi$ the $\xi$-corridors.
They open up only when $\scat < d_0(\xi)/2 = \tilde d_0(\xi)/2$.
For $\scat = 0$, the common boundary (called $\xi$-boundary) of the two
$\xi$-corridors is the line through $0$ and $\xi$.
The other boundaries are lines parallel to the $\xi$-boundary, going through
lattice points
that are called $\xi'$ and $\xi''$ in the below proof.
For $\xi = (p,q)$ (with $\gcd(p,q) = 1$), these points $\xi' = (p',q')$,
$\xi'' = (p'',q'')$
are uniquely determined by $\xi$ in the sense that $p'/q'$ and $p''/q''$ are
convergents preceding $p/q$
in the continued fraction expansion of $p/q$. In particular $|\xi'| , |\xi''|
\leq |\xi|$.
In the sequel, we usually only need one of these two $\xi$-corridors, and we
take the one with $\xi'$
in its other boundary.
\end{remark}

\begin{proof}
 If $(p,q) = (0, \pm 1)$ or $(\pm 1, 0)$, then clearly $d_0(\xi) = \tilde d_0(\xi) = 1$,
 so we can assume without loss of generality that $p \geq q > 0$.
 Let $L$ be the arc connecting $(0,0)$ to $(p,q)$.
 The corridors associated to $\xi$ intersect $[0,p] \times [0,q]$ in
 diagonal strips on either side of $L$.

 Let $\frac{q}{p} = [0;a_1, \dots, a_n=a]$ be the standard continued fraction
 expansion with $a \geq 1$,
 and the previous two convergents are denoted by
 $q'/p'$ and $q''/p''$, say $q''/p'' < q/p <q'/p'$ (the other inequality goes analogously).
 Therefore $q'p-qp' = 1$ and $q''p'-q'p'' = -1$.
 Also
 $$
 \frac{(a-1)q' + q''}{(a-1)p'+p''} < \frac{q}{p} < \frac{q'}{p'}
 $$
 are the best rational approximations of $q/p$, belonging to lattice points
 $\xi'$
 above $L$ and $\xi''$ below $L$.
 The vertical distance between $\xi'$ and the arc $L$
is $|q'-p'\frac{q}{p}| = \frac1p|q'p-p'q| = \frac1p$.
The vertical distance between $L$ and $\xi''$ is
\begin{eqnarray*}
 ((a-1)p'+p'')\frac{q}{p} - ((a-1)q'+q'') &=& \frac1p ( (a-1)(qp'-q'p) +
 qp''-q''p )\\
 &=& \frac1p ( 1-a + ( aq'+q'')p'' - (ap'+p'')q'' ) \\
 &=& \frac1p ( 1-a + a(q'p'' - q''p')) = \frac1p.
\end{eqnarray*}
The corridor's diameter is perpendicular to $\xi$, so $d_0(\xi)$ is computed
from this vertical distance
as the inner product of the vector $(0,1/p)^T$ and the vector $\xi = (p,q)^T$
rotated over $90^\circ$:
$$
\frac{1}{\sqrt{p^2+q^2}} \left\langle \begin{pmatrix} 0 \\ 1/p \end{pmatrix}
\ , \
 \begin{pmatrix} -q \\ p \end{pmatrix} \right\rangle =
 \frac{1}{\sqrt{p^2+q^2}} = \frac{1}{|\xi|}.
$$
The computation for $\tilde d_0(\xi) = |\xi|^{-1}$ is the same.
\end{proof}

\subsection{Singularities of the billiard map}\label{sec:singu}
In the coordinates $(\theta,\phi,\ell) \in \S^1 \times [-\frac\pi2,
\frac\pi2] \times \Z^2$ (or $\times \Z$ if it is a Lorentz tube),
the size of the scatterers $\scat$ doesn't appear, but it comes back in the
formula of the billiard map $T_{\scat}$
and in its hyperbolicity. Also the curvature of the scatterers is $\cK \equiv
1/\scat$.
We recall some notation from the Chernov \& Markarian book \cite{CM} (going back to the work of Sina\u{\i}),
bearing in mind that we have to redo several of their estimates to track the precise dependence on $\scat$.
The phase space is
$\widehat\cM =  \cM \times \Z^2 = \bigcup_{\ell \in \Z^2} \cM_\ell$, where each
$\cM_\ell$ is a copy of the cylinder
$\S^1 \times [-\frac\pi2, \frac\pi2]$, see Figure~\ref{fig:phasespace}.

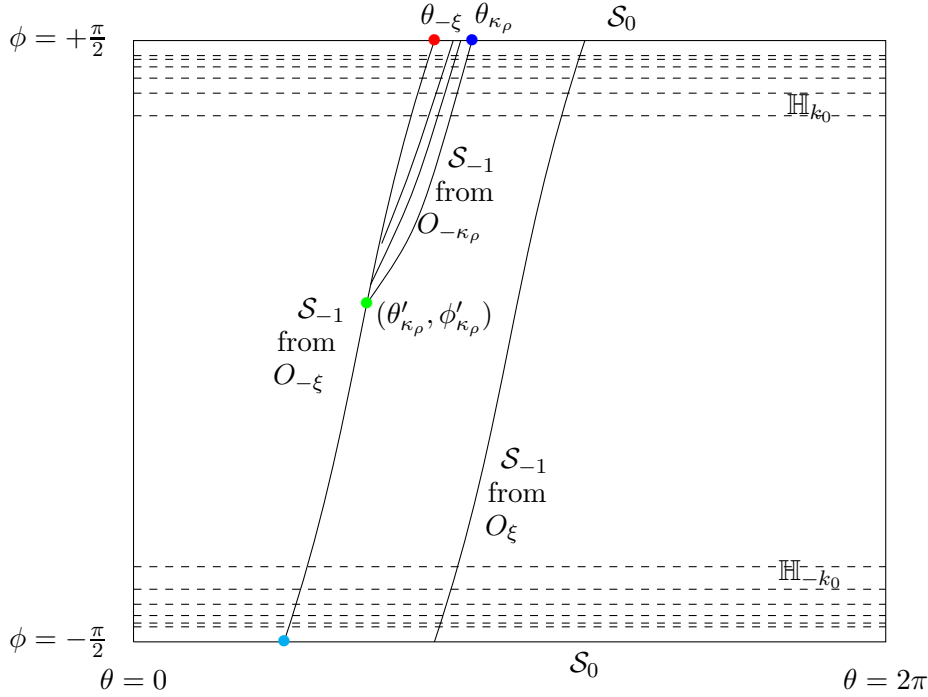
\begin{figure}[ht]
\begin{center}
\begin{tikzpicture}[scale=1]
\draw[-] (0,0) -- (10,0) -- (10,8) -- (0,8) -- (0,0);
\node at (-1,0) {\small $\phi = -\frac\pi2$}; \node at (-1,8) {\small $\phi = +\frac\pi2$};
\node at (0,-0.5) {\small $\theta = 0$}; \node at (10,-0.5) {\small $\theta =
2\pi$};
\node at (6.5,8.3) {\small $\cS_0$}; \node at (6,-0.3) {\small $\cS_0$};
\draw[-] (4,8) .. controls (3, 5) and (3,3) .. (2,0);
\node at (2.5,4.4) {\small $\cS_{-1}$};\node at (2.3,4) {\small from};\node
at (2.2,3.5) {\small $O_{-\xi}$};
\draw[-] (4,0) .. controls (5, 3) and (5,5) .. (6,8);
\node at (5.2,2.4) {\small $\cS_{-1}$};\node at (5.1,2) {\small from};\node
at (4.9,1.5) {\small $O_{\xi}$};
\draw[-] (4.5,8) .. controls (3.8,5.5) .. (3.1,4.5);
\node at (4.5,6.4) {\small $\cS_{-1}$};\node at (4.4,6) {\small from};\node
at (4.2,5.5) {\small $O_{-\ks}$};
\draw[-] (4.35,8) .. controls (3.75,6) .. (3.15,4.75);
\draw[-] (4.25,8) .. controls (3.7,6.3) .. (3.3,5.3);
\draw[-, dashed] (0,1) -- (10,1); \draw[-, dashed] (0,0.7) -- (10,0.7);
\draw[-, dashed] (0,0.5) -- (10,0.5); \draw[-, dashed] (0,0.35) --
(10,0.35);
\draw[-, dashed] (0,0.25) -- (10,0.25); \draw[-, dashed] (0,0.2) --
(10,0.2);
\draw[-, dashed] (0,7) -- (10,7); \draw[-, dashed] (0,7.3) -- (10,7.3);
\draw[-, dashed] (0,7.5) -- (10,7.5); \draw[-, dashed] (0,7.65) --
(10,7.65);
\draw[-, dashed] (0,7.75) -- (10,7.75); \draw[-, dashed] (0,7.8) --
(10,7.8);
\node at (9,7.1) {\small $\bH_{k_0}$}; \node at (9,0.9) {\small
$\bH_{-k_0}$};
\node at (4.1,8.3) {\small $\theta_{-\xi}$}; \node[color=red] at (4,8) {\small
$\bullet$};
\node at (4.8,8.3) {\small $\theta_{\ks}$}; \node[color=blue] at (4.5,8) {\small
$\bullet$};
\node at (4,4.3) {\small $(\theta'_{\ks}, \phi'_{\ks})$}; \node[color=green] at
(3.1,4.5) {\small $\bullet$};
 \node[color=cyan] at (2,0) {\small $\bullet$};
\end{tikzpicture}
\caption{The parameter subset $\cM_0$ with singularity lines and $\ks =
\xi'-M\xi$.}
\label{fig:phasespace}
\end{center}
\end{figure}

Let $\cS_0 = \{ \phi = \pm\frac\pi2 \}$ be the discontinuity of the billiard
map corresponding to grazing collisions.
The forward and backward discontinuities are
$$
\cS_n = \cup_{i=0}^n T_{\scat}^{-i}(\cS_0)
\quad \text{ and } \quad
\cS_{-n} = \cup_{i=0}^n T_{\scat}^{i}(\cS_0),
$$
so that $T_{\scat}^n:\cM \setminus \cS_n \to \cM \setminus \cS_{-n}$ is a
diffeomorphism.
We line the curve $\cS_0$ with homogeneity strips $\bH_k$ bounded by
curves $|\pm\frac{\pi}{2}-\phi| = k^{-r_0}$ and $|\pm\frac{\pi}{2}-\phi| =
(k+1)^{-r_0}$, $k \geq k_0$,
for a fixed number $r_0 > 1$. The standard value is $r_0 = 2$, but as
distortion results and some other estimates
improve when $r_0$ is larger, we choose the optimal value of $r_0$ later.

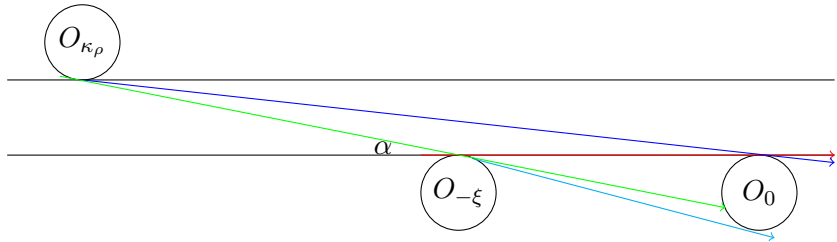
\begin{figure}[ht]
\begin{center}
\begin{tikzpicture}[scale=1]
\draw (10,0.5) circle (0.5); \node at (10,0.5) {\small $O_0$};
\draw (6,0.5) circle (0.5); \node at (6,0.5) {\small $O_{-\xi}$};
\draw (1,2.5) circle (0.5); \node at (1,2.5) {\small $O_{\ks}$};
\draw[-] (0,1) -- (11,1);
\draw[-] (0,2) -- (11,2);
\draw[->, draw=red] (5.5,1) -- (11,1);
\draw[->, draw=cyan] (6.2,0.95) -- (10.2,-0.1);
\draw[->, draw=blue] (1,2) -- (11,0.9);
\draw[->, draw=green] (0.7,2.05) -- (9.55,0.3);
\node at (5,1.1) {\small $\alpha$};
\end{tikzpicture}
\caption{A corridor collision map from $O_{-\xi}$ and $O_{-\ks}$ to
$O_0$.}
\label{fig:corridor4}
\end{center}
\end{figure}

The set $\cS_{-1}$ consists of multiple curves inside $\cM_0$, one for each
scatterer
from which a particle can reach $O_0$ in the next collision.
In Figure~\ref{fig:phasespace} we consider the corridor in the direction of
$\xi \in \Z^2$,
and drew the parts of $\cS_{-1}$ coming from scatterers $O_\xi$, $O_{-\xi}$ and $O_{-\ks}$
for some scatterer on the other side of this corridor.

\begin{lemma}\label{lem:S-kappa}
 For the $\xi$-corridor, let $(\theta_{-\xi}, \frac{\pi}{2}) \in \cM_0$ be the point
 of intersection of $\cS_0$ and the part of
 $\cS_{-1}$ associated to the scatterer $O_{-\xi}$, and
 $(\theta_{\ks}, \frac{\pi}{2}) \in \cM_0$, $\ks = \xi'-M\xi$,
 be the point of intersection of $\cS_0$ and the part of
 $\cS_{-1}$ associated to the scatterer $O_{\ks} = O_{\xi'-M\xi}$ at the other side
 (i.e., the $\xi'$-boundary) of the $\xi$-corridor, see Figure~\ref{fig:corridor4}.
 Let $(\theta'_{\ks}, \phi'_{\ks})$ be the intersection of the parts of
 $\cS_{-1}$ associated to
 the scatterers $O_{-\xi}$ and the scatterer $O_{\ks}$, see Figure~\ref{fig:phasespace}.
 Then
 $$
 |\theta_{-\xi} - \theta_{\ks}| = \frac{d_\scat(\xi)}{|\xi| M}
 \left(1+\cO\left(\frac{\scat}{|\xi| M} \right) \right)
 $$
 and
 $$
 \frac{\pi}{2}-\phi'_{\ks} = \sqrt{ \frac{2 d_{\scat}(\xi)}{\scat M} }
 \left( 1-\cO\left(\frac{\scat}{|\xi|} - \frac{1}{M} +
 \frac{\sqrt{d_{\scat}(\xi) \scat}}{|\xi| \sqrt{M} }\right)\right).
 $$
\end{lemma}

\begin{proof}
The angle $\theta_{-\xi}$ refers to the point where the common tangent line of $O_0$ and $O_{-\xi}$ touches $O_0$.
For the value $\theta_{\ks}$, $\ks = \xi'-M\xi$, we take the common tangent line to $O_0$
and $O_{\ks}$ which has slope $\frac{d_{\scat}(\xi)}{M |\xi|}
\left(1+\cO(\frac{\scat}{|\xi| M} ) \right)$.
This is then also $|\theta_{-\xi}-\theta_{\ks}|$.

\begin{figure}[ht]
\begin{center}
\begin{tikzpicture}[scale=1.]
\draw[-] (0,4) -- (12,4);
\draw(10,2) circle (2); \node at (10.35,2) {\small $O_0$};
\draw(2,2) circle (2); \node at (2.5,2) {\small $O_{-\xi}$};
\node at  (2,2) {\small $\bullet$};\node at  (10,2) {\small $\bullet$};
\draw[-] (10,2) -- (10,4); \node at (10,4) {\small $\bullet$}; \node at
(10,4.25) {\small $R$};
\node at (10,3) {\small $\bullet$}; \node at (10.3,2.9) {\small $Q$};
\node at (8.4,3.2) {\small $\bullet$}; \node at (8,2.9) {\small $P$};
\draw[-, draw = blue] (0,4.3) -- (10,3); \draw[-] (10,3.2) -- (8.4,3.2);
\node at (10,3.2) {\small $\bullet$}; \node at (10.3,3.4) {\small $Q'$};
\draw[-] (10,2) -- (7.1,4.2); \draw[->, draw = blue] (8.4,3.2) -- (7.4,6);
\node at (5,3.8) {\small $\alpha$}; \node at (9.7,2.5) {\small $-\theta'_{\ks}$};
\node at (7.4,3.5) {\small $-\phi'_{\ks}$}; \node at (7.95,3.8) {\small $\phi'_{\ks}$};
\draw[->] (5,5.2) -- (10,5.2);  \draw[->] (5,5.2) -- (2,5.2); \node at
(6,5.5) {\small $|\xi|$};
\draw[->] (5,4.5) -- (8.4,4.5);  \draw[->] (5,4.5) -- (2,4.5); \node at
(5.6,4.8) {\small $|\xi|- \scat \sin \theta'_{\ks}$};
\end{tikzpicture}
\caption{Illustration of the proof of Lemma~\ref{lem:S-kappa}}
\label{fig:lemkappa}
\end{center}
\end{figure}
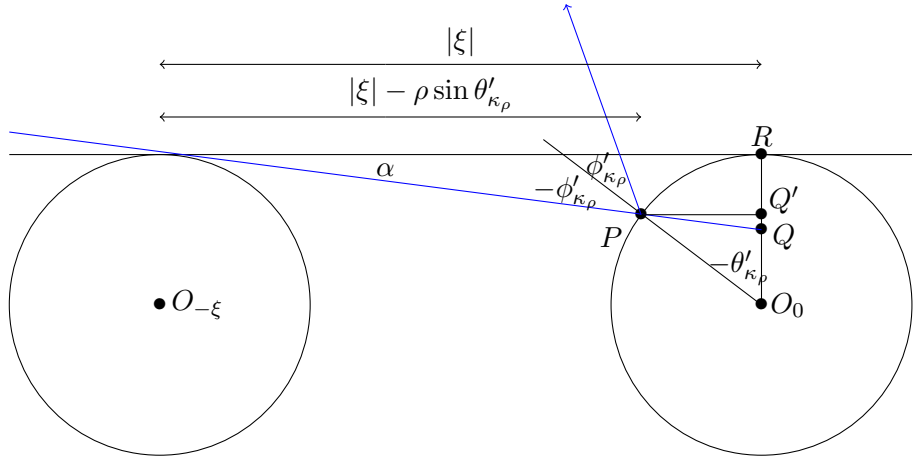

Now for the other endpoint of this piece of $\cS_{-1}$, consider
the common tangent line to $O_{-\xi}$ and $O_{\ks}$
which has slope $\tan \alpha := \frac{d_{\scat}(\xi)}{(M-1) |\xi|}
(1+\cO(\frac{\scat}{|\xi|(M-1)}))$, hitting the scatterer $O_0$ in point $P$
and when extended inside $O_0$ hits the vertical line through the center $O_0$
in point $Q$.
Let also $R$ be the tangent point of $O_0$ to the corridor, and $Q'$ is the
point on $O_0R$ at the same horizontal height as $P$, see Figure~\ref{fig:lemkappa}.
Then $|RQ| = |\xi|\sin \alpha$
whereas $|O_0Q'| = \scat - (|\xi| - \scat \sin \theta'_{\ks}) \sin \alpha = \scat \cos \theta'_{\ks}$.
The latter gives
$$
\theta'_{\ks} = \sqrt{ \frac{2|\xi|}{\scat} \sin \alpha
\left(1-\cO(\frac{\scat}{|\xi|} \sin \theta)\right)  }
=  \sqrt{ \frac{2d_\scat(\xi)}{\scat M} \left(1-\cO(\frac{\scat}{|\xi|} -
\frac{1}{M} )\right)  }.
$$
The triangle $\triangle PO_0Q$ has angles $\phi'_{\ks}$, $\alpha+\frac{\pi}{2}$ and
$\theta'_{\ks}$, which add up to $\pi$.
Hence
\begin{equation}\label{eq:tildephi}
\frac{\pi}{2} - \phi'_{\ks} = \alpha+\theta'_{\ks}
=   \sqrt{ \frac{2d_\scat(\xi)}{\scat M} }
\left( 1-\cO\left(\frac{\scat}{|\xi|} - \frac{1}{M} +
\frac{\sqrt{d_{\scat}(\xi) \scat}}{|\xi| \sqrt{M} }\right)\right)
\end{equation}
as claimed.
\end{proof}

\subsection{Hyperbolicity of the Lorentz gas with small scatterers}\label{sec-Hyperb}

The derivative $DT_{\scat} : \cT\cM \to  \cT\cM$ preserves the unstable cone field
\begin{equation}\label{eq:unstable_cone}
\cC^u_x = \left\{ (d\theta, d\phi) \in \cT_x\cM : 1 \leq \frac{1}{2\pi}
\frac{d\phi}{d\theta}
\leq 1 + \frac{\scat}{\tau_{\min}} \right\}.
\end{equation}
This is \cite[page 74]{CM} in the coordinates $\theta = r/2\pi \scat$, and we
can sharpen this cone
by replacing $\tau_{\min}$ by $\tau(x)$, the flight time at $x$ before the
next collision.
The derivative of the inverse of the billiard map preserves the stable cone
field
\begin{equation}\label{eq:stable_cone}
\cC^s_x = \{ (d\theta, d\phi) \in \cT_x\cM : -1-\frac{\scat}{\tau_{\min}}
\leq \frac{1}{2\pi} \frac{d\phi}{d\theta} \leq -1 \}.
\end{equation}
Clearly, these cone-fields are transversal uniformly over $\cM$,
and $\cS_n$ is a unstable (or stable) curve if $n > 0$ (or $n < 0$).

In the billiard literature it is common to use a pseudo-norm, the $p$-norm for unstable vectors, defined
as $\| dx \|_p = \cos \phi \, dr$. When restricted to the unstable cone, the p-norm is non-degenerate.
With the notation $\cR(x) = \frac{2}{\scat \cos \phi}$, the expansion/contraction factor $\Lambda$ on unstable vectors in the p-norm satisfies
$$
\Lambda \geq 1 + \tau(x) \cR(x) \geq 1 + \tau_{\min} \cR_{\min}  = 1 +
\frac{2\tau_{\min}}{\scat}.
$$
This proves uniform hyperbolicity of the billiard map.

In our coordinates the $p$-norm can be also expressed as  $\| dx \|_p = 2\pi\, \scat \, \cos \phi \, d\theta$, and it is related to the standard  Euclidean
norm as
$$
\| dx \| = \frac{ \sqrt{1+(\frac{d\phi}{dr})^2} }{\cos \phi}  \| dx\|_p
= \frac{ \sqrt{4\pi^2 \scat^2 +(\frac{d\phi}{d\theta})^2} }{2\pi \scat \cos
\phi}  \| dx\|_p.
$$
The expansion of $DT_{\scat}$ of unstable vectors is uniform in the $p$-norm, see
\cite[Formula (3.40)]{CM}:
$$
\frac{ \| DT_{\scat}(dx)\|_p} {\|dx\|_p} = 1+\frac{\tau(x)}{\cos \phi} (\cK +
\frac{d\phi}{dr})
= \frac{\tau(x)}{\scat \cos \phi} \left(1+\frac{1}{2\pi}
\frac{d\phi}{d\theta} + \frac{\scat \cos \phi}{\tau(x)}\right).
$$
Expressed in Euclidean norm, this gives, for $DT_{\scat}(dx) = (d\theta_1,
d\phi_1)$,
\begin{equation}\label{eq:expansion}
\frac{\|DT_{\scat}(dx)\|}{\| dx \|} =
\sqrt{\frac{4\pi^2\scat^2 + (\frac{d\phi_1}{d\theta_1})^2}{4\pi^2\scat^2 +
(\frac{d\phi}{d\theta})^2} }
\ \frac{\tau(x)}{\scat \cos \phi_1}\ \left(1+\frac{1}{2\pi}
\frac{d\phi}{d\theta} + \frac{\scat \cos \phi}{\tau(x)} \right).
\end{equation}
For later use, if $T_{\scat}(x)$ is in the homogeneity strip $\bH_k$, then $\cos
\phi_1 \approx k^{-r_0}$.

\section{Growth lemmas}\label{sec:growth}

As already mentioned in the introduction, the main line of our argument uses perturbed transfer operators acting on the Banach spaces constructed in  \cite{DZ11} and \cite{DZ14}. These works, as essentially all other methods studying statistical properties of hyperbolic billiards, rely on appropriately formulated growth lemmas, which quantify the competition of the two main dynamical effects, singularities and expansion, in these systems. The constructions of \cite{DZ11} and \cite{DZ14} involve several exponents, which thus are present in our setting, too. Additionally, we have to introduce some further exponents as we study perturbed transfer operators. Before stating the growth lemmas, here we include a table summarizing the role and the interrelation of these exponents.
Essentially, we use the same notation as in~\cite{DZ11}
except for some subscripts ${}_0$, and in fact some of the constants reduce to their value in~\cite{DZ11} if $r_0=2$.
\begin{equation}\label{eq:constants}
 \begin{cases}
  r_0 \geq 2& \text{is the exponent of the homogeneity strips:}\\
  & \qquad \bH_{\pm k} = \{ |\pm \frac{\pi}{2} - \varphi|
     \in [(k+1)^{-r_0}, k^{-r_0})\},\\[1mm]
  0<\nu<\frac12-\frac{1}{2r_0} & \text{the exponent of $\ks$ in the continuity estimate for the}\\
  & \text{transfer operator,}\\[1mm]
  \vs_0=1-\frac{2r_0\nu}{r_0-1} & \text{upper bound on }\vs \text{ in the Jensenized growth lemma, see \eqref{eq:WnuJensen},} \\[1mm]
  \alpha_0 < \min\left(\frac1{2(r_0+1)},\vs_0\right)    & \text{needed for \cite[Lemma 3.7]{DZ11}
  for general } r_0,\\[1mm]
  s_0 = \frac{1-\alpha_0(r_0+1)}{2r_0} > 0
  & \text{used in Lemma~\ref{lemma-bw},}\\[1mm]
  0 < q_0 < p_0 < \frac{1}{r_0+1} & \text{cf.\
  Lemma~\ref{lem:Distortion}},\\[1mm]
  0 < \beta_0 < \min\{\frac{\alpha_0}{2}, p_0-q_0\}. &
 \end{cases}
\end{equation}

We use a class $\cW^s$ of {\em admissible stable leaves} defined as $C^2$
leaves $W$
in the phase space such that all its tangent lines are in the stable cone
bundle,
their second derivative is uniformly bounded, $W$ is contained in a single
homogeneity strip, $\ks(x)$ is constant on $W$
and there is a $\scat$-dependent upper bound on $|W|$, namely
\begin{equation} \label{eq-size W}
\sup_{W \in \cW^s} |W| =\delta_0 := c\scat^{\nu},
\end{equation}
where the small $c > 0$, to be fixed below, is independent of $\scat$.

Let $W \in \cW^s$ be an admissible stable leaf.
The preimage $T_{\scat}^{-1}(W)$ is cut by the discontinuity lines $\cS_1$
and boundaries of homogeneity strips into at most countably many pieces
$V_i$. Note that we may have to cut the pieces $V_i$ further
into curves $W_i$ of length $\leq \delta_0$.

\subsection{The growth lemma in terms of $V_i$}

The particle  can reach the scatterer $O_0$ at the origin from corridors in all directions, indexed by $(\xi, \xi') \in \cXi $, see
Figure~\ref{fig:corridor4}.
If the previous scatterer is $\pm \xi$ itself, we call this a trajectory from
the $\xi$-boundary;
if the previous scatterer is at lattice point $\xi'-M\xi$, the trajectory
comes in from the $\xi'$-boundary,
see Remark~\ref{rem:width}.
To each such scatterer and homogeneity strip $\bH_k$ belongs at most one $V_i$, and the contraction $|T_{\scat}V_i|/|V_i|$ is governed by \eqref{eq:expansion}, where the distortion
$T_{\scat}:V_i \to T_{\scat}V_i$ is uniformly bounded, see Appendix~\ref{sec:distortion}.

\begin{prop}\label{prop:growthscaled}
Assume $0 \leq \nu < \frac12-\frac{1}{2r_0}$.
Then there is a constant $C > 0$, uniform in $\scat,\nu$ and $r_0$ such
that
$$
\sum_i |\ks(V_i)|^{\nu} \ \frac{|T_{\scat}V_i|}{|V_i|}
\leq  C \left( \scat  +
 \scat^{-\nu} \, \delta_0 \right)
$$
for every  stable leaf $W \in \cW^s$.
\end{prop}

\begin{remark}\label{rem-usegr3}
(i) Since $|W| \leq \delta_0 \leq c \scat^{\nu}$, there is $\theta_* <
1$ such that
$$
\sum_{V_i} |\ks(V_i)|^{\nu} \  \frac{|T_{\scat}V_i|}{|V_i|}
\leq  3C (\scat +c)
\leq \theta_*,
$$
for $\scat$ sufficiently small, and $c$ chosen appropriately small.
In addition, we assume that
\begin{equation}\label{eq-delta1}
\delta_1 \in (0, \delta_0/2)
\ \text{ is such that } \ \theta_* e^{C_d \delta_1^{1/(r_0+1)} } =: \theta_1 < 1
\end{equation}
for distortion constant $C_d$ from Lemma~\ref{lem:Distortion};
\\[1mm]
(ii) As later we will need $\nu > \frac13$, we can take $r_0=5$
and $\nu = \frac38$.
\end{remark}

\begin{proof}
The homogeneous admissible preimage curves $T_{\scat}^{-1}W=\cup_i V_i$
    are obtained by partitioning according to
    \begin{itemize}
    \item incoming corridors $\xi$;
    \item for a fixed corridor $\xi$, the scatterer on which $V_i$ is located. Accordingly,
        $\ks(V_i)=M\xi-\xi'$ for some $M \in \N$, and the summation is over $M$;
    \item for a fixed scatterer, the homogeneity strip containing $V_i$, that is, $V_i\subset \bH_k$ for some $k$.
    \end{itemize}

If $W$ is on the scatterer $O_{0}$ and $V_i$ is on the
    scatterer
$O_{\xi'-M\xi}$, then both of these scatterers are tangent to the same
corridor. The trajectory makes and angle $\sim \frac{d_{\scat}(\xi)}{M
|\xi|}$ with the corridor and there is a lower bound on the collision angle given by
\eqref{eq:tildephi}. This puts restrictions on how $M$ is related to
 $k$; as reflected by allowed
intersections of homogeneity strips and $M$-cells on
Figure~\ref{fig:phasespace}.
In particular
\begin{equation}\label{eq:k}
k \geq C(\scat d_{\scat}(\xi)^{-1} M)^{\frac{1}{2r_0}}
\end{equation}
which determines the range of $k$ for $M$ fixed.

We sum over the
homogeneity strips for $\xi$ and $M$ fixed on the $\xi'$ boundary.
\begin{align*}
\sum_{V_i \in \cM_{\xi'-M\xi}} |\ks(V_i)|^{\nu} \
\frac{|T_{\scat}V_i|}{|V_i|}
&\ll \frac{\scat |\xi|^{\nu} M^\nu}{|\xi| M}
\sum_{k \geq (\max\{ C(\frac{\scat M}{d_{\scat}(\xi)} , 1\} )^{\frac{1}{2r_0}
}} \frac{1}{k^{r_0}} \\
&\ll \scat^{\frac{1}{2r_0}+\frac{1}{2}} |\xi|^{\nu-1}
d_{\scat}(\xi)^{\frac{1}{2}-\frac{1}{2r_0}}
M^{\nu-\frac{3}{2}+\frac{1}{2r_0} } \\
&\ll \scat^{\frac{1}{2r_0}+\frac{1}{2}}
|\xi|^{\nu-\frac{3}{2}+\frac{1}{2r_0}}
M^{\nu-\frac{3}{2}+\frac{1}{2r_0} },
\end{align*}
where we used that the exponent $\frac{1}{2}-\frac{1}{2r_0}$
of $d_{\scat}(\xi)$ is non-negative.
By our assumption that
$\nu<\frac{1}{2}-\frac{1}{2r_0}$,
this expression is summable over $M$, and therefore
the sum over the $\xi'$-boundary of the entire $\xi$-corridor is
\begin{eqnarray*}
\sum_{\text{corridor }\xi} |\ks(V_i)|^{\nu} \
\frac{|T_{\scat}V_i|}{|V_i|} &\ll& \scat^{\frac{1}{2}+\frac{1}{2r_0} }
|\xi|^{\nu-\frac{3}{2}+\frac{1}{2r_0} }.
\end{eqnarray*}
The sum over homogeneity strips for $\xi$ fixed on the $\xi$-boundary is no
different:
\begin{align*}
\sum_{V_i \in \cM_{-\xi}} |\ks(V_i)|^{\nu} \ \frac{|T_{\scat}V_i|}{|V_i|}
&\ll \frac{\scat |\xi|^{\nu} }{ |\xi| }
\sum_{k \geq 1} \frac{1}{k^{r_0}}
\ll \scat |\xi|^{\nu-1}.
\end{align*}

 Next we sum over all opened-up corridors, indexed by
all the ``visible'' lattice points inside a sector of angle
$|W|/\sqrt{1+4\pi^2}$,
because only trajectories from scatterers within such a narrow sector can hit $O_0$ at coordinates in $W$.
The ``visible'' corridors will be denoted by $\cXi_W$.
It can happen that a single corridor, or even a single scatterer in a corridor blocks the entire sector, and we reserve one term for $|\xi| \geq 1$ (which is the worst case because the contraction of $T_{\scat}$ is the weakest).
Apart from this corridor, and since we need an upper bound, we can replace we replace $|W|$
by a stable curve of length $\delta_0$,
and apply Lemma~\ref{lem:sector_sum} for $a = 1-\nu$ and $a = \frac{3}{2} - \nu - \frac{1}{2r_0}$.
This gives
\begin{align*}
\sum_{V_i}  |\ks(V_i)|^{\nu}  \  \frac{|T_{\scat}V_i|}{|V_i|}
&\ll \scat +
 \sum_{(\xi,\xi')\in\cXi_W} \scat |\xi|^{\nu-1}
+ \scat^{\frac{1}{2}+\frac{1}{2r_0}}
|\xi|^{\nu-\frac{3}{2}+\frac{1}{2r_0}} \\
&\ll \scat + \scat^{-\nu} \delta_0 + \scat^{1-\nu} \log(1/\scat)
+ \scat^{1-\nu} \delta_0^{-1} \\
&\qquad
+  \scat^{-\nu} \delta_0 + \scat^{1-\nu} \log(1/\scat)
+ \scat^{2-\nu} \delta_0^{-1} \\
&\ll  \scat + \scat^{-\nu} \delta_0 + \scat^{1-\nu} \log(1/\scat) + \scat^{1-\nu} \delta_0^{-1}.
\end{align*}
Since $\delta_0 = c\scat^{\nu}$ and $\nu < \frac12$, this completes the proof.
\end{proof}

\subsection{The growth lemma in terms of $W_i$}\label{sec:growthW}

The pieces of preimage leaf $V_i \subset T_{\scat}^{-1}(W)$ emerge by natural
cutting
at the discontinuity set $\cS_1$ and the homogeneity strips, but even so,
their lengths can be larger than $\delta_0$, the bound of admissible stable
leaves.
We therefore need to cut them into shorter pieces, denoted as $W_i$.
In the worst case, each $V_i$ needs to be cut into $\delta_0^{-1}$ pieces,
which gives the estimate
\begin{equation}\label{eq:W}
 \sum_i |\ks(W_i)|^{\nu} \ \frac{|T_{\scat}W_i|}{|W_i|}
\leq  C \left( \scat \delta_0^{-1} + \scat^{-\nu}  \right) \ll \scat^{-\nu}.
\end{equation}
Although this estimate suffices for some purposes, it is not always good
enough for larger
iterates $T_{\scat}^n$. The next lemma (which follows \cite[Lemma 3.2]{DZ11} or \cite[Lemma 3.3]{DZ14})
achieves an estimate, uniform in $n$,
for $\nu=0$.

For the next lemma we recall some notation used in~\cite{DZ14}.
For $W \in \cW^s$,
we construct the components $\cG_k(W)$ of $T_{\scat}^{-k}W$ inductively on $k=0,\dots,n$. That is $\cG_0(W) = \{ W \}$, and to obtain $\cG_{k+1}(W)$ first we apply Proposition~\ref{prop:growthscaled} to each curve
in $\cG_k(W)$, and then we partition curves that are longer then $\delta_0$ into pieces of length between
$\delta_0$ and $\delta_0/2$.
We enumerate the leaves of the $k$-th generation $\cG_k(W)$ as $W_i^k$.

\begin{lemma}\label{lem:W}
There is a constant $C_s>0$, independent of $\scat$, such that
\begin{equation}\label{eq-consgr01}
\sum_{W_i^n\in \cG_n(W)} \frac{|T_{\scat}^nW_i^n|}{|W_i^n|}\leq C_s,
\end{equation}
and
\begin{equation}\label{eq-consgr3}
\sum_{W_i^n\in \cG_n(W)}  \frac{|W_i^n|^\vs}{|W|^\vs} \,
\frac{|T_{\scat}^nW_i^n|}{|W_i^n|}\leq C_s^{1-\vs},
\end{equation}
for all $\vs\in [0,1)$.
\end{lemma}

\begin{proof}
Define $\cL_k$ as the collection of indices such that $W_i^k \in \cG_k(W)$ is long, i.e., $|W_i^k| \geq \delta_1$ for $i \in \cL_k$, and
$\cI_n(W_j^k)$ as the collection indices of $W_i^n$ such that their
most recent long ancestor is $W_j^k \in \cG_k(W)$. If
for some $W^n_{i_1}$ no such long ancestor exists, then
set $k(i_1)=0$ and $W_{i_1}^n$ belongs to $\cI_n(W)$; if $W^n_{i_2}$ is itself long, then set $k(i_2)=n$.
Fix some $j \in \cL_k$. As for $W^i_n\in \cI_n(W_j^k)$ the preimages under
$T_{\scat}^{n-k}$ of $T_{\scat}^{n-k}W^i_n$
need not be cut
artificially (they are already short),
and due to the distortion bound from Lemma~\ref{lem:Distortion},
\begin{equation} \label{eq-consgr001}
\sum_{i \in \cI_n(W^k_j)}  \frac{|T_{\scat}^{n-k}W_i^n|}{|W_i^n|}\leq
\theta_1^{n-k},
\quad \text{ for } \theta_1 = \theta_*e^{C_d |\delta_1|^{\frac{1}{r_0+1}}}.
\end{equation}
Recall that by our assumption $\delta_1$ is so small that $\theta_1 < 1$.
In the estimate below, we group $W_i^n\in \cG_n(W)$ according to their most recent long ancestors.
\begin{eqnarray}
\label{eq-k0}
\nonumber\sum_i  \frac{|T_{\scat}^nW_i^n|}{|W^n_i|}
 &=& \sum_{k=1}^n \sum_{W_j^k \in \cL_k(W)} \sum_{i \in \cI_n(W_j^k)}
 \frac{|T_{\scat}^nW_i^n|}{|W^n_i|} + \sum_{i \in \cI_n(W)}
 \frac{|T_{\scat}^nW_i^n|}{|W^n_i|} \\
\nonumber &\leq& \sum_{k=1}^n \sum_{W_j^k \in \cL_k(W)} \left( \sum_{i \in
\cI_n(W_j^k)}
   \frac{|T_{\scat}^{n-k}W_i^n|}{|W^n_i|} \right)
   e^{\delta_1^{1/r_0+1} C_d}  \frac{|T_{\scat}^kW_j^k|}{|W^k_j|} + \theta_1^n\\
 &\leq& \sum_{k=1}^n \sum_{W_j^k \in \cL_k(W)} \theta_1^{n-k}
 \delta_1^{-1}  |T_{\scat}^kW_j^k| +\theta_1^n \nonumber \\
 &\leq&
C\delta_1^{-1}|W| \sum_{k=1}^n \theta_1^{n-k} +\theta_1^n\le C_s,
\end{eqnarray}
where we have used that for fixed $k$ and ${W_j^k \in \cL_k(W)}$, (i) $|W_j^k|\ge \delta_1$,
(ii) the $T_{\scat}^kW_j^k$ are pairwise disjoint subcurves of $W$, and (iii) $|W|\le \delta_1$.
By Jensen's inequality and \eqref{eq-k0},
\begin{eqnarray*}
 \sum_i \frac{|W_i^n|^\vs}{|W|^{\vs}} \frac{|T_{\scat}^nW_i^n|}{|W^n_i|}
 &=& \sum_i \left(  \frac{|W|}{|W_i^n|} \right)^{1-vs}
 \frac{|T_{\scat}^nW_i^n|}{|W|}
\leq  \left( \sum_i \frac{|T_{\scat}^nW_i^n|}{|W^n_i|} \right)^{1-\vs}
 \ll C_s^{1-\vs},
\end{eqnarray*}
which proves the second statement.
\end{proof}

It is worth including the following bound, which follows from \eqref{eq-consgr001} by Jensen inequality:
\begin{equation}\label{eq-consgr001Jensen}
  \sum_{i \in \cI_n(W)}  \frac{|W_i^n|^\vs}{|W|^{\vs}} \frac{|T_{\scat}^n W_i^n|}{|W_i^n|}\leq
\theta_1^{(1-\vs)n}, \qquad \text{ for all } \vs\in[0,1).
\end{equation}

\begin{remark}\label{rmk:vs0}
For further reference, we state a  version of~\eqref{eq-consgr3}
for $\nu>0$, $n=1$. Let $\vs_0=1-\frac{2r_0\nu}{r_0-1}$.
\begin{equation}\label{eq:WnuJensen}
 \sum_i |\ks(W_i)|^{\nu} \ \frac{|T_{\scat}W_i|}{|W_i|} \frac{|W_i|^\vs}{|W|^{\vs}}
\ll  \scat^{-\nu}, \qquad \text{ for all } \vs\in[0,\vs_0).
\end{equation}
This follows by Jensen's inequality from
\eqref{eq:W}, applied with $\frac{\nu}{1-\vs}$ in place of $\nu$.
The condition $\vs<\vs_0$ ensures that $\frac{\nu}{1-\vs}<\frac12-\frac{1}{2r_0}$. For the choices $r_0=5$, $\nu=\frac38$ we have $\vs_0=\frac1{16}$.
\end{remark}

\section{Banach spaces and spectral gap}
\label{subsec-Bsp}

For the exponents $p_0$ and $q_0$ defined in \eqref{eq:constants} we define the Banach spaces (of distributions) $C^{p_0}, \cB, \cB_w,\\
(C^{q_0})'$ in
analogy to~\cite{DZ14}. \footnote{Note that our set-up fits the conditions (H1)-(H5) in \cite[Section 2.1]{DZ14},
with $f(x)=f(\theta,\phi)=\cos\phi$ and $\ks=1$ in (H1), $r_h=r_0+1$ in (H2), $\xi=\frac12$ and $t_0=1$ in (H3),
$p_0=\frac{1}{r_0+1}$ in (H4) and $\gamma_0=0$ in (H5).}
We recall that  $(C^{q_0})'$ is the topological dual
of $C^{q_0}$.

Given $W\in\cW^s$, let $m_W$ be the Lebesgue measure on $W$, and define
$$
|\psi|_{W,\alpha,p_0} := |W|^{\alpha} \cos W\, |\psi|_{C^{p_0}}, \qquad
|\psi|_{C^{p_0}} := |\psi|_{C^0} + H^{p_0}_W(\psi),
$$
for $\alpha \geq 0$, $\cos W = |W|^{-1} \int_W \cos \phi\, dm_W$ (note that
$\cos W \ll k^{-r_0}$ if $W \subset \bH_{\pm k}$),
and $H^{p_0}_W(\psi)$
the H\"older constant of $\psi$ along $W$.
Also let $d_W(W_1, W_2)$ stand for the distance between leaves as in \cite[Section 3.1]{DZ11} or
\cite[Section 3.1]{DZ14};
in particular, if $W_1$ and $W_2$ belong to the same homogeneity strip, $d_W(W_1, W_2)$ is the $C^1$ distance
of their graphs in the $(\theta,\phi)$ coordinates, and otherwise infinite.

Given $W\in\cW^s$ and $h\in C^1(W)$, define the \emph{weak norm}\footnote{In the definition of
the weak norm \cite{DZ14} uses test functions with $|\psi|_{W,\gamma,p}\le 1$ for some $\gamma>0$, and requires $p<\gamma$.
However, this is needed only to ensure that the inclusion $\cB_w\hookrightarrow (C^p)'$ is injective, cf.~\cite[Lemma 3.8]{DZ14}.
Since we do not use this property, we can take $\gamma=0$ in the definition of the weak norm, and avoid additional restrictions on $p_0$.}
\begin{equation}\label{eq-weaknorm}
\|h\|_{\cB_w}:=\sup_{W\in\cW^s}\;\sup_{\stackrel {|\psi| \in C^{p_0}(W) }{
|\psi|_{W,0,p_0}\le 1} }  \int_W h\psi\, dm_W.
\end{equation}
With $q_0 < p_0$ fixed we define the distance between functions
$d(\psi_1,\psi_2)$ in the same way
as in~\cite[Section 3.1]{DZ11}.
We define the \emph{strong stable norm} by

\begin{equation}\label{eq-strongnormst}
\|h\|_s:=\sup_{W\in\cW^s}\;\sup_{\stackrel{\psi \in
C^{q_0}(W)}{|\psi|_{W,\alpha_0,q_0}\le 1}}
\int_W h\psi\, dm_W.
\end{equation}
Choosing $\eps_0 \in (0,\delta_0)$ and $\beta_0 \in (0, \min\{ \alpha_0,
p_0-q_0\})$, we define the \emph{strong unstable norm} by
\begin{equation}\label{eq-strongnormunst}
\|h\|_u:=\sup_{\eps\le\eps_0} \sup_{\stackrel {W_1, W_2\in\cW^s}{ d(W_1,
W_2)\le \eps}}
\sup_{\begin{subarray}{1}{\psi_i \in C^{p_0}(W),\ }\\ { |\psi_i|_{C^1(W)}
\leq 1 }
\\ { d_{q_0}(\psi_1,\psi_2) \leq \eps}\end{subarray}}
\frac{1}{\eps^{\beta_0}}\left|  \int_{W_1} h\psi_1\, dm_{W} - \int_{W_2}
h\psi_2\, dm_{W}\right|.
\end{equation}
The \emph{strong norm} is defined by $\|h\|_{\cB}=\|h\|_s+c_u\|h\|_u$,
where we will fix $c_u\ll 1$ (but independent of $\scat$) at the beginning of Subsection~\ref{sec:periphspgap}.

Since $C^{p_0}\subset\cB\subset\cB_w\subset (C^{q_0})'$ (see
Subsection~\ref{subsec-prop}), we have $\|h\|_{\cB_w}+\|h\|_{\cB}\leq C\|h\|_{C^1}$. As in~\cite{DZ14}, we define $\cB$ to be the completion of $C^1$ in the
strong norm and $\cB_w$ to be the completion in the weak norm.

\subsection{Transfer operator on $\cB$}
\label{subsec-prop}

Throughout we let $R_\scat: L^1(m)\to
L^1(m)$ be  the transfer operator of the billiard map $T_\scat$.
We recall that~\cite[Lemmas 3.7-3.10]{DZ11} ensure that:
i) $R_\scat(C^1)\subset\cB$ and as a consequence $R$ is well defined on $\cB$; $\cB_w$;
ii) the unit ball of $\cB$ is compactly embedded in $\cB_w$,
 and iii) $C^{p_0}\subset\cB\subset\cB_w\subset(C^{q_0})'$.

It follows that $R_\scat$ is well defined on $\cB$ and $\cB_w$,
and we also let $R_\scat$ denote the extension of this transfer operator to $\cB_w$.

\subsection{Lasota-Yorke inequalities}
\label{subsec-LY}

Using Proposition~\ref{prop:growthscaled} with $\nu=0$ and Lemma~\ref{lem:W}
we obtain the analogue of the Lasota-Yorke inequality~\cite[Proposition
2.3]{DZ14}. As our set-up fits \cite{DZ14}, our only concern is the dependence on $\scat$.
It is important to point out that our all estimates in Section~\ref{sec:growth} and
Appendix~\ref{sec:distortion} are independent of $\scat$, except that $\delta_1<\delta_0\ll \scat^{\nu}$.
\begin{lemma}[Weak norm]
\label{lemma-wn} There exists a uniform constant $C>0$ so that for all
$h\in\cB$
and for all $n\ge 0$,
\[
\|R_\scat^n h\|_{\cB_w}\le C\cdot C_s\,\|h\|_{\cB_w},
\]
where $C_s$ is given by~\eqref{eq-consgr01}.
\end{lemma}

\begin{proof}
For $W\in\cW^s$, $h\in C^1(\cM_0)$, $\psi\in C^{p_0}(W)$ with
$|\psi|_{W,\alpha_0,p_0}\le 1$,
\[
\int_W R_\scat^n h \psi\, dm_W=\sum_{W_i^n\in \cG_n(W)}\int_{W_i^n} h
\frac{J_{W_i^n} T_\scat^n}{|DT_{\scat}^n|}\psi\circ T_\scat^n\, dm_W.
\]
Using the present definition of the weak norm,
\[
\int_W R_\scat^n h \psi\, dm_W\le \sum_{W_i^n\in \cG_n(W)}\int_{W_i^n}
\|h\|_{\cB_w}
\frac{|J_{W_i} T_\scat|_{C^{p_0}(W_i)}}{|DT_\scat|}|\psi\circ T_\scat|_{C^{p_0}(W_i)}\cos
(W_i^n)\, dm_W.
\]
From here on the argument goes almost word for word as the argument in~\cite[Section 4.1]{DZ14},
except for the use of equation~\eqref{eq-consgr01} (the analogue of~\cite[Lemma 3.3(a)]{DZ14} with $\vs=0$).~\end{proof}

\begin{lemma}[Strong stable norm]
\label{lemma-ssn} Take $\delta_1$ as in~\eqref{eq-delta1} and
$\theta_1$ as in~\eqref{eq-consgr001}.
There exists a uniform constant $C>0$ so that for all $h\in\cB$
and all $n\ge 0$,
\[
\|R_\scat^n h\|_{s}\le C \left(
\theta_1^{(1-\alpha_0)n}+C_s^{1-\alpha_0}\Lambda^{-q_0n}\right)\|h\|_{s}+
C\delta_1^{-\alpha_0}\|h\|_{\cB_w}.
\]
\end{lemma}

\begin{remark}
\label{rmk-WhereScatInLY}
The compact term $C\delta_1^{-\alpha_0}\|h\|_{\cB_w}$ in Lemma~\ref{lemma-ssn}
is the only point in the Lasota-Yorke inequalities where a $\scat$-dependence arises, via $\delta_1 = c \scat^{\nu}$.
\end{remark}

\begin{proof} The argument goes almost word for word as the~\cite[Argument in Section 4.2]{DZ14},
except for the  differences:

i) We use of  equation~\eqref{eq-consgr3} with $\vs=\alpha_0$ instead
of~\cite[Lemma 3.3 (b)]{DZ14} (also with $\vs=\alpha_0$)
in~\cite[Equation (4.5)]{DZ14}. In particular, using the present definition of
the stable norm, with the same notation as in
~\cite[Section 4.2]{DZ14}, we have the following analogue of~\cite[Equation
(4.5)]{DZ14}:
\begin{align*}
\sum_{W_i^n\in \cG_n(W)} & \int_{W_i^n} h \frac{J_{W_i^n}
T_\scat^n}{|DT_\scat^n|}\left(\psi\circ T_\scat^n-\bar\psi_i\right)\, dm_W\\
&\ll \Lambda^{-q_0n} \|h\|_{s}\sum_{W_i^n\in \cG_n(W)}
\frac{|W_i^n|^{\alpha_0}}{|W|^{\alpha_0}} \,
\frac{|T_\scat^nW_i^n|}{|W_i^n|} \, \ll \Lambda^{-q_0n} \|h\|_{s},
\end{align*}
where we have used the distortion bounds of Appendix~\ref{sec:distortion} and Formula
\eqref{eq-consgr3} (with $\vs=\alpha_0$).

 ii) To obtain the analogue of \cite[Equation (4.6)]{DZ14}, as in
 \cite[Section 4.2]{DZ14}, we split the sum
\[
\sum_{k=0}^n\sum_ {j\in L_k}\sum_{i\in \cI_n(W_j^k)} |W|^{-\alpha_0} (\cos W)^{-1}
\int_{W_i^n} h \frac{J_{W_i^n}T_{\scat}^n}{|DT_{\scat}^n|}\, dm_W
\]
into a term for $k=0$  and further terms for $k=1,\dots, n$. For $k=0$, we use the strong stable norm and
\eqref{eq-consgr001Jensen} (the analogue of \cite[Lemma 3.3(a)]{DZ14}) with $\vs=\alpha_0$, giving a contribution
$\ll \|h\|_{s} \theta_1^{n(1-\alpha_0)}$. For the terms $k=1,\dots n$, we use the weak norm,
\eqref{eq-consgr3} (the analogue of \cite[Lemma 3.3(b)]{DZ14}) with $\vs=\alpha_0$,
and the fact that $|W^k_j|\ge \delta_1$ for $j\in\cL_k(W)$, resulting in a contribution of
$O(\|h\|_{\cB_w} \delta_1^{-\alpha_0})$.
\end{proof}

As in \cite{DZ14}, dealing with the strong unstable norm is the most delicate
part of the Lasota-Yorke inequality. The only difference from \cite[Argument in Section 4.3]{DZ14} is that we apply \eqref{eq-consgr01} (instead of \cite[Lemma 3.3 (b)]{DZ14}) multiple times. Note that our bound in \eqref{eq-consgr01} is independent of $\scat$, so no
$\scat$-dependence arises here.

\begin{lemma}[Strong unstable norm]
\label{lemma-usn} There exists a uniform constant $C>0$ so that for all
$h\in\cB$
and for all $n\ge 0$,
\[
\|R_\scat^n h\|_{u}\le C\cdot C_s \cdot
\Lambda^{-\beta_0 n}\|h\|_{u}+ C\cdot C_s \cdot n\|h\|_{s}.
\]
\end{lemma}
\begin{proof}
Given $W_1,W_2 \in \cW^s$ with $d(W_1,W_2)\le\varepsilon$, we may
identify matched and unmatched pieces in $T_\scat^{-n}W_{\ell}$, $\ell=1,2$.
The estimates of \cite{DZ14} on the length of the \textit{unmatched pieces} apply, thus we may estimate their contribution by the strong stable norm using
\eqref{eq-consgr01} (instead of \cite[Lemma 3.3 (b)]{DZ14}). As the length estimates give $\varepsilon^{\alpha_0/2}$, $\beta_0<\alpha_0/2$ is essential here (cf. \cite[Formulas (4.10) and (4.11)]{DZ14}, noting that $\gamma=0$ in our case).

To bound the contribution of the \textit{matched pieces} we use, on the one hand, the strong unstable norm (as in \cite[Formula (4.14)]{DZ14}) and, on the other hand, the strong stable norm (as in \cite[Formula (4.17)]{DZ14}). Here again we rely on  equation~\eqref{eq-consgr01} which plays the role of \cite[Lemma 3.3 (b)]{DZ14}. $\beta_0<p_0-q_0$ ensures that after division by $\varepsilon^{\beta_0}$ the proof of Lemma~\ref{lemma-usn} can be completed.~\end{proof}

\section{Perturbed transfer operators}
\label{sec-opRk}

A standard way of obtaining  limit theorems for dynamical systems is via the
perturbed transfer operator method.
In Section~\ref{sec-lmthm} we will use the spectral properties of the family
of perturbed transfer operators $\hat R_\scat(t), t\in \R$
with $\hat R_\scat(t) h =R(e^{it\ks} h)$, $h\in L^1(m)$.

\subsection{Continuity properties}

By definition,  $\hat R_\scat(0)=R_\scat$. Take $0 \le \nu < \frac{1}{2} - \frac{1}{2r_0}$ as in Proposition~\ref{prop:growthscaled}.
In this subsection we show the following continuity estimate:
\begin{equation}
\label{eq-cont}
\|(\hat R_\scat(t)-\hat R_\scat(0))h\|_{\cB} \le C \scat^{-\nu} |t|^\nu
\|h\|_{\cB}
\end{equation}
for some uniform constant $C$.

The argument goes parallel to Subsection~\ref{subsec-LY}, except that this time we need the estimates (i) for $\nu>0$  and (ii) only for $n=1$, we rely on~\eqref{eq:W} and~\eqref{eq:WnuJensen} instead of Lemma~\ref{lem:W}.

\begin{lemma}\label{lemma-rkw}
Assume~\eqref{eq-size W}. Then there exists a uniform constant $C>0$ so that
for all $h\in\cB$,
\[
\|R_\scat(e^{it\ks}-1) h)\|_{\cB_w} \le C \scat^{-\nu} |t|^{\nu}
\|h\|_{\cB_w}.
\]
\end{lemma}

\begin{proof} The argument goes similarly to the argument in~\cite[Section
4.1]{DZ14} restricted to the case $n=1$.
More precisely, for $W\in\cW^s$, $h\in C^1(\cM_0)$, $\psi\in C^{p_0}(W)$ with
$|\psi|_{W,\alpha_0,p_0}\le 1$,
\[
\int_W R_\scat(e^{it\ks}-1) h \psi\, dm_W=\sum_{i\in \cG_1(W)}\int_{W_i}
(e^{it\ks}-1) h \frac{J_{W_i} T_\scat}{|DT_{\scat}|}\psi\circ T_\scat\, dm_W.
\]
Using the definition of the weak norm and the inequality $|e^{ix}-1|\le
x^\nu$,
\begin{align*}
\int_W R_\scat(e^{it\ks}-1) h \psi\, dm_W \le&\ |t|^{\nu}\sum_{i\in
\cG_1(W)}\int_{W_i}\|h\|_{\cB_w} |\ks(W_i)|^\nu \\
& \qquad  \qquad \times \frac{|J_{W_i} T_\scat|_{C^{p_0}(W_i)}}{|DT_\scat|}|\psi\circ
T_\scat|_{C^{p_0}(W_i)}\cos (W_i)\, dm_W.
\end{align*}
From here on the proof goes the same as the argument in~\cite[Section
4.1]{DZ14}
except for the use of equation~\eqref{eq:W} instead of~\cite[Lemma
3.3 (b)]{DZ14}.~\end{proof}

\begin{lemma}\label{lemma-kappast}
There exists a uniform constant $C>0$ so that for all $h\in\cB$
and for all $n\ge 0$,
\[
\|R_\scat(e^{it\ks}-1) h)\|_{s}\le C|t|^{\nu} \scat^{-\nu}
\|h\|_{s}.
\]
\end{lemma}

\begin{proof}  This time we are only concerned with $n=1$, and do not need a contraction of the strong stable norm. Hence, an argument analogous to the proof of Lemma~\ref{lemma-rkw} suffices,  with
the weak norm replaced by the strong stable norm. Accordingly, we use~\eqref{eq:WnuJensen}  with $\vs=\alpha_0$
instead of~\cite[Lemma 3.3 (b)]{DZ14}.~\end{proof}

\begin{lemma}\label{lemma-kappausn}
There exists a uniform constant $C>0$ so that for all $h\in\cB$,
\[
\|R_\scat(e^{it\ks}-1) h)\|_{u}\le C|t|^{\nu}\left( \scat^{-\nu}
\cdot\|h\|_{u}+\scat^{-\nu} \cdot \|h\|_{s})\right).
\]
\end{lemma}

\begin{proof}  As with the proof of Lemma~\ref{lemma-usn}, the argument goes similar to~\cite[Argument in Section 4.3]{DZ14}, restricted to the case $n=1$. The matched and unmatched pieces can be again identified, this time for $T_{\scat}^{-1}W_{\ell}$, $\ell=1,2$.
Then, as in the proof of Lemma~\ref{lemma-rkw}, the factors $|t|^{\nu}$ and $|\ks|^{\nu}$ arise.  Clearly $\ks$ is constant on each of the (matched or unmatched) pieces, and takes the same value on any two pieces that are matched. Accordingly, the various contributions can be estimated in the same way as in proof of Lemma~\ref{lemma-usn}, with the only difference that,
by the presence of the factor $|\ks|^{\nu}$, throughout the argument \eqref{eq:W} is used instead of \eqref{eq-consgr01}.~\end{proof}

Equation~\eqref{eq-cont} follows from the definition of the norm in $\cB$
together with Lemmas~\ref{lemma-rkw},~\ref{lemma-kappast}
and~\ref{lemma-kappausn}.

\subsection{Peripheral spectrum and spectral gap}
\label{sec:periphspgap}

Choose $1>\sigma>\max\{\Lambda^{-\beta_0}, \theta_1^{(1-\alpha_0)}, \Lambda^{-q_0}\}$.
By Lemmas~\ref{lemma-wn},~\ref{lemma-ssn} and~\ref{lemma-usn} and arguing as in~\cite[Equation (2.14)]{DZ14},
we obtain the traditional Lasota-Yorke inequality for some $N\ge 1$, provided $c_u$ in the definition of $\| \ \|_{\cB}$ (below \eqref{eq-strongnormunst}) is chosen small enough in terms of $N$. That is,
\begin{equation}
\label{eq-spgap}
\|R_\scat^N h\|_{\cB}\le \sigma^N\|
h\|_{\cB}+C \delta_1^{-\alpha_0} \|h\|_{\cB_w}.
\end{equation}
Combined with the properties collected in Subsection~\ref{subsec-prop} (that
is, the relative compactness of  the unit ball of $\cB$  in $\cB_w$), equation~\eqref{eq-spgap}
shows that  the essential spectral radius of $R_\scat$ is bounded by $\sigma$ and
that the spectral radius is $1$.

Let $\Pi_\scat$ be the eigenprojection (that is, the projection on the
eigenspace of $R_\scat$) corresponding to the eigenvalue $1$. In particular,
$\Pi_{\scat} \mu=\mu$ is the invariant measure for $T_{\scat}$.
Since for every $\scat$, $T_\scat$ is mixing, the peripheral spectrum of
$R_\scat$ consists of just the simple eigenvalue at $1$.
Thus, for every $\scat>0$, the eigenprojection $\Pi_{\scat}$ corresponding to
the eigenvalue $1$ of $R_{\scat}$  can be also
characterized by
\begin{equation}
\label{eq-projmixrho}
\Pi_{\scat}h=\lim_{m\to\infty}R_{\scat}^m h,
\end{equation}
for all $h\in\cB$.

Let $Q_{\scat}$ be complementary spectral projection. From here onwards, we
exploit that for every $\scat>0$, there exist $\ga_{\scat}\in (0,1)$ and $C_\scat>0$ so that
\begin{equation}
\label{eq-complprojmixrho}
\|Q_{\scat}^m\|_{\cB}\le  C_\scat(1-\ga_{\scat})^m
\end{equation}
for every $m\ge 1$.
Altogether, $R_{\scat}^m=\Pi_{\scat} +Q_{\scat}^m$,
where $Q_{\scat}$ satisfies~\eqref{eq-complprojmixrho}.

\section{Asymptotics of the dominant eigenvalue}
\label{sec:ev}

To establish limit theorems (such as Theorem~\ref{th-mainth} below) we study the asymptotics of
\(
\E_\mu(e^{it\kappa_{m,\scat} } 1)=\E_\mu(\hat R_{\scat}(t)^m 1),
\)
as $t\to 0$ and $m\to\infty$ via the properties of $\hat R_{\scat}(t) h= R_{\scat}(e^{it\ks} h)$, $h\in\cB$.

We already know that for every $\scat \in (\frac13,\frac12)$, $1$ is a simple  eigenvalue of $\hat
R_{\scat}(0)=R_{\scat}$ when viewed as an operator from $\cB$ to $\cB$.
Due to~\eqref{eq-cont}, $\hat R_{\scat}(t)$  is $C^\nu$ (in $t$)  from $\cB$
to $\cB$.
It follows that  for $t$ in a neighbourhood of $0$, $\hat R_{\scat}(t)$ has a
dominant eigenvalue $\lambda_{\scat}(t)$ (with $\lambda_{\scat}(0)=1$).

Let $\ga_{\scat}$ be as in equation~\eqref{eq-complprojmixrho}. The continuity properties together with~\eqref{eq-complprojmixrho}
ensure that for any $\delta\in (0,\ga_{\scat})$ and $t\in B_\delta(0)$,
\begin{equation}
\label{eq-spdec}
\hat R_{\scat}(t)^m=\lambda_{\scat}(t)^m\Pi_{\scat}(t)+Q_{\scat}(t)^m,\quad \|Q_{\scat}(t)^m\|_{\cB}\le C_\scat(1-\ga_{\scat})^m,
\end{equation}
for some $C_\scat>0$ and $\Pi_{\scat}(t)^2=\Pi_{\scat}(t)$, $\Pi_{\scat}(t)Q_{\scat}(t)=Q_{\scat}(t)\Pi_{\scat}(t)=0$.
Further, for all $t\in B_\delta(0)$,
\begin{equation}
\label{eq-formpi}
\Pi_{\scat}(t)=\int_{|u-1|=\delta}(u-\hat R_{\scat}(t))^{-1}\, du,
\end{equation} for all $t$ small enough.
A standard consequence of~\eqref{eq-cont} and~\eqref{eq-complprojmixrho} is that
for every $\delta\in (0,\ga_{\scat})$ and for all $u$ so that $|u-1|=\delta$,
\begin{align}\label{eq-contdefpi}
\|(u-\hat R_{\scat}(t))^{-1}-(u-\hat R_{\scat}(0))^{-1}\|_{\cB}& \le C
\scat^{-\nu} |t|^{\nu}\|(u-\hat R_{\scat}(t))^{-1}\|_{\cB}\|(u-\hat R_{\scat}(0))^{-1}\|_{\cB} \nonumber \\
& \le  C \scat^{-\nu}\ga_{\scat}^{-2}|t|^{\nu}.
\end{align}
Hence, $\|\Pi_{\scat}(t)-\Pi_{\scat}(0)\|_{\cB}\le C\scat^{-\nu}
|t|^{\nu}\scat^{-\nu}\ga_{\scat}^{-2}|t|^{\nu}$.

The rest of this section is allocated to the study the
asymptotics of $\lambda_{\scat}(t)$ as $t\to 0$.

The following property was used in~\cite{LT16,BT17, BTT20} (see
\cite[assumption (H2)]{BTT20}) for the study of
eigenvalues of perturbed transfer operators in the Banach spaces introduced
in~\cite{DemersLiverani08}.
Here we use it to obtain an adequate analogue for the present set-up.

\begin{lemma}\label{lemma-bw}
Take $s_0 = \frac{1-\alpha_0(r_0+1)}{2r_0}$  as in \eqref{eq:constants}.
Let $h\in\cB$ and $v\in C^{p_0}$.
For every corridor with boundaries determined by $O_\xi$ and $O_{\xi'}$,
there exists a constant $C > 0$ independent of $\scat$ and $\xi$ so that
\[
\left| \int h v 1_{\{ \ks = \xi'+ N\xi \}} \, dm \right|
\le C \| h \|_s |v|_{C^{q_0}} d_{\scat}(\xi)^{\frac32-s_0} |\xi|^{-1}
\scat^{-\frac12+s_0} N^{-\frac52+s_0}.
\]
\end{lemma}

\begin{proof}
 Let $\{ W_\ell\}_{\ell \in L}$ be the foliation of the set $\{\ks
 = \xi'+\xi N\}$ into stable leaves. We can parametrise
 these leaves by their endpoints $(\ell, \frac{\pi}{2})$ in $\cS_0$, then
 $L$ is an interval of length $c \ll \frac{d_{\scat}(\xi)}{N^2 |\xi|}$ according to Lemma~\ref{lem:S-kappa}.
 The lengths of these stable leaves $|W_\ell| \leq c'$  for another constant
 $c' \ll \sqrt{\frac{2 d_{\scat}(\xi)}{\scat N}}$, again by Lemma~\ref{lem:S-kappa}.
 The measure $dm_{W_\ell}$ is Lebesgue on the $C^1$ stable leaf $W_\ell$, and
 it can be parametrised
 as $(w_\ell(\phi), \phi)$ where $w$ is $C^1$ with
 $-\frac{1}{2\pi} \frac{\scat + \tau_{\min}}{\tau_{\min}} < w'(\phi) <-
 \frac{1}{2\pi}$
 because of the direction of the stable cones, see \eqref{eq:stable_cone}.

 Let $\nu$ be a measure on $L$ that produces the decomposition of Lebesgue
 measure $m$ on
 $\{ \ks = \xi'+\xi N\}$
 along stable leaves. We have $\nu \ll m_L$ (and $d\nu/dm_L$ is bounded
 above).
 Since we need to partition stable leaves $W_\ell$ by the homogeneity strips
 $\bH_k$ near $\cS_0$ into
 pieces $W_{\ell,k} := W_\ell \cap \bH_k$, we get an extra sum over $k \geq k(c') := \lfloor (c')^{-1/r_0} \rfloor$.
 Then
 \begin{eqnarray*}
  \left| \int h v 1_{\{ \ks = N\xi+\xi' \}} \, dm\right| &=&
  \left| \int_L \sum_{k \geq k(c')} \int_{W_{\ell,k}} h\, v \, dm_{W_\ell} \
  d\nu(\ell) \right| \\
   &\ll& \left| \int_L  |v|_{C^{q_0}} \sum_{k \geq k(c')}
   \int_{W_{\ell,k}} h\, \frac{v}{|v|_{C^{q_0}}} \, dm_{W_\ell} \
   d\ell\right| \\
   &\leq& |v|_{C^{q_0}} \|h\|_s \int_L \left( \sum_{k \geq k(c')}
   |W_{\ell,k}|^{\alpha_0-1}  \int_{W_{\ell,k}} \!\!\!\!\!\!\! \cos \phi \,
   \sqrt{1+|w'(\phi)|^2} \, d\phi  \right) \ d\ell \\
   &\ll& |v|_{C^{q_0}} \|h\|_s \int_L \sum_{k \geq k(c')}
   |W_{\ell,k}|^{\alpha_0} k^{-r_0-(r_0+1)\alpha_0} \, d\ell \\
   &\leq& |v|_{C^{q_0}}\, \|h\|_s\, c\, k(c')^{1-\alpha_0(r_0+1)-r_0} \\
   &\ll& |v|_{C^{q_0}} \|h\|_s |\xi|^{-1} d_{\scat}(\xi)^{\frac32-s_0}
   \scat^{-\frac12+s_0} N^{-\frac52+s_0},
 \end{eqnarray*}
 for $s_0 = \frac{1-\alpha_0(r_0+1)}{2r_0}$, as claimed.
\end{proof}

Using~\eqref{eq-contdefpi}, Lemma~\ref{lem:p6} and Lemma~\ref{lemma-bw} we obtain
the asymptotics of the eigenvalue in Proposition~\ref{prop-eigv} below.

\begin{lemma}\label{lemma-matrix}
For $t\in\R^2$, let
$\bar A(t,\scat)=\sum_{|\xi|  \leq 1/(2\scat)} \frac{d_{\scat}(\xi)^2 \langle
t, \xi \rangle^2}{|\xi|}$.
Then
\[
 \lim_{\scat \to 0} \frac{\scat}{2} \bar A(t,\scat)
  = \frac{|t|^2}{\pi}
= \langle \Sigma t,t \rangle
\quad \text{ for } \quad
\Sigma = \begin{pmatrix} \frac{1}{\pi} & 0 \\ 0 &  \frac{1}{\pi}
\end{pmatrix} \quad \text{ as defined in~\eqref{eq:nottr}}.
\]
\end{lemma}

\begin{proof}
The coordinate axes $p=0$ and $q=0$, and the two diagonals $p=q$ and $p=-q$ divide the plane into eight sectors.
Here we count counter-clockwise with the first sector
$\cXi_1$ directly above the positive $p$-axis.
Let $\gamma = \gamma(t,\xi)$ be the angle between the vectors $t$ and $\xi$.
Let $\alpha = \arctan q/p$ and $\theta$ be the polar angles of $\xi$ and $t
\in \R^2$ respectively,
so $\gamma = \theta-\alpha$.
 For the first sector $\cXi_1$, taking into account that for every $\xi$
 there are two $\xi'$, we have
 \begin{align*}
  \sum_{(\xi,\xi') \in \cXi_1} \frac{d_{\scat}(\xi)^2 \langle t, \xi
  \rangle^2 }{|\xi|} &=
  2|t|^2 \sum_{(\xi,\xi') \in \cXi_1} \frac{d_{\scat}(\xi)^2 (|\xi| \cos
  \gamma)^2 }{|\xi|}  \\
 &= 2|t|^2 \sum_{(\xi,\xi') \in \cXi_1} \frac{d_{\scat}(\xi)^2 (\cos \theta
 \cos \alpha |\xi| + \sin \theta \sin \alpha |\xi|)^2 }{|\xi|}  \\
 &= 2|t|^2 \sum_{(\xi,\xi') \in \cXi_1} \frac{d_{\scat}(\xi)^2 (p \cos \theta + q\sin
 \theta)^2 }{|\xi|}.
 \end{align*}
The eighth sector $\cXi_8$ directly below the positive
$p$-axis gives the same result with $-q$ instead of $q$,
and sectors $\cXi_4$ and $\cXi_5$ above and below the negative $p$-axis
give the same results as sectors $\cXi_8$ and $\cXi_1$.
Therefore
$$
\sum_{(\xi,\xi') \in \cXi_1 \cup \cXi_4 \cup \cXi_5 \cup \cXi_8}
\frac{d_{\scat}(\xi)^2  }{|\xi|} =
4 |t|^2 \sum_{(\xi,\xi') \in \cXi_1}
  \frac{d_{\scat}(\xi)^2 }{|\xi|} (p^2 \cos^2 \theta +q^2 \sin^2 \theta).
$$
The same result holds the remaining  sectors with
 $\cos \theta$ replaced by $\sin \theta$ and vice versa.
Putting the results on all eight sectors together, we get
by Lemma~\ref{lem:corridor_sum}
\begin{eqnarray*}
\sum_{(\xi,\xi') \in \cXi} \frac{d_{\scat}(\xi)^2 \langle t, \xi \rangle^2
}{|\xi|}
&=& |t|^2 \sum_{(\xi,\xi') \in \cXi}
  \frac{ (|\xi|^{-1}-2\scat)^2 }{|\xi|} (p^2+q^2) \\
&=& |t|^2 \sum_{(\xi,\xi') \in \cXi}
|\xi|^{-1}-4\scat + 4\scat^2|\xi| \\
&=& |t|^2 \frac{2\pi}{\zeta(2)} \frac{1}{2\scat}
(1-\frac{2}{2}+\frac{1}{3})  (1+o(1))
= \frac{2|t|^2}{\scat \pi} (1+o(1)).
\end{eqnarray*}
Hence $\langle \Sigma t, t\rangle = \lim_{\scat\to 0} \frac{\scat}{2} \bar A(t,\scat) = \frac{|t|^2}{\pi}$, as required.
\end{proof}

For the result on the asymptotics of the eigenvalue in
Proposition~\ref{prop-eigv}, we will
also assume  some correlation decay type results. Namely, we assume that
there exist $\scat$-dependent constants $\gb_{\scat}\in (0,1)$ and $\hat C_\scat > 0$ so that
for every $j\ge 1$,
\begin{equation}\label{eq-corel}
\left|\int_{\cM_0}  (e^{it\ks}-1)\, (e^{it\ks}-1)\circ
T_{\scat}^{j}\, d\mu-\Big(\int_{\cM_0}  (e^{it\ks}-1)\,d\mu \Big)^2
\right|\le \hat C_\scat |t|^2 (1-\gb_{\scat})^j.
\end{equation}
More generally, we assume that that there exist $\scat$-dependent constants $\wgs\in (0,1)$ and
$\wCs > 0$ so that for  every $j\ge 1$ and
every $m\ge 0$
\begin{align}\label{eq-corel2}
\Big|&\int_{\cM_0}  (e^{it\ks}-1)\cdot
R_\scat(0)^{m}(e^{it\ks}-1)\, (e^{it\ks}-1)\circ
T_{\scat}^{j}\, d\mu\\
\nonumber &-\int_{\cM_0}  (e^{it\ks}-1)
R_\scat(0)^{m}(e^{it\ks}-1)\,d\mu \int_{\cM_0}
(e^{it\ks}-1)\,d\mu\\
\nonumber &-C\Big(\int_{\cM_0}  (e^{it\ks}-1)\,d\mu\Big)
\int_{\cM_0}  (e^{it\ks}-1)\, (e^{it\ks}-1)\circ
T_{\scat}^{j}\, d\mu\\
\nonumber &+C\Big(\int_{\cM_0} (e^{it\ks}-1)\,d\mu\Big)^3\Big|\le
\wCs |t|^2 (1-\wgs)^{m+j},
\end{align}
where $C=0$ if $m=0$ and $C=1$ if $m\ge 1$.
As justified in Proposition~\ref{prop:decaykappa} in Appendix~\ref{sec:decay} via the argument used in~\cite[Proof of
Proposition 9.1]{ChDo09}, assumptions~\eqref{eq-corel} and~\eqref{eq-corel2} are natural.

\begin{prop} \label{prop-eigv} Assume~\eqref{eq-complprojmixrho}, \eqref{eq-corel} and \eqref{eq-corel2}, and
let $\bar A(t,\scat) $ be as defined in Lemma~\ref{lemma-matrix}.
Let $\nu \in (\frac13,\frac12)$ and $\delta\in (0,\frac12\min\{\ga_{\scat},\hat{\gamma}_{\scat}\})$, ensuring that~\eqref{eq-spdec} holds.
Then for any $\delta_0\le\delta^{4/(3\nu-1)}$ and $t\in B_{\delta_0}(0)$,
\[
1-\lambda_{\scat}(t)=\bar A(t,\scat) \frac{\log(1/|t|)}{8\pi \scat}+E(t,
\scat),
\]
where $|E(t,\scat)|\le \wCs \,\wgs^{-2}\,|t|^2 + C|t|^2
\scat^{-2}$ for $\wCs$ and $\wgs$ as
in~\eqref{eq-corel2}
and some uniform constant $C$.\end{prop}

\begin{remark}
\label{rem-shrdelta}
It is possible to shrink $\delta_0$ further to $\delta_0<e^{-\max\{\wCs\wgs^{-2},\scat^{-2}\}}$ leading to $E(t,\scat)=o(|t|^2\log|1/t|)$.
This would mean that in the proof of main results in Section~\ref{sec-lmthm}
we would work on this very small neighborhood and obtain the same range of $n$ and $\scat$ in the final statements.
We find it more convenient to work  on the neighborhood $B_{\delta_0}(0)$ as in the statement above.
\end{remark}

\begin{remark}
\label{rmk:flightev}
Let $q_{\scat}$ be the flight function taking values
in $\R^2$ as opposed to  the displacement function $\ks$ taking values in $\Z^2$.
A similar statement holds for the the dominant eigenvalue of the perturbed operator
$R_\scat(e^{it q_{\scat}})$.
The proof is similar to the one below using that
 $|q_{\scat}-\ks| \le 1$.
\end{remark}

\begin{proofof}{Proposition~\ref{prop-eigv}} In the notation of Banach spaces
of distributions (see, for instance,~\cite{LT16})
for $h\in C^{q_0}$ we write $\langle h, \Id \rangle=\langle \Id, h\rangle=\int h \Id\, d m$ and $\langle m, h\rangle=\int h\, d m$,  where $\Id$ is both an element of
$\cB$ and of $(C^{q_0})'$.
Let $v_{\scat}(t)=\frac{ \Pi_{\scat}(t) 1}{\langle \Pi_{\scat}(t) 1, 1\rangle}$
and recall that $v_{\scat}(0)=\Id$.
Recall also that for every $\scat$,  $\lambda_{\scat}(t) v_\scat(t)=\hat
R_{\scat}(t) v_{\scat}(t)$ for $t$ small enough, and that
$\lambda_{\scat}(0)=1$.
Since $\langle v_{\scat}(t),\Id\rangle=1$,
\begin{align*}
1-\lambda_{\scat}(t)&=1-\langle \hat
R_{\scat}(t)v_\scat(t),\Id\rangle=\mu(1-e^{it\ks})+\langle (\hat
R_{\scat}(t)-\hat R_{\scat}(0))(v_{\scat}(t)-\Id),\Id\rangle\\
&=: \mu(1-e^{it\ks})+ V(t,\scat).
\end{align*}
With the meaning of inner product  clarified, for ease of notation from here
on we will write
$V(t,\scat)=\int_{\cM}(e^{it\ks}-1)(v_{\scat}(t)-\Id) \, d
m$.
We recall the terminology in Remark~\ref{rem:width}. For $\xi =
(p,q)$ with $\gcd(p,q) = 1$, we let $\xi' = (p',q')$ be the point
 uniquely determined by $\xi$ in the sense that $p'/q'$ is  convergent
 preceding $p/q$
in the continued fraction expansion of $p/q$; in particular $|\xi'|  \leq
|\xi|$.
Recall that $\cXi$ is the set of all such pairs $(\xi,\xi')$ with $|\xi|
\leq 1/(2\scat)$.
With this specified, we write
\begin{align*}
 \mu(1-e^{it\ks})=\sum_{(\xi,\xi')\in\cXi}\sum_{N=1}^\infty (e^{it(
 \xi'+ N\xi)}-1)  \mu(\{  \ks = \xi'+ N\xi \}).
\end{align*}
Using the fact that $\int \ks \, d\mu=0$, we compute that
\begin{align*}
\mu(1-e^{it\ks})&=\sum_{(\xi,\xi')\in\cXi}\sum_{N=1}^\infty
\left(e^{it( \xi'+ N\xi)}-1-it( \xi'+ N\xi)\right) \mu(\{  \ks =
\xi'+ N\xi \})\\
&=\sum_{(\xi,\xi')\in\cXi}\sum_{N=1}^{1/|t|}  \left(e^{it( \xi'+ N\xi)}-1-it(
\xi'+ N\xi)\right)  \mu(\{  \ks = \xi'+ N\xi \}) \\
&\quad +O\left(|t|\sum_{(\xi,\xi')\in\cXi}|\xi|\sum_{N>1/|t|} N\mu(\{
\ks = \xi'+ N\xi \}) \right )\\
&=\sum_{(\xi,\xi')\in\cXi}\sum_{N=1}^{1/|t|} \frac12\, \langle t,  \xi'+
N\xi\rangle^2 \mu(\{ \ks = \xi'+ N\xi \})
+O(|t|^2):=I(t,\scat)+O(|t|^2),
\end{align*}
where the involved constants in the last big $O$ are independent of $\scat$.
Further, using Lemma~\ref{lem:corridor_sum},
\begin{align*}
I(t,\scat)&=\frac{1}{4\pi \scat}\sum_{|\xi|  \leq
1/(2\scat)}\frac{d_{\scat}(\xi)^2}{|\xi|} \langle t,
\xi\rangle^2\sum_{N=\max\{1,d_\scat(\xi)/(2\scat)\}}^{1/|t|} \frac{1}{N}
\\& \qquad +O\left(|t|^2\sum_{(\xi,\xi')\in\cXi}\frac{1}{4\pi |\xi|\scat}
\sum_{N<\max\{1, d_\scat(\xi)/(2\scat)\}}4\scat^2\, N\,|\xi|\right)\\
&=\frac{1}{4\pi \scat}\sum_{|\xi|  \leq
1/(2\scat)}\frac{d_{\scat}(\xi)^2}{|\xi|}\langle t, \xi\rangle^2
\sum_{N=\max\{1,d_\scat(\xi)/(2\scat)\}}^{1/|t|} \frac{1}{N}
+O\left(|t|^2\scat^{-1}\right)\\
&=\frac{\log(1/|t|)}{4\pi \scat}\sum_{|\xi|  \leq
1/(2\scat)}\frac{d_{\scat}(\xi)^2}{|\xi|}\langle t, \xi\rangle^2 +
O\left(|t|^2\scat^{-1}\log(1/\scat)\right).
\end{align*}
Hence, with $\bar A(t,\scat)$ as in Lemma~\ref{lemma-matrix},
\[
\mu(1-e^{it\ks})=\bar A(t,\scat) \frac{\log(1/|t|)}{4\pi\scat}+O\left(|t|^2\scat^{-1}\log(1/\scat)\right).
\]
Thus, $1-\lambda_{\scat}(t)=\bar A(t,\scat) \frac{\log(1/|t|)}{4\pi\scat}+E(t,\scat)$, where
$E(t,\scat)=O\left(|t|^2\scat^{-1}\log(1/\scat)\right) + V(t,\scat)$.
It remains to estimate $V(t,\scat)$.
Note that
\begin{equation*}
v_{\scat}(t)-\Id=\frac{\mu((\Pi_{\scat}(t)-\Pi_{\scat}(0)) \Id)}{\mu(
\Pi_{\scat}(t) \Id)}\, \Pi_\scat(0)\Id+\frac{(\Pi_{\scat}(t)-\Pi_{\scat}(0)) \Id}{\mu_{\scat}(
\Pi_{\scat}(t) 1)}.
\end{equation*}
Hence,
\begin{eqnarray*}\label{eq-corel3}
V(t,\scat) &=& \frac{\mu((\Pi_{\scat}(t)-\Pi_{\scat}(0)) \Id)}{\mu( \Pi_{\scat}(t)
\Id)}\int_{\cM_0}  (e^{it\ks}-1)\, d \mu
+\frac{\int_{\cM_0} (e^{it\ks}-1)
(\Pi_{\scat}(t)-\Pi_{\scat}(0)) \Id\, d m}{\mu( \Pi_{\scat}(t)
\Id)}\\
&=&I_1(t,\scat)+I_2(t,\scat).
\end{eqnarray*}

{\bf Estimating $ I_1(t,\scat)$.} Since $\int_{\cM_0} \ks\,
d\mu=0$, we have
\[I_1(t,\scat)=\frac{\mu((\Pi_{\scat}(t)-\Pi_{\scat}(0)) \Id)}{\mu( \Pi_{\scat}(t)
\Id)}\int_{\cM_0}  (e^{it\ks}-1-it\ks)\, d \mu. \]
Now, by~\eqref{eq-contdefpi} and Lemma~\ref{lemma-bw},
\begin{align}
\label{eq-proj1111}
\nonumber \int_{\cM} | (\Pi_{\scat}(t)-\Pi_{\scat}(0))\Id|\, d\mu
&=\sum_{(\xi,\xi')\in\cXi}\sum_{N=1}^\infty  \int_{\cM} 1_{\{  \ks
= \xi'+ N\xi \}} |(\Pi_{\scat}(t)-\Pi_{\scat}(0)) \Id| \\
\nonumber &\le \sum_{(\xi,\xi')\in\cXi} |\xi|^{-\frac52+s_0}
\scat^{-\frac12+s_0}\|\Pi_{\scat}(t)-\Pi_{\scat}(0)\|_s\sum_{N=1}^\infty
N^{-\frac52}\\
&\le C \scat^{-\nu} \ga_{\scat}^{-2}|t|^{\nu}
\end{align}
for some uniform $C$. Using also that $|e^{ix}-1-ix|\le x^y$, for any $y\in (0,2]$,
\[
|I_1(t,\scat)|\le C \scat^{-\nu} \ga_{\scat}^{-2}|t|^{\nu}|t|^{2-\nu/2}\int_{\cM_0}
|\ks|^{2-\nu}\, d \mu\le C \scat^{-\nu-1} \ga_{\scat}^{-2}|t|^{\nu/2+2},
\]
where in the last inequality we have used Lemma~\ref{lem:kappa_-norm}.
Note that for $|t|\in B_{\delta_0}(0)$ with $\delta_0\le \ga_{\scat}^{4/(3\nu-1)}$, as in the statement,
$|t|^{\nu/2}<\ga_{\scat}^{2\nu/(3\nu-1)}< \ga_{\scat}^{2}$ for all $\nu \in (\frac13,\frac12)$. Thus, $|I_1(t,\scat)|\le C \scat^{-\nu-1} |t|^{2}$.

{\bf Estimating $ I_2(t,\scat)$.}
Recall that~\eqref{eq-complprojmixrho} holds and that $\delta$ is chosen so
that~\eqref{eq-formpi} holds.
Using the definition of $\Pi_{\scat}(t)$ and noting that for every $\scat$,
$(u-\hat R_{\scat}(0))^{-1}\Id=(1-u)^{-1}$,
\begin{align*}
(\Pi_{\scat}(t)-\Pi_{\scat}(0) 1&=\int_{|u-1|=\delta}(u-\hat
R_{\scat}(t))^{-1}(\hat R_{\scat}(t)-\hat R_{\scat}(0))(u-\hat
R_{\scat}(0))^{-1}\Id\, du\\
&=\int_{|u-1|=\delta}(1-u)^{-1}(u-\hat R_{\scat}(t)^{-1}(\hat R_{\scat}(t)-\hat
R_{\scat}(0))\Id\, du.
\end{align*}
Thus,
\begin{align}
\label{eq-j1j2}
\nonumber
I_2(t,\scat) &=\int_{\cM_0}(e^{it\ks}-1)\int_{|u-1|=\delta}(1-u)^{-1}(u-\hat
R_{\scat}(t))^{-1}(\hat R_{\scat}(t)-\hat R_{\scat}(0))\Id\, d u\, d m\\
&:=J_1(t,\scat)+J_2(t,\scat),
\end{align}
for
$$
J_1(t,\scat) :=\int_{\cM_0}(e^{it\ks}-1)\int_{|u-1|=\delta}(1-u)^{-1}(u-\hat
R_{\scat}(0))^{-1}(\hat R_{\scat}(t)-\hat R_{\scat}(0))\Id\, d u\, d m
$$
and
\begin{align*}
J_2(t,\scat) &:= \int_{\cM_0}(e^{it\ks}-1)\int_{|u-1|=\delta}(1-u)^{-1}\left((u-\hat
R_{\scat}(t))^{-1}-(u-\hat R_{\scat}(0))^{-1}\right) \\
& \qquad (\hat R_{\scat}(t)-\hat R_{\scat}(0))\Id\, d u\, d m\\
&= \int_{\cM_0}(e^{it\ks}-1)\int_{|u-1|=\delta} (1-u)^{-1}(u-\hat
R_{\scat}(t))^{-1}(\hat R_{\scat}(t)-\hat R_{\scat}(0)) (u-\hat R_{\scat}(0))^{-1} \\
&\qquad \times (\hat R_{\scat}(t)-\hat R_{\scat}(0))\Id\, d u\,
d m =: K_1(t,\scat)+K_2(t,\scat),
\end{align*}
where
\begin{align}\label{eq-K2}
\nonumber
K_1(t,\scat)=\int_{\cM_0}(e^{it\ks}-1)\int_{|u-1|=\delta}&(1-u)^{-1}(u-\hat
R_{\scat}(0))^{-1}(\hat R_{\scat}(t)-\hat R_{\scat}(0))\\
&\times(u-\hat R_{\scat}(0))^{-1}(\hat R_{\scat}(t)-\hat R_{\scat}(0))\Id\, d u\,
d m
\end{align}
and
\begin{align*}
\nonumber K_2(t,\scat)=\int_{\cM_0}(e^{it\ks}-1) &
\int_{|u-1|=\delta} (1-u)^{-1}
\left((u-\hat R_{\scat}(t))^{-1}-(u-\hat R_{\scat}(0))^{-1}\right) \\
&\times (\hat R_{\scat}(t)-\hat R_{\scat}(0)) (u-\hat R_{\scat}(0))^{-1}(\hat
R_{\scat}(t)-\hat R_{\scat}(0))\Id\, d u\, d m.
\end{align*}
We first treat $K_2(t,\scat)$.
Note that for $u$ in the chosen contour, $\|(u-R_\scat (t))^{-1}\|_{\cB}\le \ga_{\scat}^{-1}$.
Using~\eqref{eq-contdefpi}, for all such $u$,
\begin{align*}
&\left\|\left((u-\hat R_{\scat}(t))^{-1}-(u-\hat R_{\scat}(0))^{-1}\right)(\hat
R_{\scat}(t)-\hat R_{\scat}(0)) (u-\hat R_{\scat}(0))^{-1}(\hat R_{\scat}(t)-\hat
R_{\scat}(0))\right\|_{\cB} \\
&\le C  \scat^{-2\nu}|t|^{3\nu}\ga_{\scat}^{-3}.
\end{align*}
This together with Lemma~\ref{lemma-bw} gives that
\begin{align*}
|K_2(t,\scat)| &\leq  \sum_{(\xi,\xi')\in\cXi} \sum_{N=1}^\infty
\int_{\cM_0} \int_{|u-1|=\delta}|1-u|^{-1}1_{\{  \ks = \xi'+ N\xi
\} } |e^{it\ks}-1|\\
&\times \left|(u-\hat R_{\scat}(t))^{-1}-(u-\hat R_{\scat}(0))^{-1}
(\hat R_{\scat}(t)-\hat R_{\scat}(0) (u-\hat R_{\scat}(0))^{-1}(\hat
R_{\scat}(t)-\hat R_{\scat}(0)) \Id\right|
\, d u\, d m\\
&\le|t|^{3\nu} \scat^{-3\nu} \ga_{\scat}^{-3} \sum_{(\xi,\xi')\in\cXi}
|\xi|^{-\frac52+s_0} \scat^{-\frac12+s_0}\sum_{N=1}^\infty |t|
N^{-\frac32}\le C \scat^{-3\nu} \ga_{\scat}^{-3}|t|^{3\nu +1}.
\end{align*}
Hence, $|K_2(t,\scat)| \leq C \scat^{-1}\ga_{\scat}^{-3}\,|t|^{2}\, t^{3\nu-1}=C \scat^{-1}\ga_{\scat}^{-3}\,|t|^{2}\, \ga_{\scat}^4$
for all $|t|\in B_{\delta_0}$ with $\delta_0<\ga_{\scat}^{4/(3\nu-1)}$. It follows that
$|K_2(t,\scat)| \leq C \scat^{-1}|t|^{2}$.\\

{\bf{Estimating $J_1(t,\scat)$ in~\eqref{eq-j1j2} and $K_1(t,\scat)$
in~\eqref{eq-K2}.}}
These terms are in, some sense, independent of the Banach space $\cB$ (see
the explanation below)
and can be analysed either via the correlation
function~\eqref{eq-corel} or its generalization~\eqref{eq-corel2}.
The rest of the proof is allocated to this type of analysis.

We start with $J_1(t,\scat)$ defined in~\eqref{eq-j1j2}, which is easier
using~\eqref{eq-corel}.
Recall that $\hat R_{\scat}(0)=R_{\scat}$ and $\int_{|u-1|=\delta}
(1-u)^{-2} \, du=0$ due to Cauchy's theorem.
This gives
\begin{align*}
J_1(t,\scat)
&=\int_{\cM_0}(e^{it\ks}-1)\int_{|u-1|=\delta}(1-u)^{-1}\sum_{j=0}^\infty
u^{-j-1} R_{\scat}^j\, R_\scat(e^{it\ks}-1)\Id\, d u\, d m \\
&\quad -\left(\int_{\cM_0}(e^{it\ks}-1)\, d \mu\right)^2
\int_{|u-1|=\delta}(1-u)^{-1}\sum_{j=0}^\infty u^{-j-1}\, d u\\
&=\int_{|u-1|=\delta}(1-u)^{-1}\sum_{j=0}^\infty
u^{-j-1}\int_{\cM_0}(e^{it\ks}-1)  R_\scat^j\,
R_{\scat}(e^{it\ks}-1)\Id\, d m\, d u \\
&\quad -\left(\int_{\cM_0}(e^{it\ks}-1)\, d \mu\right)^2
\int_{|u-1|=\delta}(1-u)^{-1}\sum_{j=0}^\infty u^{-j-1}\, d u.
\end{align*}
Swapping the order of the integrals is allowed due
to~\eqref{eq-corel}. The quantity
\[
\Big(\int_{\cM_0} (e^{it\ks}-1)
R_\scat^{j+1}(e^{it\ks}-1)\, d\mu
-\int_{\cM_0}(e^{it\ks}-1)\, d \mu \Big)^2\]
decays exponentially fast.
Hence, we can write
\begin{align*}
J_1(t,\scat)& =\int_{|u-1|=\delta}(1-u)^{-1}\sum_{j=0}^\infty u^{-j-1}\\
&\quad \times\Big(\int_{\cM_0}
(e^{it\ks}-1)\,(e^{it\ks}-1)\circ T_{\scat}^{j+1} \,
d\mu-\left(\int_{\cM_0}(e^{it\ks}-1)\, d \mu\right)^2 \Big)\, du.
\end{align*}
Using Lemma~\ref{lem:kappa_-norm} to control the dependence on $\scat$, $\left(\int_{\cM_0}(e^{it\ks}-1)\, d \mu\right)^2\le C
|t|^2\scat^{-2}$.
Next,
recall that~\eqref{eq-complprojmixrho} holds and
that $\delta < \frac12 \min\{ \ga_{\scat}, \gb_{\scat} \}$.
Note that for $|u-1|=\delta$, we have
$|u|^{-(j+1)}\ll (1-\gb_{\scat}/2)^{-(j+1)}$.
This together with~\eqref{eq-corel} gives
\begin{align*}
|J_1(t,\scat)|&\le C_\scat\,|t|^2
\int_{|u-1|=\delta}|1-u|^{-1}\sum_{j=0}^\infty
|u|^{-j-1}\left(1-\gb_{\scat}\right)^{j+1}\\
&\ll \hat C_\scat\,|t|^2
\sum_{j=1}^\infty\left(\frac{1-\gb_{\scat}}{1-\gb_{\scat}/2}\right)^{j+1}\leq
2\hat C_\scat\,|t|^2\, \gb_{\scat}^{-1}.
\end{align*}

An argument similar to the one above used in estimating $ J_1(t,\scat)$ with
~\eqref{eq-corel2} instead of~\eqref{eq-corel} allows us to deal with
$K_1(t,\scat)$ defined in~\eqref{eq-K2}.
Compute that
\begin{align*}
K_1(t,\scat)=\int_{\cM_0}(e^{it\ks}-1)\int_{|u-1|=\delta}&(1-u)^{-1}\sum_{m\ge
1} u^{-m}\sum_{j\ge 1} u^{-j}\hat R_{\scat}(0)^{j}(e^{it\ks}-1)\\
&\times\hat R_{\scat}(0)^{m}(e^{it\ks}-1)\, d u\, d m.
\end{align*}
Let
\begin{align*}
E(t,\scat)&=\int_{\cM_0}
(e^{it\ks}-1)\,d\mu\,\int_{|u-1|=\delta}(1-u)^{-1}\\
&\qquad  \times \sum_{j\ge 1}
u^{-j}\sum_{m\ge 1} u^{-m}\int_{\cM_0}  (e^{it\ks}-1)
R_\scat(0)^{m}(e^{it\ks}-1)\,d\mu\, du \\
&-\int_{|u-1|=\delta}(1-u)^{-1}\sum_{j\ge 1} u^{-j}\sum_{m\ge 1}
u^{-m}\int_{\cM_0}  (e^{it\ks}-1)\,d\mu\, \\
& \qquad \times \int_{\cM_0} (e^{it\ks}-1)\, (e^{it\ks}-1)\circ  T_{\scat}^{j}\,
d\mu\, du\\
&-\Big(\int_{\cM_0}
(e^{it\ks}-1)\,d\mu\Big)^3\int_{|u-1|=\delta}(1-u)^{-1}\sum_{j\ge
1} u^{-j}\sum_{m\ge 1} u^{-m}\, du\\
&=(E_1(t,\scat)-E_2(t,\scat))\,\int_{\cM_0}
(e^{it\ks}-1)\,d\mu-E_3(t,\scat).
\end{align*}
Using~\eqref{eq-corel2}, we obtain
\begin{align*}
\Big|K_1(t,\scat) -E(t,\scat) \Big|\le \wCs|t|^2 \sum_{m\ge 1}
|u|^{-m}\sum_{j\ge 1} |u|^{-j}(1-\wgs)^{m+j} \le 4 \wCs |t|^2
\wgs^{-2},
\end{align*}
where in the last inequality we proceeded as in estimating $J_1$ above.

Finally, we need to argue that $E$ is bounded by $|t|^2$. First,
\begin{align*}
&E_1(t,\scat)=\int_{|u-1|=\delta}(1-u)^{-1}\sum_{j\ge 1} u^{-j}\sum_{m\ge 1}
u^{-m}\int_{\cM_0}  (e^{it\ks}-1)
R_\scat(0)^{m}(e^{it\ks}-1)\,d\mu\, du\\
& =\int_{|u-1|=\delta}(1-u)^{-2}\sum_{m\ge 1} u^{-m}\int_{\cM_0}
(e^{it\ks}-1)\,(e^{it\ks}-1)\circ T_\scat^m\,d\mu\, du\\
&=\int_{|u-1|=\delta}(1-u)^{-2}\sum_{m\ge 1} u^{-m}  \\
& \qquad \times \left(\int_{\cM_0}
(e^{it\ks}-1)\,(e^{it\ks}-1)\circ T_\scat^m\,d\mu\, du -
\left(\int_{\cM_0}  (e^{it\ks}-1)\,d\mu\right)^2\,d\mu\right) \\
&\qquad + \left(\int_{\cM_0}
(e^{it\ks}-1)\,d\mu\right)^2\int_{|u-1|=\delta}(1-u)^{-2}\sum_{m\ge
1} u^{-m}\, du= E_1^1(t,\scat)+E_1^2(t,\scat).
\end{align*}
Using~\eqref{eq-corel}, we have that $|E_1^1(t,\scat)|\le 2\hat
C_\scat\,|t|^2\, \gb_{\scat}^{-1}$.

Also,
\(
E_2(t,\scat)=\int_{|u-1|=\delta}(1-u)^{-2}\sum_{j\ge 1} u^{-j}\int_{\cM_0}
(e^{it\ks}-1)\, (e^{it\ks}-1)\circ  T_{\scat}^{j}\, d\mu
\)
and again by~\eqref{eq-corel} and Cauchy's theorem,  $| E_2(t,\scat)|\leq
4\hat C_\scat\,|t|^2\, \gb_{\scat}^{-2}$.
Finally, $E_3(t,\scat)=0$.
Altogether, $| K_1(t,\scat)|\leq  8\wCs\,|t|^2\, \wgs^{-2}$.
\end{proofof}

\section{Limit theorems and mixing as $\scat\to 0$}
\label{sec-lmthm}

The first result below is the non-standard Gaussian limit law, known to hold
when the horizon is infinite. It is a precise version of Theorem~\ref{thm:meta}
stated in Subsection~\ref{subsec:ours}.

Our main contribution lies in characterizing the limit paths allowed as
$\scat\to 0$; this is done up to the
unknown $\ga_{\scat}$, $C_\scat$  in~\eqref{eq-complprojmixrho} and
$\wCs$, $\wgs$ as in~\eqref{eq-corel2}.

Throughout this section, the notation is the same in Subsection~\ref{subsec:recfix}.
In particular, $b_{n,\scat}=\frac{\sqrt{n\log (n/\scat^2)}}{\sqrt{4\pi}\ \scat}$,
and the variance matrix $\Sigma$ are defined as in~\eqref{eq:nottr}, in agreement with
Lemma~\ref{lemma-matrix}.
We recall that $\implies$ stands for convergence in distribution with respect to the invariant measure $\mu$.
\begin{thm}\label{th-mainth}
Let $\ga_{\scat}$, $C_\scat$ be as in~\eqref{eq-complprojmixrho}, let $\wgs$, $\wCs$ be as in~\eqref{eq-corel2}
and let $C$  be as in Proposition~\ref{prop-eigv}.

Set
$M(\scat)=\max\{ C_\scat
\ga_{\scat}^{-2},\scat^2 \wCs
\wgs^{-2}\} + C$. Let $\scat\to 0$ and simultaneously $n\to\infty$ in such a way that
$M(\scat)= o(\log n)$. Then
\[
\frac{\kappa_{n,\scat}}{b_{n,\scat}}\implies {\mathcal N}(0,\Sigma).
\]
\end{thm}

\begin{remark}\label{rmk:flightlimth}
A similar statement holds for the flight function $q_{\scat}$.
The only change in the proof is the use of Remark~\ref{rmk:flightev} instead of Proposition~\ref{prop-eigv}.
\end{remark}

\begin{proof} Throughout we let $\delta<\frac12 \min\{\ga_{\scat},\gb_{\scat}\}$, so that
we can use Proposition~\ref{prop-eigv} with $\delta_0=\delta^{4/(3\nu-1)}$.
By \eqref{eq-spdec}, for $t\in B_{\delta_0}(0)$,
\begin{align*}
\E_\mu(e^{i t\kappa_{n,\scat}} 1) &=\E_\mu(\hat R_{\scat}(t)^n
1)=\lambda_{\scat}(t)^n\int_{\cM_0}\Pi_{\scat}(t) 1\,
d\mu+\int_{\cM_0}Q_{\scat}(t)^n1\, d\mu\\
&=\lambda_{\scat}(t)^n\int_{\cM_0}\Pi_{\scat}(t) 1\,
d\mu+O(C_\scat\,(1-\ga_{\scat})^n).
\end{align*}

Note that the assumption $M(\scat)=o(\log n)$ ensures that, for $\scat$ small enough, $\frac{t}{b_{n,\scat}}\in B_{\delta_0}(0)$ for all $t\in\R^2$.
Hence,
as $n\to\infty$ and given the range of $n$, equivalently as $\scat\to 0$,
\[
\left| \E_\mu \left(\exp\left({it \frac{\kappa_{n,\scat}}{b_{n,\scat}}}\right)
\right) - \lambda_{\scat}
\left(\frac{t}{b_{n,\scat}}\right)^n\int_{\cM_0}\Pi_{\scat}
\left(\frac{t}{b_{n,\scat}}\right) 1\, d\mu\right|\to 0,
\]
for all $t\in\R^2$.

Also, it follows from~\eqref{eq-contdefpi} that $\|\Pi_{\scat}\left(\frac{t}{b_{n,\scat}}\right)
-\Pi_{\scat}(0) \|_{\cB}\to 0$, as $n\to\infty$ and given the range of $n$, equivalently as $\scat\to 0$. Thus,
a standard argument based on the dominated convergence theorem shows that as
$n\to\infty$, equivalently as $\scat\to 0$,
\[
\left| \E_\mu \left(\exp\left({it \, \frac{\kappa_{n,\scat}}{b_{n,\scat}}}\right)
\right) - \lambda_{\scat} \left(\frac{t}{b_{n,\scat}}\right)^n\right|\to 0.
\]
It remains to understand $\lambda_{\scat}
\left(\frac{t}{b_{n,\scat}}\right)^n$ as $\scat\to 0$.
Since $\delta_0=\delta^{4/(3\nu-1)}$, we can apply Proposition~\ref{prop-eigv}
to obtain
\begin{align*}
n\left(1-\lambda_{\scat}
\left(\frac{t}{b_{n,\scat}}\right)\right)=&\frac{n}{8\pi \scat}\,\bar
A\left(\frac{t}{b_{n,\scat}}\,,\,\scat\right) \log(b_{n,\scat}/|t|)
+n\, O\left(\left(\wCs \wgs^{-2} + C\scat^{-2}\right)
\left(\frac{|t|}{b_{n,\scat}}\right)^2\right).
\end{align*}
By assumption, $M(\scat)=o(\log n)$.  Hence, as $\scat\to0$,
\begin{align*}
& n\, \left(\wCs \wgs^{-2} + C\scat^{-2}\right)
\left(\frac{|t|}{b_{n,\scat}}\right)^2
= \left(\wCs \wgs^{-2} + C\scat^{-2}\right)\frac{4\pi
|t|^2\scat^2}{\log(n/\scat^2)} =O\left(\frac{M(\scat)}{\log n}\right)\cdot |t|^2= o(1)\cdot |t|^2\to 0.
\end{align*}
Now, given that $\bar A$ is as in Lemma~\ref{lemma-matrix},
\[
\frac{n}{4\pi \scat}\,\bar A\left(
\frac{t}{b_{n,\scat}}\,,\,\scat\right)=\frac{1}{\log
(n/\scat^2)}\frac{1}{\scat}\, \scat^2\bar A\left(t,\scat\right)=\frac{\scat\bar
A(t,\scat)}{\log (n/\scat^2)}.
\]
Also, using Lemma~\ref{lemma-matrix} and recalling the range of $n$,
\begin{align*}
\lim_{\scat\to 0}\, &\frac{n}{4\pi \scat}\,\bar A\left(\frac{t}{b_{n,\scat}}\,
,\, \scat\right) \log\left(\frac{b_{n,\scat}}{|t|}\right)
= \lim_{\scat\to 0} \frac{\scat\bar A\left(t,\scat\right)}{\log
(n/\scat^2)}\log\left(\frac{b_{n,\scat}}{|t|}\right)\\
&= \lim_{\scat\to 0}\frac{\scat}{2}
\frac{\bar A(t,\scat )} {\log \left( \frac{\sqrt n}{\scat}\right) }
\log\left(\frac{\sqrt n}{\scat} \frac{\sqrt{\log(n/\scat^2)} } { \sqrt{4\pi}
|t|} \right)
=\langle\Sigma t, t\rangle,
\end{align*}
where in the last equality we have used Lemma~\ref{lemma-matrix} and the
uniform convergence theorem for slowly varying functions.
Putting the above together,
\begin{align}
\label{eq:lambdaton}
\lim_{\scat\to 0} \lambda_{\scat} \left(\frac{t}{b_{n,\scat}}\right)^n
=\lim_{\scat\to 0} \exp\left(- n \left(1-\lambda_{\scat}
\left(\frac{t}{b_{n,\scat}}\right)\right)\right)
=\exp\left(-\frac{1}{2}\langle \Sigma t, t\rangle \right),
\end{align} for any $t\in\R^2$. This completes the proof of Theorem~\ref{th-mainth} by Levy's continuity theorem.
\end{proof}

The next result gives a local limit theorem as  $\scat\to 0$, again up to
the
unknown $\ga_{\scat}$, $C_\scat$ and $\wgs$, $\wCs$.
This is possible due to the present proof based on spectral methods
which produces the fine control of the eigenvalue in
Proposition~\ref{prop-eigv}.
The present proof of local limit theorem for the infinite horizon
is new even for $\scat$ fixed. We recall that the only proof of such a local
limit is given in~\cite{SV07} via the abstract results
in~\cite{BalintGouezel06}
for Young towers.  Our proof relies on Proposition~\ref{prop-eigv},
which is new in the set-up of the Banach spaces considered here
and it relies heavily on Appendix~\ref{sec:decay}
and on Proposition~\ref{prop:growthscaled} (which provides useful continuity estimates).

In the notation of Theorem~\ref{th-mainth} we let
$\Phi_\Sigma$ be the density of a Gaussian  random variable distributed
according to ${\mathcal N}(0,\Sigma)$
and recall from Section~\ref{subsec-prop} that $C^{p_0}\subset\cB$.

\begin{thm}\label{th-llt}
Assume the assumptions and notation of Theorem~\ref{th-mainth};
in particular $M(\scat)$ is defined in the same way.
Let $v\in C^{p_0}(\cM)$ and $w\in L^a(\cM)$, for $a>1$.

Let $\scat\to 0$ and simultaneously $n\to\infty$ in such a way that
$M(\scat)=o(\log n)$.
Then
\[
\left| \int_{\cM} v1_{\{\kappa_{n,\scat}=N\}} w\circ T_\scat^n\, d\mu
-\frac{\E_\mu(v)\,\E_\mu(w)}{(b_{n,\scat})^2}\, \Phi_\Sigma\left(\frac{N}{b_{n,\scat}}\right)\right|\to
0.
\]
uniformly in $N\in\Z^2$.
\end{thm}
\begin{remark}\label{rmk:llt}
A similar statement holds for the flight function $q_{\scat}$.
By a similar argument, using Remark~\ref{rmk:flightev} instead of Proposition~\ref{prop-eigv},
we obtain $(b_{n,\scat})^2\mu(\{q_{n,\scat}\in V\})\to\Phi_\Sigma(0)Leb_{\R^2}(V)$, for any compact
neighborhood $V\in\R^2$ with $Leb_{\R^2}(\partial V)=0$.
A uniform LLT for $q_{\scat}$ can be obtained by, for instance, a straightforward
adaptation of the
argument used in~\cite[Proof of Theorem 2.7]{MTaihp}.
\end{remark}

It is known that for every $\scat>0$, $\ks$ is aperiodic, i.e., there exists no non-trivial solution to the equation $e^{it\ks}g\circ T_\scat=g$.
The aperiodicity of $\ks$ has been used in~\cite{SV07} (in fact, in~\cite{SV04}) to provide LLT for fixed $\scat$.
Given Proposition~\ref{prop-eigv} and the aperiodicity of $\ks$, the proof of Theorem~\ref{th-llt} is classic, see~\cite{AD01}
and for a variation of it that provides the uniformity in $N$, see,  for
instance,~\cite[First part of Proof of Theorem 2]{Pene18}.
The proof below recalls the main elements needed to obtain the range of $n$ in the statement.
\\

\begin{proofof}{Theorem~\ref{th-llt}}
Let $\delta_0=\delta^{4/(3\nu-1)}$ be so that~\eqref{eq-formpi},~\eqref{eq-complprojmixrho}  and Proposition~\ref{prop-eigv} hold for all $|t|\in B_{\delta_0}(0)$. Since $\ks$ is aperiodic, a known argument (see~[Lemma 4.3 and Theorem 4.1]\cite{AD01}) shows that
$\|\hat R_{\scat}(t)^n\|_{\cB}\le C_\scat(1-\ga_{\scat})^n$, for all $|t|\ge \delta_0$.
 It follows that $|\E_\mu(\hat R_{\scat}(t)^n 1)|\le \|\hat R_{\scat}(t)^n
 \|_{\cB}\le C_\scat\,(1-\ga_{\scat})^n$ for every $|t| \in (\delta_0,\pi)$.
Thus, using that $v\in C^{p_0}\subset\cB$,
\begin{align} \label{eq:llt}
\nonumber \int_{\cM} & v1_{\{\kappa_{n,\scat}=N\}} w\circ T_\scat^n\, d\mu
=\frac{1}{4\pi^2}\int_{ [-\pi,\pi]^2} e^{-itN} \int_{\cM}\hat R_{\scat}(t)^n v
\, w \, d\mu\, dt\\
\nonumber &=\frac{1}{4\pi^2}\int_{ [-\delta_0,\delta_0]^2} e^{-itN}
\int_{\cM}\hat R_{\scat}(t)^n v \, w \, d\mu\, dt
+O\left(C_\scat\,(1-\ga_{\scat})^n\right)\\
\nonumber&=\frac{1}{4\pi^2}\int_{ [-\delta_0,\delta_0]^2} e^{-itN} \lambda_\scat
\left( t\right)^n\int_{\cM} \Pi_{\scat}(t)v \, w \, d\mu\, dt
+O\left(C_\scat\,(1-\ga_{\scat})^n+\hat C_\scat\,(1-\gb_{\scat})^n\right)\\
&=\frac{1}{4\pi^2}I(\scat,t)+O\left(C_\scat\,(1-\ga_{\scat})^n\right).
\end{align}
Recall that $w\in L^a$, $a>1$ and set $b=a/(a-1)$. Using the H\"older inequality,
\begin{align*}
I(\scat,t) &=\int_{ [-\delta,\delta]^2} e^{-itN} \lambda_\scat \left(
t\right)^n\, dt\, \int_{\cM} v \, d\mu\,\int_{\cM} w \, d\mu\\
&\quad  +\int_{ [-\delta,\delta]^2} e^{-itN} \lambda_\scat \left(
t\right)^n\int_{\cM} (\Pi_{\scat}(t)-\Pi_{\scat}(0))v \, w \, d\mu\, dt \\
&=\int_{ [-\delta,\delta]^2} e^{-itN} \lambda_\scat \left( t\right)^n\, dt\,
\int_{\cM} v \, d\mu\,\int_{\cM} w \, d\mu\\
&\quad + O\left(\|w\|_{L^a(\mu)}\int_{ [-\delta,\delta]^2}
\left|\lambda_\scat \left( t\right)^n\right|\left(\int_{\cM}
|(\Pi_{\scat}(t)-\Pi_{\scat}(0))v|^b\, d\mu\right)^{1/b}\, dt\right).
\end{align*}
Recall $v\in\cB$. Using~\eqref{eq-contdefpi},~\eqref{eq-cont} and Lemma~\ref{lemma-bw} and proceeding as in equation~\eqref{eq-proj1111},
\[
\Big(\int_{\cM} |(\Pi_{\scat}(t)-\Pi_{\scat}(0))v|^b\, d\mu\Big)^{1/b}
\le C \scat^{-\nu}\ga_{\scat}^{-2} |t|^{\nu}\le C \scat^{-2}|t|^\eps,
\]
for some uniform $C$ and some $\eps>0$. In the last inequality we have used that $|t|<\delta_0$. Thus,
\begin{align*}
I(\scat,t) &=\int_{ [-\delta_0,\delta_0]^2} e^{-itN} \lambda_\scat \left(
t\right)^n\, dt\, \int_{\cM} v \, d\mu\,
\int_{\cM} w \, d\mu+ O\left(\|w\|_{L^a(\mu)}\scat^{-2} \int_{
[-\delta_0,\delta_0]^2} |t|^{\eps} \left|\lambda_\scat \left( t\right)^n\right|\, dt\right).
\end{align*}
With a change of variables,
\begin{align}
\label{eq:bigOllt}
\nonumber I(\scat,t) &=\frac{1}{(b_{n,\scat})^2}\int_{ [-\delta_0
b_{n,\scat},\delta_0 b_{n,\scat}]^2} e^{-iu\frac{N}{b_{n,\scat}}}
\lambda_{\scat} \left(\frac{u}{b_{n,\scat}}\right)^n\, du\, \int_{\cM} v \,
d\mu\,\int_{\cM} w \, d\mu\\
&\qquad + O\left(\|w\|_{L^a(\mu)} \frac{\scat^{-2}}{(b_{n,\scat})^3}\int_{
[-\delta_0 b_{n,\scat},\delta_0 b_{n,\scat}]^2} |u|^{\eps} \left| \lambda_{\scat}
\left(\frac{u}{b_{n,\scat}}\right)^n\right| du\right).
\end{align}
Given the range of $n$ in the statement, we use~\eqref{eq:lambdaton} to
obtain
\begin{align*}
\lim_{\scat\to 0}\left|4\pi^2\,\int_{ [-\delta_0 b_{n,\scat},\delta_0
b_{n,\scat}]^2} e^{-iu\frac{N}{b_{n,\scat}}}  \lambda_{\scat}
\left(\frac{u}{b_{n,\scat}}\right)^n\,
du-\Phi_\Sigma\left(\frac{N}{b_{n,\scat}}\right)\right|=0.
\end{align*}
To deal with the big $O$ term in~\eqref{eq:bigOllt}, we use that
by~\eqref{eq:lambdaton} there exists a uniform constant $C$ so that
\begin{align*}
\frac{\scat^{-2}}{(b_{n,\scat})^3}\int_{ [-\delta_0 b_{n,\scat},\delta_0
b_{n,\scat}]^2} |u|^{\eps} \left| \lambda_{\scat}
\left(\frac{u}{b_{n,\scat}}\right)^n\right| \,du\le
\frac{\scat^{-2}}{(b_{n,\scat})^{2+\eps}}\int_{ [-\delta_0 b_{n,\scat},\delta_0
b_{n,\scat}]^2} |u|^{\eps} e^{-C|u|^2}\, du.
\end{align*}
Since $M(\scat)=o(\log n)$, we have $n\gg \exp\left( C\scat^{-2}\right)$.
Thus,
$\frac{\scat^{-2}}{(b_{n,\scat})^{2+\eps}}\ll \frac{\log
n}{(b_{n,\scat})^{2+\eps}}=o\left(\frac{1}{(b_{n,\scat})^2}\right)$ as $\scat\to
0$.
Putting these together and using~\eqref{eq:bigOllt},
\[
\lim_{\scat\to 0}\left|4\pi^2\, I(\scat,t)
-\Phi_\Sigma\left(\frac{N}{b_{n,\scat}}\right)\, \int_{\cM} v \,
d\mu\,\int_{\cM} w \, d\mu\,\right|=0.
\]
This together with~\eqref{eq:llt} gives that as $\scat\to 0$,
\begin{align*}
&\left| \int_{\cM} v1_{\{\kappa_{n,\scat}=N\}} w\circ T_\scat^n\, d\mu
-\frac{1}{(b_{n,\scat})^2}
\Phi_\Sigma\left(\frac{N}{b_{n,\scat}}\right)\,\int_{\cM} v \,
d\mu\,\int_{\cM} w \, d\mu\,\right|\\
&\qquad =O\left((b_{n,\scat})^2\ C_\scat\,(1-\ga_{\scat})^n  \right)=o(1),
\end{align*}
where in the last equation we used that $M(\scat)=o(\log n)$.
This concludes the proof.~\end{proofof}

It is known that the local limit theorem for $\ks$ and the billiard map
$T_{\scat}$ (with $\scat$ fixed) implies mixing for the planar Lorentz map $\hat T_{\scat}$
(again $\scat$ fixed), see~\cite{Pene18}.
In fact, sharp error rates in local limit theorems and mixing are also
known, see~\cite{Pene18} for the finite horizon case and~\cite{PT20} for the infinite horizon case.

We recall from Section~\ref{sec:intro} that the Lorentz map $\hat T_{\scat}$
defined on $\widehat\cM=\cM\times\Z^2$ is given
by
$\hat T_{\scat}(\theta,\phi,\ell) = ( T_{\scat}(\theta,\phi),
\ell+\ks(\theta,\phi))$ for $ (\theta,\phi) \in  \cM,\ \ell\in \Z^2$.
Let $\hat\mu=\mu\times \mbox{Leb}_{\Z^2}$, where $\mbox{Leb}_{\Z^2}$ is the counting measure on $\Z^2$.
An immediate consequence of Theorem~\ref{th-llt} is

\begin{corollary}\label{cor:mix}
Assume the assumptions and notation of  Theorem~\ref{th-llt}.
Let $\scat\to 0$ and simultaneously $n\to\infty$ in such a way that
$M(\scat)=o(\log n)$. Then
\[
\lim_{\scat\to 0}\left|(b_{n,\scat})^2\int_{\widehat\cM}v\, w\circ \hat
T_{\scat}\, d\hat\mu-\int_{\widehat\cM} v \, d\hat\mu\,\int_{\widehat\cM} w \,
d\hat\mu\,\right|=0.
\]
\end{corollary}
\begin{remark}
The class of functions in Corollary~\ref{cor:mix} is rather restrictive as
the functions $v,w$ are supported on the cell $\cM$.
Given the work~\cite{Pene18} (see also~\cite[Section 6]{PT20}), it is very
plausible that the present mixing result can be generalized to a suitable
class of dynamically H{\"o}lder functions supported on the whole of
$\widehat\cM$.
Since the involved argument is rather delicate and not a main concern of the
present work, we omit this.
\end{remark}

\appendix

\section{Estimates on corridors}\label{sec:corridors}

\subsection{Estimating $\P(\ks = \xi'+N\xi)$}\label{sec:Pkappa}
Given a corridor associated to $\xi$, there a neighborhood $U_0$ of $x_0 =
x_0(\xi)$ in
$\partial O_0 \times [-\frac{\pi}{2},\frac{\pi}{2}]$ of initial conditions
$x$ such that the next
collision occurs at a scatterer on the opposite side of the corridor.
For this situation, \cite{SV07} use the coordinates
$(\alpha, z)$, where $\alpha$ is the angle the trajectory of some $x \in
\partial O_0$
makes with the tangent line
at $x_0$, and the intersection point is $y = x_0+z \xi$,
see Figure~\ref{fig:corridor1}.

\begin{figure}[ht]
\begin{center}
\begin{tikzpicture}[scale=1.7]
\node[circle, draw=black] at (1,1.55) {$O_0$};
\node[circle, draw=black] at (4,1.55) {$O_\xi$};
\node[circle, draw=black] at (7,1.5) {$O_{2\xi}$};
\node[circle, draw=black] at (8.35,2.85) {$O_{\ks}$};
\draw[-] (0,1.85) -- (9,1.85);
\draw[-] (0,2.51) -- (9,2.51);
\draw[-] (1,1.5) -- (1.6,2);
\draw[-] (1,1.5) -- (1,1.95);
\draw[-, draw=blue] (1.25,1.7) -- (8.1,2.6);
\node at (3.4,1.91) {\small $\alpha$};
\node at (1.5,1.8) {\small $\phi$};
\node at (1.09,1.66) {\small $\theta$};
\node at (1,1.95) {\small $x_0$};
\node at (1.35,1.65) {\small $x$};
\node at (1.75,2.2) {\small $\overbrace{\qquad \, \qquad \qquad}^{z\xi}$};
\end{tikzpicture}
\caption{A corridor and coordinates $(\alpha,\theta)$.}
\label{fig:corridor1}
\end{center}
\end{figure}
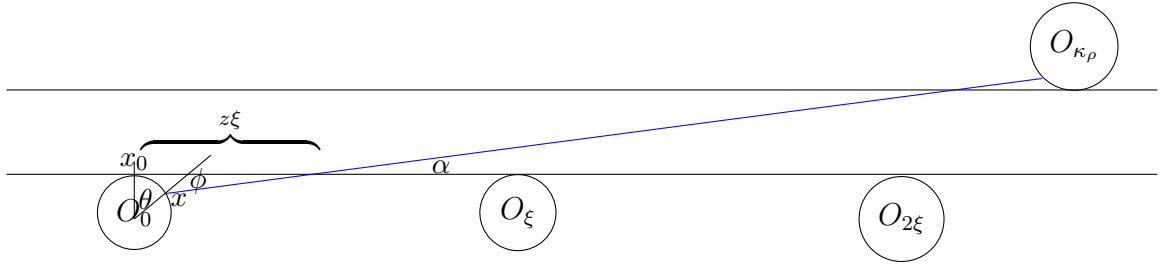

\begin{lemma}\label{lem:volume}
In coordinates $(z,\alpha)$ the volume form in a neighborhood of $x_0 =
x_0(\xi)$ is
$$
\frac{|\xi|}{4\pi \scat} \sin \alpha \, d\alpha \, dz = \frac{1}{4\pi} \cos
\phi  \, d\theta\, d\phi.
$$
\end{lemma}

\begin{proof}
The part $\sin \alpha \, d\alpha \, dz$ can be understood because the
Liouville measure of the billiard flow
projects to a form $\cos \varphi \, d\varphi \, dr$ for any transversal
section parametrised by arc-length $r$ and
with $\varphi$ the angle of the trajectory to the normal vector at the
collision point.
When this section is the line $y = x_0 + x \xi$, we have $\alpha =
\frac{\pi}{2}-\varphi$,
so $\cos \varphi = \sin \alpha$. But to get the correct normalizing constant,
we give a more extensive argument.
From Figure~\ref{fig:corridor1} we have
\begin{equation}\label{eq:beta}
\frac{\pi}{2} = \theta + \alpha + \phi, \qquad \tan \alpha = \frac{\scat
(1-\cos \theta)}{z|\xi| - \scat \sin \theta}.
\end{equation}
After making $\alpha$ and $z$ subject of these equations, we see that the
change of coordinates involved is
$$
(\alpha, z) = F(\theta,\phi) = \left( \frac{\pi}{2} - \theta - \phi,
\frac{\scat}{|\xi|} \left( \frac{1-\cos \theta}{\tan (\frac{\pi}{2}-
\theta-\phi)} + \sin \theta \right)\right).
$$
The Jacobian determinant is
$$
|\det(dF)| = \left| \det \begin{pmatrix} -1 & -1 \\
            \frac{\partial F_2}{\partial \theta} & \frac{\partial
            F_2}{\partial \phi}
           \end{pmatrix} \right|
=   \left| \frac{\partial F_2}{\partial \theta} - \frac{\partial
F_2}{\partial \phi} \right|
= \frac{\scat}{|\xi|} \left( \frac{\cos
\theta}{\tan(\frac{\pi}{2}-\theta-\phi)} + \cos\theta\right).
$$
Thus, using \eqref{eq:beta} and some trigonometric formulas,
\begin{eqnarray*}
\frac{|\xi|}{4\pi \scat} \sin\alpha \, d\alpha \, dz &=& \frac{|\xi| \sin
\alpha}{4\pi\scat}
\frac{\scat}{|\xi|} \left( \frac{\sin \theta}{\tan(\frac{\pi}{2} - \theta
-\phi)} + \cos\theta\right) \, d\theta \, d\phi \\
&=&  \frac{1}{4\pi} (\cos \alpha \sin \theta + \sin \alpha \cos\theta) \,
d\theta\, d\phi \\
&=&  \frac{1}{4\pi} \sin(\alpha + \theta) \, d\theta\, d\phi  =
\frac{1}{4\pi}\cos(\phi) \,  d\theta\, d\phi,
\end{eqnarray*}
as claimed.
\end{proof}

The following is \cite[Proposition 6]{SV07} in more detail:

\begin{lemma}\label{lem:p6}
 Suppose that the scatterers have radius  $\scat > 0$ and the width of the
 corridor given by $\xi$ is $d_\scat(\xi)$.
 Then
 $$
 \mu(\{ x \in \partial O_0 \times [-\frac{\pi}{2}, \frac{\pi}{2}]: \ks(x)
 = N |\xi| + \xi' \})
 = \frac{1}{4\pi N |\xi| \scat} \min\{ 4\scat^2 , d_{\scat}(\xi)^2 N^{-2}\}
 (1+\cO(N^{-1})),
 $$
 where $\xi'$ as in Remark~\ref{rem:width} is the integer vector on the
 boundary of the corridor opposite to the $\xi$-boundary.
\end{lemma}

\begin{figure}[ht]
\begin{center}
\begin{tikzpicture}[scale=1]
\draw[-] (0,1.5) -- (11,1.5);
\draw[-] (0,3.5) -- (11,3.5);
\draw (2,1) circle (0.5);
\draw (4,1) circle (0.5);
\draw[draw=red] (2,1.3) circle (0.2);
\draw[draw=red] (10,3.7) circle (0.2);
\draw (8,1) circle (0.5);
\draw (3,4) circle (0.5);
\node at (3,4) {\small $O_{\xi'}$};
\draw (8,4) circle (0.5);
\draw (10,4) circle (0.5);
\node at (2,0.8) {\small $O_0$};
\node at (4,1) {\small $O_\xi$};
\node at (8,1) {\small $O_{N \xi}$};
\node at (10,4.2) {\small $O_{\ks}$};
\node at (8,4) {\small $O_{\ks-\xi}$};
\draw[-, draw=blue] (1,1.2) -- (9.5,3.95); \draw[-, draw=blue] (2.5,1.05) --
(11.5,3.95);
\draw[-, draw=red] (1,1.2) -- (11.5,4.35); \draw[-, draw=red] (1,0.8) --
(11.5,3.95);
\node at (2,1.7) {\small $z_0\xi$}; \node at (4,1.7) {\small $z_1\xi$}; \node
at (3.3,1.7) {\small $z'_1\xi$};
\draw[->] (0.5,1.7) -- (0.5,3.4); \draw[->] (0.5,3.3) -- (0.5,1.6); \node at
(1.1,2.5) {\small $d_{\scat}(\xi)$};
\end{tikzpicture}
\caption{$[z_0,z_1]$ given by two tangent lines for
$2\scat > \frac{d_{\scat}(\xi)}{N}$ (blue) or $2\scat <
\frac{d_{\scat}(\xi)}{N}$ (red).}
\label{fig:corridor2}
\end{center}
\end{figure}
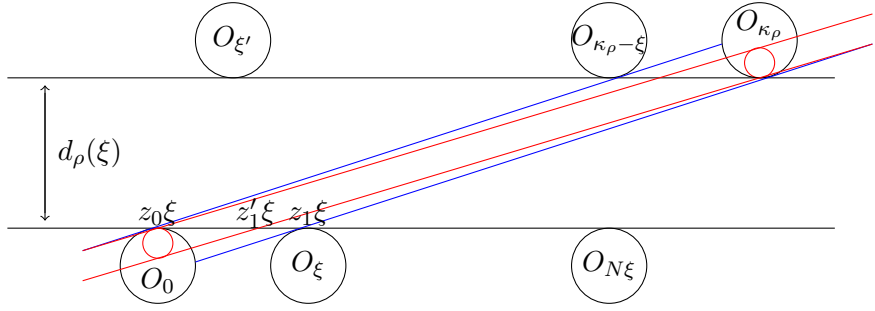

\begin{proof}
We take the region in $(z,\alpha)$-coordinates where $\ks = N \xi +
\xi'$.
In the $z$-direction this is an interval $[z_0, z_1]$, where for $z = z_0$,
there is only one line connecting
$O_0$ and $O_{\ks}$, namely the common tangent line of $O_0$ and
$O_{\ks-\xi}$.
For $z = z_1$ there is also is only one line,
namely the common tangent line of $O_\xi$ and $O_{\ks}$, see
Figure~\ref{fig:corridor2}.
These two lines are obtained from each other by translation over one unit
$\xi$, so $z_1\xi - z_0\xi = |\xi|$.
However, if $\scat$ is small compared to $N$, these two tangent lines are the
common tangent lines
at the upper sides of $O_0$ and $O_{\ks}$ and at the lower sides of $O_0$
and $O_{\ks}$.
In this case
\begin{equation}\label{eq:z0z1}
 |z_1\xi-z_0\xi| = \frac{2\scat}{\sin \alpha} = \frac{2 \scat
 (N|\xi|+|\xi'|)}{d_{\scat}(\xi)+2\scat}
+ \cO\left(\frac{\scat}{d_{\scat}(\xi)+2\scat}\right).
\end{equation}
This also shows that the transition between the two cases is when $2\scat =
\frac{d_\scat(\xi)}{N}$.

For each $z \in [z_0, z_1]$, the range of possible values of $\alpha$ is
again bounded by the $\alpha$'s obtained
at the tangent lines to $O_{\ks-\xi}$ and $O_{\ks}$.
Therefore, see Figure~\ref{fig:corridor3},
$$
\alpha \in [\alpha_0(z),\alpha_1(z)] := \left[
\arctan\left(\frac{d_\scat(\xi)}{N|\xi|+|\xi'|-z}\right) \ , \
\arctan\left( \frac{d_\scat(\xi)}{N|\xi| - |\xi| + |\xi'| - z}\right)
\right].
$$
Since $|\xi'| \leq |\xi|$ (see Remark~\ref{rem:width}) and $z \leq |\xi|$ as
well, each $\alpha$ in this interval satisfies
$\alpha = \frac{d_\scat(\xi)}{N\, |\xi|}(1 + \cO(N^{-1}))$ and
\begin{equation}\label{eq:alphadiff1}
 \alpha_1(z) - \alpha_0(z) = \frac{d_\scat(\xi)}{N^2|\xi|} (1+ \cO(N^{-1})).
\end{equation}
Integrating the density given in Lemma~\ref{lem:volume} for the case
$2\scat \geq \frac{d_\scat(\xi)}{N}$ (so $|z_1-z_0| = |\xi|$) and using
$|z_1-z_0|=|\xi|$ and the approximation
$\cos \alpha_0 - \cos \alpha_1 \sim \frac12 (\alpha_1+\alpha_0)
(\alpha_1-\alpha_0)$ gives:
\begin{eqnarray*}
 \int_{z_0}^{z_1} \int_{\alpha_0(z)}^{\alpha_1(z)} \frac{|\xi|}{4\pi \scat}
 \sin \alpha  \, d\alpha \, dz
 &=& \frac{|\xi|}{4\pi \scat}  \int_{z_0}^{z_1}  \cos(\alpha_0(z)) -
 \cos(\alpha_1(z)) \, dz \\
 &=& \frac{|\xi|}{4\pi \scat} \frac{d_{\scat}(\xi) }{ N |\xi| }
 \frac{d_{\scat}(\xi) }{ N^2 |\xi| } ( 1+\cO(N^{-1}) ) \\
 &=& \frac{1}{4\pi N \scat} \, \frac{d_\scat(\xi)^2}{|\xi|N^2} \left( 1 +
 \cO(N^{-1}) \right).
\end{eqnarray*}
Now for the case $2\scat < \frac{d_\scat(\xi)}{N}$, see
Figure~\ref{fig:corridor3} with small version of $O_{\ks}$, we have
$$
\alpha \in [\alpha_0(z), \alpha_1(z)] :=
\left[ \arctan \left(
\frac{d_\scat(\xi)}{N|\xi|+Q-z-2\scat\sin\alpha}\right)
\ , \ \arctan\left(\frac{d_\scat(\xi)+2\scat \cos \alpha_1(z)}{N|\xi|+Q -
z-2\scat\sin\alpha}\right) \right],
$$
so still $\alpha = \frac{d_\scat(\xi)}{N\, |\xi|} + \cO(N^{-2})$ and
$\alpha_1(z) - \alpha_0(z) = \frac{2\scat}{N|\xi|} (1+ \cO(N^{-1})$.

\begin{figure}[ht]
\begin{center}
\begin{tikzpicture}[scale=1]
\draw[-] (0,1.5) -- (11,1.5);
\draw[-] (0,3.5) -- (11,3.5);
\draw (2,1) circle (0.5);
\draw (4,1) circle (0.5);
\draw (8,1) circle (0.5);
\draw (8,4) circle (0.5);
\draw (10,4) circle (0.5);
\draw[draw=red] (10,3.7) circle (0.2);
\node at (2,1) {\small $O_0$};
\node at (4,1) {\small $O_\xi$};
\node at (8,1) {\small $O_{N \xi}$};
\node at (10,4.2) {\small $O_{\ks}$};
\node at (8,4) {\small $O_{\ks-\xi}$};
\draw[-, draw=blue] (3,1.5) -- (9.5,4.05); \draw[-, draw=blue] (3,1.5) --
(11.5,3.89);
\draw[-, draw=red] (3,1.5) -- (11,4.24); \draw[-, draw=red] (3,1.5) --
(11.5,3.95);
\node at (3,1.2) {\small $z\xi$}; \node at (3,1.48) {\small $\bullet$};
\node at (4.25,1.65) {\tiny $\alpha_1(z)$}; \node at (7,3.25) {\tiny
$\alpha_0(z)$};
\draw[->] (0.5,1.7) -- (0.5,3.4); \draw[->] (0.5,3.3) -- (0.5,1.6); \node at
(1.1,2.5) {\small $d_{\scat}(\xi)$};
\end{tikzpicture}
\caption{The parameter interval $[\alpha_0(z),\alpha_1(z)]$ given by angles
between two tangent lines.}
\label{fig:corridor3}
\end{center}
\end{figure}
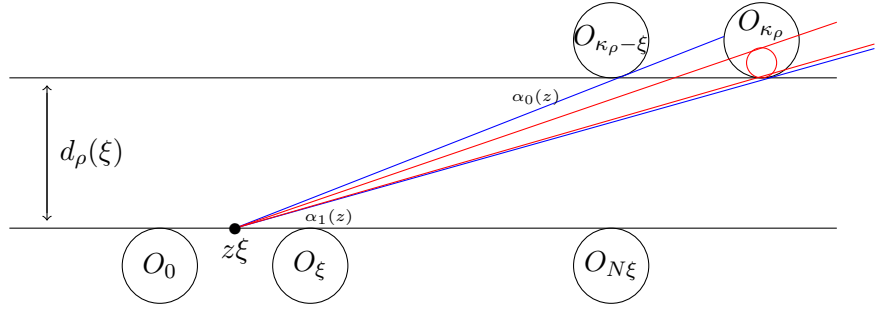

Integrating as before gives, using \eqref{eq:z0z1} and the fact that
$d_{\scat}(\xi)+2\scat =|\xi|^{-1}$
from Lemma~\ref{lem:width}:
\begin{eqnarray*}
 \int_{z_0}^{z_1} \int_{\alpha_0(z)}^{\alpha_1(z)} \frac{|\xi|}{4\pi \scat}
 \sin \alpha  \, d\alpha \, dz
 &=& \frac{|\xi|}{4\pi \scat}  \int_{z_0}^{z_1}  \cos(\alpha_0(z)) -
 \cos(\alpha_1(z)) \, dz \\
&=& \frac{|\xi|}{4\pi \scat}\, \frac{2 \scat N }{ d_{\scat}(\xi)+2\scat }\,
\frac{ d_{\scat}(\xi) }{ N |\xi| }
\frac{ 2\scat }{ N |\xi| } ( 1+\cO(N^{-1}) ) \\
 &=& \frac{ 4\scat^2 }{4\pi |\xi| N \scat} \,  \left( 1 + \cO(N^{-1})
 \right)
\end{eqnarray*}
as required.
\end{proof}

\subsection{Corridors sums}\label{sec:corsum}
Let $\varphi$ be Euler's totient function, i.e., the number of integers $1
\leq q \leq p$ coprime with $p$.
The following lemma is classical number theory, but we couldn't locate a
proof of the full statement.

\begin{lemma}\label{lem:totient}
For every $a > -2$, we have
$$
\sum_{n=1}^N n^a \varphi(n) = \frac{N^{a+2}}{a+2} \, \frac{1}{\zeta(2)}
(1+o(1)),
$$
where $\zeta$ is the Riemann $\zeta$-function, so $\zeta(2) =
\frac{\pi^2}{6}$.
\end{lemma}

\begin{proof}
 Let $\mu$ be the M\"obius function. A standard equality is
 $\varphi(n) = \sum_{d|n} \mu(d) \frac{n}{d}$.
 Therefore
 \begin{eqnarray*}
  \sum_{n=1}^N n^a \varphi(n)  &=& \sum_{n=1}^N \sum_{d|n} n^a \mu(d)
  \frac{n}{d}
  = \sum_{n=1}^N \sum_{d|n} d^a \mu(d) \left(\frac{n}{d}\right)^{a+1} \\
  &=& \sum_{d=1}^N \sum_{m=1}^{\frac{N}{d}} d^a \mu(d) m^{a+1} =
  \sum_{d=1}^N  d^a \mu(d) \frac{1}{a+2}
  \left(\frac{N}{d}\right)^{a+2} (1+o(1)) \\
  &=& \frac{N^{a+2}}{a+2} \, \sum_{d=1}^N \frac{\mu(d)}{d^2} (1+o(1)) =
  \frac{N^{a+2}}{a+2} \, \frac{1}{\zeta(2)} (1+o(1)),
 \end{eqnarray*}
where we used the Dirichlet series identity $\sum_{d=1}^\infty
\frac{\mu(d)}{d^s} = \frac{1}{\zeta(s)}$
for $s=2$.

As an aside, there are asymptotic formulas for $s > 2$
\begin{equation}\label{eq:Riemann-zeta}
\sum_{p \geq 1} \frac{\varphi(p)}{p^s} = \frac{\zeta(s-1)}{\zeta(s)}
\quad \text{ and } \quad
\sum_{p=1}^N \frac{\varphi(p)}{p} = \frac{N}{\zeta(2)} + \cO((\log
N)^{\frac23} (\log \log N)^{\frac43}),
\end{equation}
see \cite[Theorem 288]{HW}.
\end{proof}

In the course of this paper we denote, for a fixed value of $\scat$,  the set of corridors that are  ``visible'' from the origin by $\cXi$.
As described in Lemma~\ref{lem:width}, these can be characterized by pairs $(\xi,\xi')\in \Z^2 \times \Z^2$ where $\xi = (p,q)$, $\gcd(p,q) = 1$ and $|\xi| \leq (2\scat)^{-1}$, while $\xi'$ may denote either the first or the second convergent preceding $\xi$ in the continued fraction expansion of $p/q$, see Remark~\ref{rem:width}.
Sums of the type in the following lemma are used throughout the paper.

\begin{lemma}\label{lem:corridor_sum}
We have
$$
\sum_{(\xi,\xi') \in \cXi} |\xi|^a
 \begin{cases}
  \sim \frac{2}{a+2} \frac{2\pi}{\zeta(2)}(2\scat)^{-(a+2)} & \text{ if } a >
  -2;\\
  \asymp |\log\scat| & \text{ if } a = -2;\\
  \leq -\frac{4\pi}{a+2}  & \text{ if } a < -2.
 \end{cases}
$$
\end{lemma}

\begin{proof}
Using the two coordinate axes and their bisectrices, we divide the plane into eight sectors
and for each sector, we sum the scatterers in $\cS$.
Circular sections of radius $R$ have asymptotically $\frac{\pi}{4}$ as many
points as triangular sectors with base $R$.
Also, every corridor direction in this sector comes with two corridors $(\xi,\xi')$ and $(\xi,\xi'')$.
By Lemma~\ref{lem:totient}, their sum is, for $a > -2$,
\begin{eqnarray*}
 \sum_{(\xi,\xi') \in \cXi} |\xi|^a
 &\sim& \frac{16\pi}{4} \sum_{0 \leq q \leq p \leq (2\scat)^{-1}} |\xi|^a
 = 4\pi \sum_{1 \leq p \leq (2\scat)^{-1}} \phi(p) p^a
 \sim \frac{4\pi}{2+a} \frac{1}{\zeta(2)} (2\scat)^{-(2+a)}.
\end{eqnarray*}
If $a = -2$, then a similar computation gives $\asymp |\log \scat|$,
and for $a < -2$, the series is summable:
$4\pi \sum_{1 \leq p \leq (2\scat)^{-1}} \phi(p) p^a
\leq 4\pi\int_1^\infty x^a \, dx = -\frac{4\pi}{2+a}$.
\end{proof}

\begin{lemma}\label{lem:kappa_-norm}
For $p \in [1,2)$, the $p$-norm of the displacement function satisfies
$$
\|\ks\|_{L^p} \ll (p(2-p))^{-1/p} \ \scat^{-1}.
$$
\end{lemma}

\begin{proof}
Take $p \in [1,2)$. We estimate over all $\xi$-corridors similarly as in Lemma~\ref{lem:corridor_sum}:
 \begin{eqnarray*}
  \int |\ks|^p \, d\mu &\ll& 2\sum_{|\xi| \leq (2\scat)^{-1} } \sum_{N \geq
  1}
  |\xi|^p N^p \frac{1}{4\pi |\xi| N \scat} \min\{ 4\scat^2, d_{\scat}(\xi)^2
  N^{-2}\} \\
 &\leq& \frac{1}{2\pi \scat} \sum_{|\xi| \leq (2\scat)^{-1}} |\xi|^{p-1}
 \left( \sum_{N=1}^{\lfloor d_{\scat}(\xi)/(2\scat) \rfloor } 4\scat^2
 N^{p-1} + \sum_{N=\lfloor d_{\scat}(\xi)/(2\scat)\rfloor}^\infty
 d_{\scat}(\xi)^2  N^{p-3}\right) \\
  &\leq& \frac{1}{2\pi \scat}  \left( \frac1p (2\scat)^{2-p} + \frac{1}{2-p}
  (2\scat)^{2-p}\right)  \sum_{|\xi| \leq  (2\scat)^{-1}}  |\xi|^{-1} \\
  &\sim&  \frac{2}{\zeta(2)}  \left( \frac1p + \frac{1}{2-p} \right)\, (
  2\scat)^{-p}.
\end{eqnarray*}
Taking the $p$-th root gives the result.
\end{proof}

\begin{lemma}\label{lem:sector_sum}
 Let $W \in \cW^s$ be a stable leaf, and let $\cXi_W$
 stand for all lattice points $\xi = (p,q) \in \cXi$ that can be reached from $O_0$ with coordinates in $W$.
 Then for every $a \in (\frac12,1)$,
 $$
 \sum_{(\xi,\xi') \in \cXi_W} |\xi|^{-a}  \ll \scat^{a-2} |W|  + \scat^{a-1} \log(1/\scat) + \scat^{a-1}|W|^{-1}.
 $$
\end{lemma}

\begin{proof}
There is an arc $\tilde W \in \S^1$ of length $|\tilde W| \ll |W|$
such that every lattice point that can be reached from $O_0$
with coordinates in $W$ has its polar angle in $\tilde W$.
Due to the symmetries in the $\Z^2$, it suffices to study
$\tilde W \subset [0, \pi/2]$, so the lattice point $\xi = (p,q)$
in this sector satisfy $0 \leq q \leq p$ and $\tan(\tilde W) \subset [0,1]$.
In fact, we will start by assuming that
$\tan(\tilde W) \in [\frac{1}{10}, \frac{9}{10}]$.

 Because $p^2+q^2 \geq 2pq$ for all $(p,q) = \xi$, we have
 $\sum_{(\xi,\xi') \in \cXi_W} |\xi|^{-a} \ll 2^{-a/2} \sum_{(\xi,\xi') \in \cXi} \frac{1}{(pq)^{a/2}} 1_{\tilde W}(\frac{q}{p})$.
 We will apply an estimate from \cite[Theorem 2.2]{Weber}, which, in our terminology,  reduces to
 \begin{eqnarray}\label{eq:weber}
  \sum_{(\xi,\xi') \in \cXi} \frac{1}{(pq)^{a/2}}\, \psi\left(\frac{p}{q}\right)
  &=& C_a\, \scat^{a-2} \int \psi(x) \, dx + O(\scat^{1-a} \log(1/\scat)) \nonumber \\
  && +\ O\left(\sum_{\ell \neq 0} c_\psi(\ell)
  \sum_{\stackrel{d \leq (2\scat)^{-1}}{d | \ell}} d^{1-a} \sum_{k \leq (2\scat d)^{-1}} \frac{\mu(k)}{k^a} \right),
 \end{eqnarray}
 where $C_a$ is a constant depending only on $a$, and
 $c_{\psi}(\ell)$ is the $\ell$-th
 Fourier coefficient of $x \mapsto \psi(x) x^{-a}$.

 If $\psi = 1_{\tilde W}$, then these Fourier coefficients are not summable,
 so we first smoothen $1_{\tilde W}$ to a function $\psi$
 with $\supp(\psi)$ concentric to $\tilde W$
 and $|\supp(\psi)| = |\tilde W|= 3|W|$.
 On $\tilde W$ itself, $\psi \equiv 1$ and on the two interval components $\psi$ is a translated copy of the function $f_W:[-\frac{|W|}{2}, \frac{|W|}{2}] \to \R$
defined by
$$
f_W(x) = \frac12-\frac{1}{2\pi} \sin \frac{2\pi x}{|W|} + \frac{x}{|W|}.
$$
Then $\int \psi \, dx = 2|W|$ and integrating by parts twice gives an estimate of the Fourier coefficients of $x \mapsto \psi(x) x^{-a}$.
$$
|c_{\psi}(\ell)| \ll \left| \int \frac{(\psi(x)x^{-a})''}{(2\pi \ell)^2} \, e^{2\pi i \ell x} \, dx \right| \ll \frac{1}{|W|\ell^2}
$$
because $\supp(\psi)$ is bounded away from $\{ 0,1\}$ (so $x^{-a}$
doesn't blow up) and $(\psi(x) x^{-a})'' = 0$ outside $\supp(\psi)$.

The Dirichlet series of the M\"obius function
can be estimated as
$\left|\sum_{k=1}^{\lfloor1/(2\scat d)\rfloor} \mu(k) k^{-a}\right| \leq
(2\scat d)^{1-a}$.
We use this
and the fact that $\ell \in \N$ has $O(\ell^{1/2})$ divisors
to estimate the last big $O$-term in~\eqref{eq:weber}.
\begin{eqnarray*}
\sum_{\ell \in \N} |c_\psi(\ell)| \sum_{\stackrel{d \leq (2\scat)^{-1}}{d | \ell}} d^{1-a}
\sum_{k \leq (2\scat d)^{-1}} \frac{\mu(k)}{k^a}
 &\ll&
\frac{(2\scat)^{a-1}}{|W|} \sum_{\ell \in \N} |c_\psi(\ell)| \sum_{\stackrel{d \leq (2\scat)^{-1}}{d | \ell}} 1
 \\
 &\ll& \frac{(2\scat)^{1-a}}{|W|} \sum_{\ell \neq 0}
|\ell|^{-\frac32} \leq \frac{(2\scat)^{a-1}}{|W|}.
\end{eqnarray*}
Hence \eqref{eq:weber} becomes
$$
 \sum_{(\xi,\xi') \in \cXi} \frac{1}{(pq)^{a/2}} 1_{\tilde W}(\frac{q}{p})
 \leq \sum_{(\xi,\xi') \in \cXi} \frac{1}{(pq)^{a/2}}\, \psi(\frac{q}{p})
\ll \scat^{a-2} |W| + \scat^{a-1} \log(1/\scat) +
 \scat^{a-1} |W|^{-1},
$$
as required.

It remains to consider the cases that $\tan(\tilde W) \not\subset
[\frac{1}{10}, \frac{9}{10}]$.
Suppose instead that $\tan(\tilde W) \subset
(0, \frac{1}{10}]$ (we ignore $\xi = (0,1)$ and $\xi = (1,0)$).
In this case, we give an injection between the lattice points in the $\tilde W$-sector with coprime coordinates to the set of
lattice points (with coprime coordinates and comparable norm) in a sector of comparable width, but near polar angle $\frac12$.
Indeed, set $\Q_{cp} = \{ q/p : 0 \neq p, q \in \Z, \gcd(p,q) = 1\} \cup \{0\}$ and
$\Z_{cp} := \{ (p, q) \in \Z^2 : \gcd(p,q) = 1\}$,
and define the Calkin-Wilf map $f:\Q_{cp} \to \Q_{cp}$
as well as $g:\Z_{cp} \to \Z_{cp}$ by
$$
f: x \mapsto \frac{1}{1-x-2\lfloor x \rfloor},\qquad \qquad
g: (p,q) \mapsto ( p-q+2p \lfloor q/p \rfloor \ , \ p).
$$
The $f$-orbit of $0$ enumerates all non-negative lowest-term rationals, see \cite{CW},
and $g$ is the same map expressed on the collection of lattice points.
Since $f^2( (0,\frac{1}{10}] ) \subset (\frac12,\frac{10}{21}]$
and $|g(\xi)| \leq 4|\xi|$, the second iterate $g^2$ provides the required injection. In case
$\tan(\tilde W) \subset [\frac{9}{10},1)$ we use $g^3$.
\end{proof}

\section{Distortion properties}\label{sec:distortion}

Throughout, a uniform constant is a constant that is independent of $\scat$.

Let us recall some terminology and notations from \cite[Chapter 4]{CM}. Unstable curves
generate dispersing wavefronts,
which are evolved by the free flight, and then leave traces of unstable
curves on the scatterer at the next collision.
For wavefronts it is convenient to use the Jacobi coordinates
$(d\xi,d\omega)$, and an important quantity\footnote{Usually called
$\mathcal B$ in billiard literature such as \cite{CM}, but we write $\Omega$
to avoid confusion with Banach spaces $\cB$.}
$\Omega=\frac{d\omega}{d\xi}$, the curvature of the wavefront.
Let $\Omega^-$ and $\Omega^+$ denote
its value immediately before and after a particular collision, respectively.

On the scatterer, the traditional coordinates are $(r,\phi)$ yet, we prefer
to use the $\scat$-independent
$(\theta,\phi)$ and take advantage of
\[
\frac{d}{d\theta}=(2\pi \scat) \frac{d}{dr}.
\]
First we relate $\Omega^-$ to the slope of the unstable curve:
$\frac{1}{2\pi} \frac{d\phi}{d\theta} =\scat \Omega^- \cos\phi+1$.
Differentiating with respect to $\theta$ gives
\begin{equation}\label{eq:SecondDer}
\frac{1}{2\pi} \frac{d^2\phi}{d\theta^2}=\frac{d\Omega^-}{d\theta} \scat
\cos\phi -\scat \Omega^- \sin\phi \frac{d\phi}{d\theta}.
\end{equation}

\begin{lemma}\label{lem:CurvBound}
There exists a uniform constant $C>0$ such that for any $C^2$ smooth unstable
curve $W$ there exists $n_W$ such that for $n\ge n_W$ on all components of
$T_{\scat}^nW$ we have
\begin{equation}\label{eq:CurvBound}
\left|\frac{d^2\phi}{d\theta^2}\right| \le C\scat.
\end{equation}
\end{lemma}

Thus we may restrict to the class of \textit{regular} unstable curves for which \eqref{eq:CurvBound} holds. Also, this shows that as $\scat\to 0$, the
unstable curves limit in a $C^2$ sense to straight lines of slope $2\pi$.
\\[4mm]
\begin{proof}
The properties of the free flight are not effected by shrinking the
scatterers or using the $\theta$-coordinate. Thus
\[
0\le \Omega^-\le (\tau_{min})^{-1}
\]
and, by \eqref{eq:SecondDer}, it is enough to show
\[
\left|\frac{d\Omega^-}{d\theta}\right| \le C
\]
to prove the lemma.
Now $\frac{d\Omega^-}{d\theta}=(2\pi \scat) \frac{d\Omega^-}{dr}$, and the
evolution of $\frac{d\Omega^-}{dr}$ is discussed in \cite[section 4.6]{CM}.
Following the notation there, introduce
\[
\cE_1=\frac{d\Omega}{d\xi};\qquad F_1=\frac{\cE_1}{\Omega^3}
\]
and use superscripts $-$ and $+$ to denote pre- and post-collision values of these quantities, respectively.
\cite[Formula (4.37)]{CM} states
\[
-F_1^+=\left(\frac{\Omega^-}{\Omega^+}\right)^3 F_1^- +H_1,
\]
where
\[
H_1=\frac{6\scat^{-2}\sin\phi +6\scat^{-1} \Omega^- \cos\phi\sin\phi
}{(2\scat^{-1}+\Omega^-\cos\phi)^3}
\]
and by the analysis of \cite[page 81]{CM}: \begin{itemize}
\item $F_1$ remains constant between collisions
\item there exists a uniform constant $\Theta<1$ such that
    $\frac{\Omega^-}{\Omega^+}\le \Theta$,
\item there exists a uniform constant $C_1>0$ such that $|H_1|\le C_1$. This
    remains valid for shrinking $\scat$ as the denominator scales with
    $\scat^{-3}$ while the
numerator scales with $\scat^{-2}$.
\end{itemize}
Hence it follows that $|F_1(n+1)|\le \Theta^3 |F_1(n)| +C$, where $F_1(n)$ is
the value of $F_1$ between the $n$-th and the $(n+1)$st
collision. This implies that there exists $C_2>0$ and $n_W$ (depending on the
curve $W$) such that for any $n\ge n_W$ we have $|F_1(n)|\le C_2$.

Now $|\cE_1^-|=|F_1^-|\cdot (\Omega^-)^3\le C_3$ for some uniform $C_3>0$, and finally \cite[Formula (4.24)]{CM} states
\[
\frac{d\Omega^-}{dr}=\cE_1^- \cos\phi - (\Omega^-)^2\sin\phi,
\]
which thus implies that $\left|\frac{d\Omega^-}{dr}\right|\le C_4$ for some
uniform constant $C_4>0$.
This bound completes the proof of the lemma.
\end{proof}

It follows that regular unstable curves can be parametrised by the coordinate
$\theta$, and for any smooth function $f:W\to\bR$, $\frac{df}{d\theta}\asymp
\frac{df}{dx}$,
where $x$ is (Euclidean) arc-length along the curve --
$dx^2=d\theta^2+d\phi^2$ (not to be confused with the arc-length $r$ along
the scatterer).

Let us also recall that an unstable curve is homogeneous if it is regular and
contained in one of the homogeneity
strips $\bH_k=\{(\theta,\phi) : \frac{\pi}{2}-k^{-r_0} < \phi <
\frac{\pi}{2}-(k+1)^{-r_0} \}$.
For such curves, analogous to \cite[Formula (5.13)]{CM}, we have
\begin{equation}\label{eq:DistortCos}
|W|\le C \cos^{\frac{r_0+1}{r_0}} \phi
\end{equation}
for some uniform constant $C>0$, where $\phi$ corresponds to any point of $W$. (This follows as the slope of the curve is uniformly
bounded away from $0$ and $\infty$.)

Distortion bounds are stated as follows. Let $W$ be a homogeneous unstable
curve, and assume that for some $N\ge 1$,
$W_n=T_{\scat}^{-n}W$ is a homogeneous unstable curve for $n=0,1,\dots ,N$.
For $x\in W$, let $x_n=T_{\scat}^{-n}x\in W_n$. Let $J_WT_{\scat}^{-n}(x)$ and $J_{W_n}
T_{\scat}^{-1}(x_n)$ denote the respective Jacobians.
\begin{lemma}\label{lem:Distortion}
Consider $W$ and $N$ as above and $y,z\in W$ arbitrary. There exists a
uniform constant $C_d>0$ such that
\[
|\log J_WT_{\scat}^{-N}(y)-\log J_WT_{\scat}^{-N}(z)|\le C_d|W|^{\frac{1}{r_0+1}}.
\]
\end{lemma}

\begin{proof}
The lemma relies on the inequality
\begin{equation}\label{eq:Distortion1}
\left|\frac{d}{dx_n} \log J_{W_n} T_{\scat}^{-1}(x_n) \right|\le
\frac{C}{\cos\phi_n}
\end{equation}
for some uniform $C>0$, cf.~\cite[Formula (5.8)]{CM}.

Using this formula the argument in the proof of \cite[Lemma 5.27]{CM} can be
repeated literally:
\begin{eqnarray}
 |\log J_WT_{\scat}^{-N}(y)-\log J_WT_{\scat}^{-N}(z)|&\le & \sum_{n=0}^{N-1} |\log
 J_{W_n}T_{\scat}^{-1}(y_n)-\log J_{W_n}T_{\scat}^{-1}(z_n)| \nonumber \\
  &\le& \sum_{n=0}^{N-1} |W_n| \max \left|\frac{d}{dx_n} \log J_{W_n}
  T_{\scat}^{-1}(x_n) \right|  \\
  &\le&  C \sum_{n=0}^{N-1} \frac{|W_n|}{\cos\phi_n}
  \le C \sum_{n=0}^{N-1} |W_n|^{\frac{1}{r_0+1}}\le C
   |W|^{\frac{1}{r_0+1}}, \nonumber
\end{eqnarray}
where we have used the chain rule, \eqref{eq:Distortion1},
and~\eqref{eq:DistortCos} and the uniform hyperbolicity.

It remains to prove \eqref{eq:Distortion1}. Here we essentially follow
\cite[pp. 106--107]{CM}. We have
\begin{eqnarray}
  \log J_{W_n} T_{\scat}^{-1}(x_n) &=& \log \cos\phi_n + \frac{1}{2}
  \log\left(4\pi^2\scat^2 +\left(\frac{d\phi_n}{d\theta_n}\right)^2
  \right)-\frac{1}{2} \log\left(4\pi^2\scat^2
  +\left(\frac{d\phi_{n+1}}{d\theta_{n+1}}\right)^2 \right)\nonumber \\
    && - \log\left( 2\scat^{-1} \tau_{n+1}+
    \cos\phi_{n+1}(1+\tau_{n+1}\Omega^-_{n+1}) \right).\nonumber
  \end{eqnarray}
We consider the derivatives of these terms separately. As noted above,
differentiation with respect to $\theta_n$ and $x_n$ can be interchanged. By
Lemma~\ref{lem:CurvBound}, the derivative of the second term w.r.t.\ $\theta_n$ is uniformly bounded. The same applies to the derivative of the
third term with respect to $\theta_{n+1}$, while
\[
\frac{dx_{n+1}}{dx_n}= J_{W_n} T_{\scat}^{-1}(x_n)
\]
is uniformly bounded from above. The first term gives the main contribution:
as $\cos\phi_n$ is not bounded away from $0$, the derivative of its logarithm is
\[
\left|\frac{d(\log \cos\phi_n)}{dx_n}\right|\le C \left|\frac{d(\log
\cos\phi_n)}{d\theta_n}\right|\le \frac{C}{\cos\phi_n}.
\]
The fourth term is the logarithm of the quantity
\[
2\scat^{-1} \tau_{n+1}+ \cos\phi_{n+1}(1+\tau_{n+1}\Omega^-_{n+1})
\]
which is bounded from below, but not from above. It is thus (more than) enough
to show that, when taking the derivative,
all contributions to the numerator are uniformly bounded. This holds
immediately by the previous discussion for all
the terms except $\frac{2}{\scat} \frac{d\tau_{n+1}}{dx_n}$ which requires
further investigation. Note that
\[
\tau_{n+1}=\mathop{dist}(P(x_n), P(x_{n+1}))
\]
where $P(x_n)$ and $P(x_{n+1})$ are points on the billiard table (and thus on
$\bR^2$) associated to the points $x_n\in W_n$ and $x_{n+1}\in W_{n+1}$ on
the two scatterers, respectively. In an appropriate reference frame
$P(x_n)=(\scat\cos\theta_n,\scat\sin\theta_n)$ hence the $\theta_n$-derivatives
of both coordinates are $\ll \scat$, and the same holds for the
$\theta_{n+1}$-derivatives of the coordinates of $P(x_{n+1})$. Thus
\[
\left|\frac{d\tau_{n+1}}{dx_n}\right|\le C \scat,
\]
which is sufficient for our purposes.
\end{proof}

\section{Decay of correlation for $\ks$.}\label{sec:decay}

The main result of this section is the justification of~\eqref{eq-corel2}, that is

\begin{prop}\label{prop:decaykappa}
There exist $\hat C_\scat>0$ and $\hat\vartheta_\scat<1$ such that

\begin{itemize}
\item for any $j \geq 1$ we have
\begin{align}\label{eq:prop_c1_l0}
 \Big| \int_{\cM_0} & (e^{it\ks}-1)\, (e^{it\ks}-1)\circ
 T_{\scat}^{j}\, d\mu
- \int_{\cM_0}  (e^{it\ks}-1)
\,d\mu \int_{\cM_0}
(e^{it\ks}-1)\,d\mu \Big|
 \le\hat C_\scat |t|^2 \hat\vartheta_\scat^j, \quad
\end{align}

\item furthermore, there exist $\wCs>0$ and $\wts<1$ such that for any $j, \ell\ge 1$ we have
\begin{align}
 \Big| \int_{\cM_0} & (e^{it\ks}-1)
 R_\scat^{\ell}(e^{it\ks}-1)\, (e^{it\ks}-1)\circ
 T_{\scat}^{j}\, d\mu \nonumber\\
&-\int_{\cM_0}  (e^{it\ks}-1)
R_\scat^{\ell}(e^{it\ks}-1)\,d\mu \int_{\cM_0}
(e^{it\ks}-1)\,d\mu \nonumber\\
&-\,\Big(\int_{\cM_0}  (e^{it\ks}-1)\,d\mu\Big) \int_{\cM_0}
(e^{it\ks}-1)\, (e^{it\ks}-1)\circ  T_{\scat}^{j}\,
d\mu \nonumber\\
&+\,\Big(\int_{\cM_0} (e^{it\ks}-1)\,d\mu\Big)^3\Big|\ \le\
\wCs |t|^2 \wts^{\ell+j}.
\label{eq:prop_c1_l1}
\end{align}

\end{itemize}
\end{prop}

The  $\scat$-dependence of this exponential rate gives the main source of
unknown dependence on $\scat$ in the main results of our paper. During the proof we will
point out the exact sources of unknown dependence of $\wCs>0$ and $\wts<1$ on $\scat$.

Let us make some comments on the relations of the two estimates of Proposition~\ref{prop:decaykappa}.
We will first prove \eqref{eq:prop_c1_l0} with some
 $\hat C_\scat>0$ and $\hat\vartheta_\scat<1$ that we can explicitly relate to the correlation decay rates of the map $T_{\scat}$ on H\"older functions, as expressed in
 \eqref{eq:edc_const_relate} below. Then we extend our argument to obtain \eqref{eq:prop_c1_l1} for some $\wCs>\hat C_\scat$ and $\wts\in(\hat\vartheta_\scat,1)$.
  Obtaining relations similar to \eqref{eq:edc_const_relate} for the constants $\wCs$ and $\wts$ seems quite difficult and we do not push this point.

The proof of Proposition~\ref{prop:decaykappa} consists in: a) reconsider \cite[Proposition 9.1]{ChDo09};
b) only for \eqref{eq:prop_c1_l1}, work with a version of $R_\scat$
with spectral gap in a Banach space embedded in some $L^p$ space with $p>1$.
Item a) is needed in order to obtain the bound $|t|^2$ and the decay of correlation in $j$.
Item b) is needed to obtain the joint decay in $j$ and $\ell$.
Item b) is possible because  for every $\scat>0$, \emph{there exists} a Young tower
$\Delta_\scat$ and a tower map $T_{\Delta_\scat}$ associated with the billiard map $T_\scat$; this is ensured by the construction in~\cite{Chernov99, Young98}.
We emphasize that we will not exploit any fine dependence on $\scat$ of
$T_{\Delta_\scat}$ (the mere existence is enough), which is why this part of
 our arguments can be worked on the Young tower $\Delta_\scat$.

\subsection{Standard pair argument}
\label{subsec:stpair}

In this section we reconsider \cite[Proposition 9.1]{ChDo09}.
 Let us introduce truncation levels $H,\hat H > 0$ to be fixed later and
\begin{align*}
\ks'=\ks\cdot \textbf{1}_{|\ks|\le H} \qquad\qquad
&\ks''=\ks-\ks';\\
\ks'''=\ks\cdot \textbf{1}_{|\ks|\le \hat H} \qquad\qquad
&\ks''''=\ks-\ks'''.
\end{align*}
As $|\ks| \asymp |\xi| m$ on $D_{\xi,m}$, the truncation $\ks'$
restricts $\ks$ to the cells $D_{\xi,m}$ with $m\le H |\xi|^{-1}$.

The result we will use in the proof of Proposition~\ref{prop:decaykappa}
below is

\begin{lemma}\label{lemma:stpair} For any $c_0>2$ we have
\begin{itemize}
\item[(i)']$\int |\ks'| \cdot |\ks''''| \circ T_{\scat}^j \, d\mu   \le  CH^2 \hat H^{-1} \scat^{-3}$,
\item[(ii)'] $\int |\ks'' | \cdot |\ks| \circ T_{\scat}^j \, d\mu
\le C |\log \scat|\cdot \left(H^{-\frac12 +\frac{1}{2r_0}} \, \log H \, \scat^{-3-\nu}+ H^{2-c_0} \scat^{-2-c_0}\right)$.
\end{itemize}
Furthermore, for any $q\in \left(1,\frac87-\frac{6}{7(7r_0-1)}\right)$ and
$c \in (\frac{q+1}{2-q}\, , \, \frac{1-1/r_0}{2q-2}-1)$,
\begin{itemize}
\item[(i)]$\int |\ks'|^q \cdot |\ks''''|^q \circ T_{\scat}^j \, d\mu   \le  CH^{q+1} \hat H^{q-2} \scat^{-3}$,
\item[(ii)] $\int |\ks'' |^q \cdot |\ks|^q \circ T_{\scat}^j \, d\mu
\le C \left(H^{-\frac32+q+c(q-1)+\frac{1}{2r_0}} \scat^{c(q-1)-q-2-\nu} + H^{c(q-2)+q+1} \scat^{-1-q-c(2-q)}\right)$.
\end{itemize}
\end{lemma}

\begin{remark}\label{rmk:stpair}
Let $q(r_0)=\frac87-\frac{6}{7(7r_0-1)}$, the upper bound on $q$ for $r_0$ fixed.
Furthermore, let $c_1(q)=\frac{q+1}{2-q}$ and $c_2(q)=\frac{1-1/r_0}{2q-2}-1$,
the lower and upper bounds on $c$ for $q$ fixed.
Note that $c_1(q)$ is increasing in $q$, while $c_2(q)$ is decreasing in $q$, and $c_1(q(r_0))=c_2(q(r_0))$. Also $c_1(1)=2$ and
$c_2(1)=\infty$, which is in accordance with the conditions on $c_0$.
Note also that:
\begin{itemize}
\item The condition $c<c_2(q)=\frac{1-\frac{1}{r_0}}{2q-2}-1$ is equivalent to $q+c(q-1)<\frac32-\frac{1}{2r_0}$. This ensures that the power of $H$ in the first term of (ii) is negative.
\item Since $c>c_1(q)=\frac{q+1}{2-q}$, the power of $H$ in the second term of (ii) is negative.
\item Choosing $\hat H=H^c$, the power of $H$ in (i) is also negative, again for $c>c_1(q)=\frac{q+1}{2-q}$.
\end{itemize}
\end{remark}

\textbf{Standard pairs and families.} Let us recall some terminology related
to standard pairs, see also \cite[page 29]{ChDo09}. A \textit{standard pair}
$\ell=(W,h_W)$ is a regular unstable curve $W$ that supports a dynamically
log-H\"older continuous probability density $h_W$. As such, it can be
regarded as a probability measure on the phase space $\cM$, which will be
denoted by $\ell$, too.

A \textit{standard family} is a collection of standard pairs
$\cG=\{\ell_a\}$, $a\in\cA$ equipped with a probability factor measure
$\lambda_{\cG}$ on $\cA$. This induces a probability measure $\bP_{\cG}$ on
$\cM$.

For a standard pair $\ell=(W,h_W)$ any $x\in W$ splits $W$ into two
subcurves, let $r_W(x)$ denote the length of the shorter, and let
$\cZ_{\ell}=\sup_{\eps>0} \eps^{-1} \ell(r_W\le \eps)$. By H\"older
continuity of $\log h_W$, $\ell$ is equivalent to the normalized Lebesgue
measure on $W$ and thus $\cZ_{\ell} \asymp |W|^{-1}$.
This generalizes for the $Z$-function of a standard family $\cZ_{\cG} \asymp
\int \frac{\lambda_{\cG}(a)}{|W_a|} dm_W$.

The $T_{\scat}$-image of a standard pair is a countable collection of standard pairs.
Hence, the image of a standard family is a standard family. Given a standard
family $\cG$, for $n\ge 1$, $\cG_n$ denotes the $T_{\scat}^n$-image of $\cG$.
It follows from the growth lemma (Proposition~\ref{prop:growthscaled}) that there exists $\vartheta<1$ and $C_1,C_2>0$ such that
\[
\cZ_{\cG_n}\le C_1\vartheta^n \cZ_{\cG} + C_2 \delta_0^{-1}
\]
where $\delta_0 \asymp \scat^{\nu}$ (see \eqref{eq-size W} and
Remark~\ref{rem-usegr3}, part (i)). As  consequence, for any standard pair and
$n\ge 1$

\begin{equation}\label{eq:Zbound}
\cZ_{\cG_n}\le C \max(\cZ_{\cG_1}, \scat^{-\nu}).
\end{equation}

\textbf{Cells.} For $\xi\in\Z^2$ such that the corridor is opened up, and for
$m\in\Z$ let $D_{\xi,m}\subset \cM$ denote the set of points for which
$\ks=m\xi+\xi'$. The geometric properties of $D_{\xi,m}$ and its image
$T_{\scat}D_{\xi,m}$ will play an important role in the argument. $T_{\scat}D_{\xi,m}$ is
depicted in Figure~\ref{fig:phasespace}.
A similar description applies to $D_{\xi,m}$; it is delimited by a long singularity
curve, decreasing in the $(\theta,\varphi)$ coordinates, which is connected
to the boundary of $\cM$ by two shorter decreasing singularity curves,  of
length $\asymp (|\xi|\scat m)^{-1/2}$, running at a distance $\asymp (|\xi|
m)^{-2}$ from each other. Further properties:
\begin{itemize}
\item $\mu(D_{\xi,m})=\mu(T_{\scat}D_{\xi,m})\asymp \scat^{-1} |\xi|^{-3}m^{-3}$
(due to the factor $\cos\phi$ in the measure);
\item an unstable curve may intersect $D_{\xi,m}$ in a subcurve of length
$\le C (|\xi| m)^{-2}$;
\item  $T_{\scat}D_{\xi,m}$ intersects homogeneity strips of index $k\ge C(\scat
    |\xi| m)^{\frac{1}{2r_0}}$
\end{itemize}

If $\ell=(W,h_W)$ is a standard pair, then it can intersect $D_{\xi,m}$ in a
subcurve of length $\le C (|\xi| m)^{-2}$, thus the intersection has
probability bounded
above by $C (|\xi| m)^{-2}|W|^{-1} \asymp \cZ_{\ell} (|\xi| m)^{-2}$. It
follows that for a standard family $\cG$ we have
\begin{equation}\label{eq:Ztrans}
  \P_{\cG} (D_{\xi,m})\le C (|\xi| m)^{-2} \cZ_{\cG} .
\end{equation}

Our argument below follows the proof of \cite[Proposition 9.1]{ChDo09}
 taking into account that the corridor structure depends on $\scat$.\\

\begin{proofof}{Lemma~\ref{lemma:stpair}}
For item (i), using $\mu(T_{\scat}^{-j}D_{\hxi,\hat m}) \ll \scat^{-1}
\hat m^{-3} |\hxi|^{-3}$ as well as Lemma~\ref{lem:corridor_sum} several
times, we get

\begin{align*}
\int |\ks'|^q \cdot |\ks''''\circ T_{\scat}^j|^q \, d\mu
   & \le  C \sxi \shxi |\xi|^q |\hxi|^q  \sum_{m=1}^{\frac{H}{|\xi|}}
   \sum_{\hm=\frac{\hat H}{|\hxi|}}^{\infty} m^q \hm^q \mu(D_{\xi,m}\cap
   T_{\scat}^{-n}D_{\hxi,\hm}) \\
   &\le  C \scat^{-1}  \sxi \shxi |\xi|^q |\hxi|^q \sum_{m=1}^{\frac{H}{|\xi|}} m^q
   \sum_{\hm=\frac{\hat H}{|\hxi|}}^{\infty} \hm^{q-3} |\hxi|^{-3}
   \\
   &\le  C \scat^{-1} \sxi   H^{q+1} |\xi|^{-1} \shxi   \hat H^{q-2}
   |\hxi|^{-1}
   \le  CH^{q+1} \hat H^{q-2} \scat^{-3}.
\end{align*}
We will take $\hat H=H^c$ for $c>0$ to be determined.
To get a negative power of $H$, we need $q<2$ and $c>\frac{q+1}{2-q}$.

For the proof of (ii), we need to estimate
\begin{equation} \label{eq:decaykappa1}
\int |\ks'' |^q\cdot |\ks\circ T_{\scat}^j|^q \, d\mu \le C \sxi \shxi |\xi|^q |\hxi|^q
\sum_{m=\frac{H}{|\xi|}}^{\infty} m^q \sum_{\hm=1}^{\infty} \hm^q
\mu(D_{\xi,m}\cap T_{\scat}^{-j}D_{\hxi,\hm}).
\end{equation}

For different ranges of the indices, we will use two different estimates to
bound $\mu(D_{\xi,m}\cap T_{\scat}^{-j}D_{\hxi,\hm})$. On the one hand, as before, we
have
\begin{equation}\label{eq:decaykappa2}
\mu(D_{\xi,m}\cap T_{\scat}^{-n}D_{\hxi,\hm})\le \mu(D_{\hxi,\hm})\le C \scat^{-1}
|\hxi|^{-3} \hm^{-3}.
\end{equation}

For the other estimate, foliate $D_{\xi,m}$ with unstable curves $|W|$ of
length $\asymp (|\xi|m)^{-2}$. The image of any such curve stretches along
$T_{\scat}D_{\xi,m}$, crossing homogeneity strips with indices $k\ge C(\scat |\xi|
m)^{\frac{1}{2r_0}}$. The piece of $T_{\scat}W$ in the $k$-th homogeneity strip will
be denoted by $T_{\scat}W_k$, it has length $\asymp k^{-r_0-1}$, and its preimage has
length
 \[
 |W_k|\asymp k^{-r_0-1} \frac{\scat}{|\xi| m k^{r_0}}=\frac{\scat}{|\xi| m
 k^{2r_0+1}}
 \]
 as the expansion factor of $T_{\scat}$ on $W_k$ is $\asymp \scat^{-1} |\xi| m
 k^{r_0}$.
 Equipped with the conditional measure induced by $\mu$, $W$ is a standard
 pair $\ell=(W,h_W)$, and its image is a standard family $T_{\scat}\ell$ associated
 to the curves $T_{\scat}W_k$. To obtain the Z function, we use that the weight of
 $|T_{\scat}W_k|$ within this family is $\frac{|W_k|}{|W|}$, thus
 \begin{align*}
 \cZ_{T_{\scat}\ell} &\asymp  \sum_{k\ge C(\scat |\xi| m)^{\frac{1}{2r_0}}}
 \frac{|W_k|}{|W|} |T_{\scat}W_k|^{-1} \asymp \sum_{k\ge C(\scat |\xi|
 m)^{\frac{1}{2r_0}}} \frac{\scat |\xi|^2 m^2}{|\xi| m k^{2r_0+1}} k^{r_0+1} \\
 &\asymp  \scat m |\xi| \sum_{k\ge C(\scat |\xi| m)^{\frac{1}{2r_0}}} k^{-r_0}
 \asymp (\scat m |\xi|)^{\frac12+\frac{1}{2r_0}}.
 \end{align*}
This analysis applies to all the curves in the foliation. Accordingly, $\mu$
conditioned on $D_{\xi,m}$ can be regarded as a standard family $\cG$, and
the $\cZ$-function of its $T_{\scat}$-image satisfies
 \[
\cZ_{\cG_1}  \asymp C (\scat m |\xi|)^{\frac12+\frac{1}{2r_0}}.
 \]
For further iterates, it follows form \eqref{eq:Zbound} that
 \[
\cZ_{\cG_n}  \le C \scat^{-\nu} (m
|\xi|)^{\frac12+\frac{1}{2r_0}}.
 \]
Now we apply \eqref{eq:Ztrans} to get
\begin{align}
\mu(D_{\xi,m}\cap T_{\scat}^{-n}D_{\hxi,\hm}) &= \mu(D_{\xi,m}) \P_{\cG_n}
(D_{\hxi,\hm}) \le C \mu(D_{\xi,m}) \cZ_{\cG_n} |\hxi|^{-2} \hm^{-2} \nonumber
\\
&\le C |\hxi|^{-2}\hm^{-2} |\xi|^{-\frac52+\frac{1}{2r_0}}
m^{-\frac52+\frac{1}{2r_0}} \scat^{-1-\nu}.
\label{eq:decaykappa3}
\end{align}
We split \eqref{eq:decaykappa1} into two parts. If $\hm\le m^c$ (for some $c
> 0$ to be determined), we use \eqref{eq:decaykappa3} and get

\begin{align*}
\sxi \shxi & |\xi|^q |\hxi|^q \sum_{m=\frac{H}{|\xi|}}^{\infty} m^q
\sum_{\hm=1}^{m^c} \hm^q \mu(D_{\xi,m}\cap T_{\scat}^{-n}D_{\hxi,\hm})
\\
&\le C \scat^{-1-\nu} \sxi \shxi |\xi|^{-\frac52+q+\frac{1}{2r_0}} |\hxi|^{q-2}
\sum_{m=\frac{H}{|\xi|}}^{\infty} m^{-\frac52+q+\frac{1}{2r_0}} m^{c(q-1)}
\\
&\le C  \scat^{-1-\nu} H^{-\frac32+q+c(q-1)+\frac{1}{2r_0}}
\left(\sxi |\xi|^{-1-c(q-1)}
\right) \left(\shxi |\hxi|^{q-2}\right) \\
&\le C H^{-\frac32+q+c(q-1)+\frac{1}{2r_0}} \scat^{c(q-1)-q-2-\nu},
\end{align*}
where we have used that because $q+c(q-1)<\frac32-\frac{1}{2r_0}$, the contribution of $m$ is summable
(this condition is equivalent to $c<c_2(q)=\frac{1-\frac{1}{r_0}}{2q-2}$, cf.~Remark~\ref{rmk:stpair}).
Note that if $q=1$ then this contribution is independent of $c$; however, there is an additional factor
of $|\log \scat|\cdot \log H$.

For $m > m^c$ we use \eqref{eq:decaykappa2} and get

\begin{align*}
\sxi \shxi & |\xi|^q |\hxi|^q \sum_{m=\frac{H}{|\xi|}}^{\infty} m^q
\sum_{\hm=m^c}^{\infty} \hm^q \mu(D_{\xi,m}\cap T_{\scat}^{-n}D_{\hxi,\hm}) \\
&\le C \scat^{-1} \sxi \shxi |\xi|^q |\hxi|^{q-3}
\sum_{m=\frac{H}{|\xi|}}^{\infty} m^q \sum_{\hm=m^c}^{\infty} \hm^{q-3}
\le C \scat^{-1} \sxi \shxi |\xi|^q |\hxi|^{q-3}
\sum_{m=\frac{H}{|\xi|}}^{\infty} m^{c(q-2)+q} \\
&\le C H^{c(q-2)+q+1} \scat^{-1} \left(\sxi |\xi|^{c(2-q)-1} \right) \left(\shxi
|\hxi|^{q-3}\right) \le C H^{c(q-2)+q+1} \scat^{-1-q-c(2-q)},
\end{align*}
and in case $q=1$ we still have an additional $|\log \scat|$ factor. The condition of summability
$c(q-2)+q<-1$ is satisfied because $c>\frac{q+1}{2-q}$.
Summarizing, we need
\[
1\le q <2, \qquad q+c(q-1)<\frac32-\frac{1}{2r_0}, \qquad \frac{q+1}{2-q}<c.
\]
First we may fix $q$ such that
\[
\frac32-\frac{1}{2r_0}>q+\frac{q+1}{2-q}(q-1)=\frac{2q-1}{2-q} \quad
\Leftrightarrow\quad  q<2-\frac{6}{7-1/r_0}
\]
and then we can fix $c$ slightly larger than $\frac{q+1}{2-q}$, such that the
conditions are still met. The range of allowed $q$ depends on $r_0$, it can
never exceed $\frac87$; for the traditional $r_0=2$ the upper bound is
$\frac{14}{13}$, while for $r_0=5$ the upper bound is
$\frac{19}{17}$.
\end{proofof}

\subsection{Exploiting the existence of a Young tower for $T_\scat$}

Let  $(\bar\Delta_\scat, T_{\bar\Delta_\scat}, \mu_{\bar\Delta_\scat})$ be the
corresponding one-sided Young tower (i.e., with stable leaves quotiented out) and let $R_{\bar\Delta_\scat}$ be the
transfer operator of  $T_{\bar\Delta_\scat}$.
Let $\hks$ be the version of $\ks$ on
$\bar\Delta_\scat$. We will also use the notations $\hks', \hks'', \hks''', \hks''''$ for the Young tower versions of the truncations
$\ks', \ks'', \ks''', \ks''''$, respectively, hence for example $\hks'=\hks \cdot \textbf{1}_{|\hks|\le H}$. 
Since $\ks$ is constant on stable leaves, we have for any $j, \ell\ge 0$,
\begin{align}
\label{eq:corel-alt}
\nonumber&\int_{\cM_0}  (e^{it\ks}-1)\,
R_\scat^{\ell}(e^{it\ks}-1)\, (e^{it\ks}-1)\circ
T_{\scat}^{j}\, d\mu\\
&=\int_{\bar\Delta_\scat}  (e^{it\hks}-1)\,
R_{\bar\Delta_\scat}^{\ell}(e^{it\hks}-1)\,
(e^{it\hks}-1)\circ T_{\bar\Delta_\scat} ^{j}\,
d\mu_{\bar\Delta_\scat}.
\end{align}

Let $r$ be the roof function of the
tower $(\bar\Delta_{\scat},\mu_{\bar\Delta_{\scat}})$.
We recall that if $d := \gcd(r) > 1$, then for every $\scat>0$,  $R_{\bar\Delta_\scat}$, when viewed as an operator acting on the Young Banach
space $\cB_{\bar\Delta_\scat}\subset L^p(\mu_{\bar\Delta_\scat})$, $p>1$,
has a spectral gap (see~\cite{Chernov99, Young98}). As clarified in Remarks~\ref{rem:gcd} and~\ref{rem:gcd2}, the decomposition of $R_{\bar\Delta_\scat}$ we shall need in the proof below
holds when $d>1$.

Before proceeding to the proof of Proposition~\ref{prop:decaykappa}, we recall one property of the norm the Young space $\cB_{\bar\Delta_\scat}\subset L^p(\mu_{\bar\Delta_\scat})$, $p>1$ that we shall need in the proof below.
(The precise definition of $\cB_{\bar\Delta_\scat}$ is not important in the proof below, and we omit it.) For any function that is constant on the partition elements of the Young tower, the involved seminorm is zero. This is the case for $\hks'$ and thus, $\|\hks'\|_{\cB_{\bar\Delta_\scat}}\le H$
(and a similar version holds for $\hks'''$). \\

\begin{proofof}{Proposition~\ref{prop:decaykappa}} We first prove the
statement for the case when $\ell=0$ and point out the required modifications when $\ell\ge 1$.

{\bf{Case $\ell=0$.}}  Given~\eqref{eq:corel-alt}, in this case we need to show
that
\begin{align}\label{eq:mo}
\left|\int_{\bar\Delta_\scat}  (e^{it\hks}-1)\,
(e^{it\hks}-1)\circ T_{\bar\Delta_\scat} ^{j}\,
\mu_{\bar\Delta_\scat}-\Big(\int_{\bar\Delta_\scat}
(e^{it\hks}-1)\, \mu_{\bar\Delta_\scat}\Big)^2\right|
\le \hat C_\scat |t|^2 \hat\vartheta_\scat^{j},
\end{align}
for some $\scat$-dependent constants $\hat\vartheta_\scat<1$ and $\hat C_\scat > 0$.

Throughout this proof, we let $\ks', \ks'', \ks''', \ks''''$ also denote
their corresponding versions on the tower $\Delta_\scat$ and the context
in which they appear will make it clear which version we are referring to.

Write
\begin{align*}
&\int_{\bar\Delta_\scat}  (e^{it\hks}-1)\,
(e^{it\hks}-1)\circ T_{\bar\Delta_\scat} ^{j}\, d
\mu_{\bar\Delta_\scat}=
\int_{\bar\Delta_\scat}(e^{i\hks t}-e^{i\ks't}) \cdot (e^{i\hks
t}-1) \circ T_{\bar\Delta_\scat}^j \, d \mu_{\bar\Delta_\scat}\\
& + \int_{\bar\Delta_\scat} (e^{i\ks' t}-1) \cdot (e^{i\hks
t}-e^{i\ks'''t}) \circ T_{\bar\Delta_\scat}^j \, d \mu_{\bar\Delta_\scat}
+ \int_{\bar\Delta_\scat}  (e^{i\ks' t}-1) \cdot (e^{i\ks''' t}-1) \circ
T_{\bar\Delta_\scat}^j \, d \mu_{\bar\Delta_\scat}\\
 &=
\int_{\bar\Delta_\scat} e^{i\ks't} \cdot (e^{i\ks'' t}-1) \cdot
(e^{i\hks t}-1) \circ T_{\bar\Delta_\scat}^j \, d \mu_{\bar\Delta_\scat}+
\int_{\bar\Delta_\scat}  (e^{i\ks' t}-1) \cdot e^{i\ks'''t} \circ
T_{\bar\Delta_\scat}^j \cdot (e^{i\ks'''' t}-1) \circ T_{\bar\Delta_\scat}^j
\, d \mu_{\bar\Delta_\scat}\\
& + \int_{\bar\Delta_\scat}  (e^{i\ks' t}-1) \cdot (e^{i\ks''' t}-1)
\circ T_{\bar\Delta_\scat}^j\, d \mu_{\bar\Delta_\scat}
=I_1(t,\scat)+I_2(t,\scat)+I_3(t,\scat).
\end{align*}

For $I_3(t,\scat)$ we use the exponential decay of correlation (see Remark~\ref{rem:gcd} below for the case that the roof function $r$ of the tower has $\gcd(r) > 1$). This gives
the only source of unknown dependence on $\scat$ in the case $m=0$.
More precisely, for every $\scat>0$, there exists $\hat\theta_\scat<1$ and $C_\scat>0$ so that
\begin{eqnarray}\label{eq:decay}
\left| I_3(t,\scat)
 - \int_{\bar\Delta_\scat} (e^{i\ks't}-1) \, d \mu_{\bar\Delta_\scat}
 \int_{\bar\Delta_\scat} (e^{i\ks'''t}-1) \, d \mu_{\bar\Delta_\scat}
 \right|
& \leq & C_{\scat} \, \hat\theta^j_{\scat}\, \| e^{it\ks' t}-1
\|_{\cB_{\Delta_\scat}} \, \| e^{it\ks''' t}-1 \|_{\cB_{\Delta_\scat}}  \nonumber \\
& \leq&  C_{\scat}\, \hat\theta^j_{\scat} H \, \hat H \, |t|^2.
 \end{eqnarray}
Thus,
\begin{align*}
&\left| I_3(t,\scat)-\Big(\int_{\bar\Delta_\scat}
(e^{it\hks}-1)\, \mu_{\bar\Delta_\scat}\Big)^2\right|
\le C_{\scat}\, \hat\theta^j_{\scat} H \, \hat H \, |t|^2\\
&+
\Big| \int_{\bar\Delta_\scat} (e^{i\ks't}-1) \, d \mu_{\bar\Delta_\scat}
\int_{\bar\Delta_\scat} (e^{i\ks'''t}-1) \, d \mu_{\bar\Delta_\scat}-
\int_{\bar\Delta_\scat} (e^{i\hks t}-1) \, d \mu_{\bar\Delta_\scat}
\int_{\bar\Delta_\scat} (e^{i\hks t}-1) \, d
\mu_{\bar\Delta_\scat}\Big|\\
& =C_{\scat}\, \hat\theta^j_{\scat} H \, \hat H \, |t|^2+ |J(t,\scat)|.
 \end{align*}
By definition,
\[
|J(t,\scat)|=\Big| \int_{\cM_0}(e^{i\ks't}-1) \, d \mu  \int_{\cM_0} (e^{i\ks'''t}-1)
\, d \mu- \int_{\cM_0} (e^{i\hks t}-1) \, d \mu  \int_{\cM_0} (e^{i\hks t}-1) \,d \mu\Big|
\]
and we note that $J(t,\scat)$ is bounded by the sum of
\[
\int_{\cM_0} |e^{i\ks' t} \cdot (e^{i t\ks''}-1)| \, d \mu\,\int_{\cM_0} |e^{i
t\ks''' t}-1| \, d \mu
\leq |t|^2 \int |\ks| 1_{\{ \ks > H\}} \, d \mu  \int_{\cM_0}|\ks'''|  \, d
\mu
\]
and a similar term with $\hat H$ instead of $H$.
Using the H\"older inequality (with exponents $\frac{2}{1+\delta}$ and
$\frac{2}{1-\delta}$),
the tail behaviour of $\ks$ and Lemma~\ref{lem:kappa_-norm}, we obtain that
\[
\int_{\cM_0} |\ks| 1_{\{|\ks| > H\}} \, d\mu \leq \|
\ks\|_{L^{2/(1+\delta)}} \,
\mu(|\ks|>H)^{(1-\delta)/2} \ll \scat^{-1} H^{-(1-\delta)}.
\]
Also $\int_{\cM_0} |\ks'''|  d\mu\leq \| \ks \|_{L^1(\mu)} \ll \scat^{-1}$.
Hence,
\begin{equation}\label{eq:thirdterm}
  |J(t,\scat)|  \ll |t|^2 \scat^{-2} \left( H^{-(1-\delta)} +
  \hat H^{-(1-\delta)} \right).
\end{equation}
Finally, note that
\begin{align}\label{eq-deci1i2}
\nonumber|I_1(t,\scat)+I_2(t,\scat)|&\leq
 |t|^2 \int_{\bar\Delta_\scat}  |\ks''| \cdot |\ks|\circ
 T_{\bar\Delta_\scat}^j \, d \mu_{\bar\Delta_\scat}
 + |t|^2 \int_{\bar\Delta_\scat}  |\ks'| \cdot |\ks''''|\circ
 T_{\bar\Delta_\scat}^j \, d \mu_{\bar\Delta_\scat}\\
&=|t|^2\left(\int_{\cM_0} |\ks''| \cdot |\ks|\circ
 T_{\bar\Delta_\scat}^j \, d \mu +\int_{\cM_0} |\ks'| \cdot |\ks''''|\circ
 T_{\bar\Delta_\scat}^j \, d \mu\right).
\end{align}

For this $\ell=0$ case,
if we fix any $r_0\ge 2$ (taking into account that
$\hat{H}=H^{c_0}$), then we may bound the coefficients of $|t|^2$ in
$|J(t,\scat)|$ from \eqref{eq:thirdterm}, $|I_1(t,\scat)|$ and $|I_2(t,\scat)|$ from \eqref{eq-deci1i2},
respectively by
\[
\scat^{-2} H^{-(1-\delta)}; \qquad H^{-\frac15} \scat^{-4} + H^{2-c_0}\scat^{-\frac{11}{5}-c_0}, \qquad H^{2-c_0}\scat^{-3},
\]
where in the bound for $|I_1(t,\scat)|$ the exponents of $H$ and $\scat$ have been slightly decreased to bound the logarithmic factors.
Fixing $c_0=\frac{11}{5}$ and $\delta=\frac45$, all these are dominated by $H^{-\frac15}\scat^{-\frac{22}{5}}$.
On the other hand the coefficient of $|t|^2$ in  $|I_3(t,\scat)|$ is
$C_{\scat}\, \hat\theta^j_{\scat} H^{c_0+1}=C_{\scat}\, \hat\theta^j_{\scat} H^{\frac{16}{5}}$. Thus letting
$H=\left(C_{\scat}^{-1}\, \hat\theta^{-j}_{\scat} \scat^{-\frac{22}{5}}\right)^{\frac{5}{17}}$ we conclude that all terms
are dominated by
\begin{equation}
\label{eq:edc_const_relate}
\scat^{-\frac{352}{85}} C_{\scat}^{\frac{1}{17}}\, (\hat\theta_{\scat}^{\frac{1}{17}})^j;\qquad \text{thus we let}\qquad \hat{C}_{\scat}=\scat^{-\frac{352}{85}} C_{\scat}^{\frac{1}{17}}, \ \hat{\vartheta}_{\scat}=\hat\theta_{\scat}^{\frac{1}{17}}.
\end{equation}

{\bf{Case $\ell\ge 1$.}} The main differences in this case come down to dealing
with integrals containing unbounded terms $\ks''$ and $\ks''''$
in such a way that can gain exponential decay in $\ell$ and then proceed as in
the case $\ell=0$ treated above. To do this, we exploit that
$\cB_{\bar\Delta_\scat}\subset L^p(\mu_{\bar\Delta_\scat})$.

Using~\eqref{eq:corel-alt}, we need to estimate
\begin{align*}
J(t,\scat) := &\int_{\bar\Delta_\scat}  (e^{it\hks}-1)\,
R_{\bar\Delta_\scat}^{\ell}(e^{it\ks}-1)\,
(e^{it\hks}-1)\circ T_{\bar\Delta_\scat} ^{j}\,
d\mu_{\bar\Delta_\scat}\\
&-\int_{\bar\Delta_\scat} (e^{it\hks}-1)
R_{\bar\Delta_\scat}^{\ell}(e^{it\hks}-1)\, d\mu_{\bar\Delta_\scat}\,
\int_{\bar\Delta_\scat}
(e^{it\hks}-1)\,d\mu_{\bar\Delta_\scat}\\
&-\int_{\bar\Delta_\scat}
(e^{it\hks}-1)\,d\mu_{\bar\Delta_\scat}\, \int_{\bar\Delta_\scat}
(e^{it\hks}-1)\, (e^{it\hks}-1)\circ
T_{\scat}^{j}\, d\mu_{\bar\Delta_\scat} +\Big(\int_{\bar\Delta_\scat}
(e^{it\hks}-1)\,d\mu_{\bar\Delta_\scat}\Big)^3.
\end{align*}
By Remark~\ref{rem:gcd}, for every $\scat>0$ and
for every $\ell \ge 1$,
\begin{align}\label{eq:Qtower}
R_{\bar\Delta_\scat}^{\ell}(e^{it\hks}-1)-\int_{\bar\Delta_\scat}
(e^{it\ks}-1)\,d\mu_{\bar\Delta_\scat}=Q_{\bar\Delta_\scat}^{\ell}(e^{it\ks}-1),\quad\quad\|Q_{\bar\Delta_\scat}^{\ell}(e^{it\hks}-1)\|_{\cB_{\bar\Delta_\scat}}\le
C_{\scat}\, \hat\theta^\ell_{\scat},
\end{align}
for some $\scat$-dependent $C_\scat$ and $\hat\theta_{\scat}<1$. This is the first
source of unknown dependence on $\scat$. Since $\cB_{\bar\Delta_\scat}\subset L^p(\mu_{\bar\Delta_\scat})$,
\begin{align}\label{eq:Qtower2}
\|Q_{\bar\Delta_\scat}^{\ell}(e^{it\hks}-1)\|_{L^p(\mu_{\bar\Delta_\scat})}\le
C_{\scat}^0\, \hat\theta^\ell_{\scat},
\end{align}
for some $\scat$-dependent $C_\scat^0$. This is the second source of unknown dependence on $\scat$. \\

With these specified, we can write
\begin{align*}
J(t,\scat) &=\int_{\bar\Delta_\scat}  (e^{it\hks}-1)\,
Q_{\bar\Delta_\scat}^{\ell}(e^{it\hks}-1)\,
(e^{it\hks}-1)\circ T_{\bar\Delta_\scat} ^{j}\,
d\mu_{\bar\Delta_\scat}\\
&\quad -\int_{\bar\Delta_\scat} (e^{it\hks}-1)
Q_{\bar\Delta_\scat}^{\ell}(e^{it\hks}-1)\, d\mu_{\bar\Delta_\scat}\,
\int_{\bar\Delta_\scat}
(e^{it\hks}-1)\,d\mu_{\bar\Delta_\scat} =E(t,\scat) -G(t,\scat).
\end{align*}
Rearranging as in the case $\ell=0$,
\begin{align*}
E(t,\scat) &=
\int_{\bar\Delta_\scat}(e^{i\hks t}-e^{i\hks't})
\,Q_{\bar\Delta_\scat}^{\ell}(e^{it\ks}-1)\, \, (e^{i\hks t}-1)
\circ T_{\bar\Delta_\scat}^j \, d \mu_{\bar\Delta_\scat}\\
& + \int_{\bar\Delta_\scat} (e^{i\ks' t}-1)
\,Q_{\bar\Delta_\scat}^{\ell}(e^{it\hks}-1)\, (e^{i\hks
t}-e^{i\ks'''t}) \circ T_{\bar\Delta_\scat}^j \, d \mu_{\bar\Delta_\scat}\\
&+ \int_{\bar\Delta_\scat}  (e^{i\ks' t}-1)\,
Q_{\bar\Delta_\scat}^{\ell}\,(e^{it\hks}-1)\,  (e^{i\ks''' t}-1)
\circ T_{\bar\Delta_\scat}^j \, d \mu_{\bar\Delta_\scat}\\
 &=
\int_{\bar\Delta_\scat} e^{i\ks't}
\,Q_{\bar\Delta_\scat}^{\ell}(e^{it\hks}-1)\,  (e^{i\ks'' t}-1)
\cdot (e^{i\hks t}-1) \circ T_{\bar\Delta_\scat}^j \, d
\mu_{\bar\Delta_\scat}\\
&+ \int_{\bar\Delta_\scat}  (e^{i\ks' t}-1)
\,Q_{\bar\Delta_\scat}^{\ell}(e^{it\hks}-1)\,  e^{i\ks'''t} \circ
T_{\bar\Delta_\scat}^j \cdot (e^{i\ks'''' t}-1) \circ T_{\bar\Delta_\scat}^j
\, d \mu_{\bar\Delta_\scat}\\
& +\int_{\bar\Delta_\scat}  (e^{i\ks' t}-1)
\,Q_{\bar\Delta_\scat}^{\ell}(e^{it\hks}-1)\,  (e^{i\ks''' t}-1)
\circ T_{\bar\Delta_\scat}^j\, d
\mu_{\bar\Delta_\scat}=E_1(t,\scat)+E_2(t,\scat)+E_3(t,\scat).
\end{align*}
Let $q\in (1,\frac87-\frac{6}{7r_0-1})$ so that Lemma~\ref{lemma:stpair} holds.
By the H\"older inequality with $\frac1p+\frac1q=1$ and~\eqref{eq:Qtower2},
\begin{align*}
|E_1(t,\scat)+E_2(t,\scat)|&\le
\|Q_{\bar\Delta_\scat}^{\ell}(e^{it\hks}-1)\|_{L^p(\mu_{\bar\Delta_\scat})}\,
|t|^2\| \, |\ks''| \cdot |\hks|\circ T_{\bar\Delta_\scat}^j
\|_{L^q(\mu_{\bar\Delta_\scat})}\\
&\quad +\|Q_{\bar\Delta_\scat}^{\ell}(e^{it\hks}-1)\|_{L^p(\mu_{\bar\Delta_\scat})}\,
|t|^2\| \, |\ks'| \cdot |\ks''''|\circ T_{\bar\Delta_\scat}^j
\|_{L^q(\mu_{\bar\Delta_\scat})}\\
&\le C_{\scat}^0\, \hat\theta^{\ell}_{\scat}\,|t|^2\, \left(\||\ks''| \cdot
|\hks|\circ T_{\bar\Delta_\scat}^j \|_{L^q(\mu_{\bar\Delta_\scat})}
+\||\ks'| \cdot |\ks''''|\circ T_{\bar\Delta_\scat}^j
\|_{L^q(\mu_{\bar\Delta_\scat})}\right).
\end{align*}
Similar to estimating~\eqref{eq-deci1i2}, using Lemma~\ref{lemma:stpair} and Remark~\ref{rmk:stpair}
and without trying for optimal bounds,
we can pick $q$ close to $1$ and $c_0 < \frac52$ such that $c_0(q-2)+q+1 = -\frac15$.
For these values,
\begin{align}\label{eq:decj1j2}
|E_1(t,\scat)+E_2(t,\scat)|\le C\,C_{\scat}^0\,
\hat\theta^{\ell}_{\scat}\,|t|^2\, H^{-\frac{1}{5q}}
\scat^{\frac{-5}{q}}.
\end{align}
Next, let
\begin{align*}
L_1(t,\scat)&=\int_{\bar\Delta_\scat}  (e^{i\ks' t}-1)
\,Q_{\bar\Delta_\scat}^{\ell}(e^{it\hks}-1)\,  (e^{i\ks''' t}-1)
\circ T_{\bar\Delta_\scat}^j\, d \mu_{\bar\Delta_\scat}\\
&\quad -\int_{\bar\Delta_\scat} (e^{it\ks'}-1)
Q_{\bar\Delta_\scat}^{\ell}(e^{it\hks}-1)\, d\mu_{\bar\Delta_\scat}\,
\int_{\bar\Delta_\scat}  (e^{it\ks'''}-1)\,d\mu_{\bar\Delta_\scat}
\end{align*}
and note that
\begin{align*}
E_3(t,\scat)-G(t,\scat)&=L_1(t,\scat)-\int_{\bar\Delta_\scat}
(e^{it\hks}-e^{it\ks'})
Q_{\bar\Delta_\scat}^{\ell}(e^{it\hks}-1)\, d\mu_{\bar\Delta_\scat}\,
\int_{\bar\Delta_\scat}
(e^{it\hks}-1)\,d\mu_{\bar\Delta_\scat}\\
&\quad -\int_{\bar\Delta_\scat} (e^{it\hks}-1)
Q_{\bar\Delta_\scat}^{\ell}(e^{it\hks}-1)\, d\mu_{\bar\Delta_\scat}\,
\int_{\bar\Delta_\scat}
(e^{it\hks}-e^{it\ks'''})\,d\mu_{\bar\Delta_\scat}\\
&=L_1(t,\scat)-L_2(t,\scat)-L_3(t,\scat).
\end{align*}
By the exponential decay of correlations as in~\eqref{eq:decay} as well as \eqref{eq:Qtower2}:
\begin{align*}
|L_1(t,\scat)| &\le C_{\scat}\, \hat\theta^j_{\scat} H \, \hat H \,
|t|^2\|Q_{\bar\Delta_\scat}^{\ell}(e^{it\hks}-1)\|_{L^p(\mu_{\bar\Delta_\scat})}
\le C_{\scat}\,C_{\scat}^0\, \hat\theta^{\ell}_{\scat} \, |t|^2 \, H^{1+c_0},
\end{align*}
where as before $c_0 < \frac52$.
Finally, by the equation before~\eqref{eq:thirdterm}, we have
\[
|L_2(t,\scat)|\le |t^2|\,\scat^{-1}
H^{-(1-\delta)}\|Q_{\bar\Delta_\scat}^{\ell}(e^{it\hks}-1)\|_{L^p(\mu_{\bar\Delta_\scat})}
\le  C_{\scat}\, \,C_{\scat}^0\,\hat\theta^{\ell}_{\scat}|t^2|\,\scat^{-1} H^{-(1-\delta)}.
\]
A similar argument applies to $L_3(t,\scat)$.

The conclusion follows with a similar choice of $H$ as in the case $\ell=0$ treated above.~\end{proofof}

\begin{remark}\label{rem:gcd}
Let $r$ be the roof function of the
one-sided tower map $(\bar\Delta_{\scat},\mu_{\bar\Delta_{\scat}})$.
If $d := \gcd(r) > 1$, then $T_{\bar\Delta}$ is not mixing on the Banach space $\cB_{\bar\Delta_\scat}$.
However, the underlying billiard map $T_{\scat}$ is mixing and thus,
\begin{equation}\label{eq:1}
 \int_{\cM} R_{\scat}^n \phi \cdot \overline{\psi} \, d\mu \to 0 \text{ as } n \to \infty,
\end{equation}
for $\phi, \psi \in \cB$ with $\int_{\cM} \phi \, d\mu = 0$.
If $\gcd(r) = d > 1$, then the eigenvalues on the unit circle
are the $d$-th roots of unity.
Hence,
$$
R_{\bar\Delta_\scat} = \Pi_{\bar \Delta_\scat} + Q_{\bar\Delta_\scat} :=
\sum_{\lambda^d=1} \lambda \Pi_\lambda
+ Q_{\bar \Delta},
\qquad\qquad \Pi_{\bar\Delta_\scat} Q_{\bar\Delta_\scat} =
 Q_{\bar\Delta_\scat} \Pi_{\bar\Delta_\scat}= 0,
$$
where $\Pi_\lambda$ denotes the projection on the (generalised) eigenspace $ \cB_{\bar\Delta_\scat,\lambda}$ of eigenvalue $\lambda$,
and $Q_{\bar\Delta_\scat}$ is the complementary projection.
The Banach space $\cB_{\bar\Delta_{\scat}}$ on $\bar\Delta_{\scat}$ can be written as the direct sum
\begin{equation}
\cB_{\bar\Delta_{\scat}} = \cB_{\Pi}
\oplus \cB_{Q}
\qquad \text{ for } \quad
\cB_{\Pi} := \oplus_{\lambda^d = 1} \cB_{\bar\Delta_\scat,\lambda}
= \ker(Q_{\bar \Delta_\scat})
\quad \text{ and }  \quad \cB_Q = \ker(\Pi_{\bar\Delta_\scat}),
\end{equation}
As the kernels of projections, $\cB_{\Pi}$ and $\cB_Q$
are closed $R_{\bar\Delta_{\scat}}$-invariant
subspaces of $\cB_{\bar\Delta_\scat}$, and hence Banach spaces themselves.
Also, as clarified below, for every $\scat>0$, the restriction $R_{\bar\Delta_{\scat}}$ to $\cB_{Q}$
has spectral radius less than $1$. That is, for every $\scat>0$, there exists $\hat\theta_\scat < 1$
so that
\begin{equation}
\label{eq:spgapgcd}\| R^n_{\scat} \phi \|_{\cB_{\bar\Delta_\scat}} \ll \hat\theta_\scat^n\,\| \phi \|_{\cB_{\bar\Delta_\scat}}.
\end{equation}

Consider the lifted version of $\phi$:
$\phi_{\bar \Delta_\scat}(x) = \int_{\ell(x)} \phi \circ \pi \, d\mu_{\Delta_\scat,\ell(x)}$
where $\ell(x)$ is the stable leaf through $x \in \bar\Delta$ and $\mu_{\Delta_\scat,\ell(x)}$
the measure on this leaf emerging from the disintegration of the measure $\mu_{\Delta_\scat}$ of the two-sided tower.
The transfer operator $R_{\bar\Delta_\scat}$ on the one-sided tower satisfies
\begin{equation}\label{eq:2}
 \int_{\cM} R^n \phi \cdot \overline{\psi} \, d\mu =
 \int_{\bar\Delta} R^n_{\bar\Delta_\scat} \phi_{\bar\Delta_\scat} \cdot \overline{\psi_{\bar\Delta_\scat} }\, d\mu_{\bar\Delta_\scat}.
\end{equation}
If $\Pi_{\bar\Delta_\scat} \phi_{\bar\Delta_\scat} \neq 0$, then there exists
$\psi_{\bar\Delta_\scat} \in \cB_{\bar\Delta_\scat}$ such that
$\int_{\cM} R^n_{\bar\Delta_\scat} \phi_{\bar\Delta_\scat} \cdot \psi_{\Delta_\scat} \, d\mu_{\Delta_\scat} \not\to 0$.
(In fact, taking $\psi_{\bar \Delta_\scat} = \Pi_{\bar\Delta_\scat} \phi_{\bar\Delta_\scat}$, we get
$\int_{\cM} R^{dn}_{\bar\Delta_\scat} \phi_{\bar\Delta_\scat} \cdot \psi_{\bar\Delta_\scat} \, d\mu_{\bar\Delta_\scat} \to
\int_{\cM} \Pi_{\bar\Delta_\scat} \phi_{\bar\Delta_\scat}
\overline{\Pi_{\bar\Delta_\scat} \phi_{\bar\Delta_\scat}} \, d\mu_{\bar\Delta_\scat}  \neq 0$.)
This contradicts \eqref{eq:1} and/or \eqref{eq:2}.
Hence $\phi_{\bar\Delta_\scat} \in \cB_{Q}$ and
$\| R^n_{\bar\Delta_\scat} \phi_{\bar\Delta_\scat} \|_{\cB_Q} \leq
\|\, R_{\bar\Delta_{\scat}}|_{\cB_Q}\, \|^n \, \| \phi_{\bar\Delta_\scat}\|_{\cB_Q} \ll \hat\theta_\scat^n\, \| \phi_{\bar\Delta_\scat}\|_{\cB_Q}$.
Property~\eqref{eq:spgapgcd} follows.
\end{remark}

\begin{remark}\label{rem:gcd2}
We note that mixing of the underlying map $T_\scat$
is not required for an useful version of~\eqref{eq:Qtower} to hold.
Indeed the property of $Q_{\bar\Delta_\scat}$ in~\eqref{eq:Qtower} holds independent of mixing
and for this we just need to work with~\eqref{eq:spgapgcd}, which holds for $d>1$.
The downside of using~\eqref{eq:spgapgcd} directly
is that in assumption~\eqref{eq-corel2} we would have to extract
$\sum_{\lambda^d=1} \lambda \Pi_\lambda(e^{it\ks}-1)$
instead of $\int_{\cM} (e^{it\ks}-1)\, d\mu$.
We found it more convenient to work
with the 'clean' assumption~\eqref{eq-corel2}.
\end{remark}

\end{document}